\documentclass[preprint]{imsart}
\RequirePackage[OT1]{fontenc}
\RequirePackage{amsthm,amsmath}
\RequirePackage[numbers]{natbib}
\RequirePackage[colorlinks,citecolor=blue,urlcolor=blue]{hyperref}
\arxiv{arXiv:math.PR/0000000}
\usepackage{latexsym}
\usepackage{amsmath}
\usepackage{amsthm}
\usepackage{amsfonts}
\usepackage{tikz}
\usetikzlibrary{trees}
\usetikzlibrary{decorations.pathmorphing}
\usetikzlibrary{decorations.markings}

\startlocaldefs
\theoremstyle{plain}\theoremstyle{plain}
\newtheorem{theorem}{Theorem}
\newtheorem{definition}{Definition}

\newtheorem{corollary}{Corollary}
\newtheorem{lemma}{Lemma}
\newtheorem{remark}{Remark}

\newcommand{\field}[1]{\mathbb{#1}}
\newcommand{\C}{\field{C}}
\newcommand{\R}{\field{R}}
\newcommand{\N}{\field{N}}
\newcommand{\Z}{\field{Z}}
\numberwithin{equation}{section}
\numberwithin{lemma}{section}
\numberwithin{theorem}{section}
\numberwithin{corollary}{section}
\numberwithin{remark}{section}
\numberwithin{definition}{section}
\newcommand{\at}{\makeatletter@\makeatother}
\endlocaldefs

\begin{document}
\begin{frontmatter}
\title{Random walk to $\phi^4$ and back}
\runtitle{Random walk $\leftrightarrow \phi^4$ }

\begin{aug}
\author{\fnms{Daniel} \snm{H\"of}
\ead[label=e1]{dhoef@gmx.net}}




\address{
2. K\i{}s\i{}m Mah. Ay\c ci\c cek Sok. 17 D: 10\\
Bah\c ce\c sehir\\
TR-34488 \.Istanbul \\
Republic of Turkey\\
\printead{e1}\\
}

\end{aug}

\begin{abstract}
In this paper we establish an exact relationship between the asymptotic probability distributions $\nu_0$ and $\nu_2$ of the multiple point range of the planar random walk and the proper functions $\Gamma^{[0]}$ and $\Gamma^{[2]}$ respectively of the planar, complex $\phi^4$-theory, setting the number of components $m=0$: The characteristic functions $\Phi_0$ and $\Phi_2$ of $\nu_0$ and $\nu_2$ have simple integral transforms $\zeta^{[0]}$ and $\zeta^{[2]}$ respectively which turn out to be the extensions of the proper functions $\Gamma^{[0]}$ and $\Gamma^{[2]}$ onto a Riemann surface (with infinitely many sheets) in the coupling constant $g$ and are well defined mathematically. $\zeta^{[0]}$ and $\zeta^{[2]}$ restricted to a specific sheet have a (sectorwise) uniform asymptotic expansion in $g=0$.  The standard perturbation series of $\Gamma^{[0]}$ and $\Gamma^{[2]}$ in $g$ have expansion coefficients $\Gamma^{[0],pt}_r$ and $\Gamma^{[2],pt}_r$ which are polynomials in $m$. Order by order the lowest nontrivial polynomial coefficient in $m$: $\Gamma^{[0],pt}_{r,1} = \zeta^{[0]}_{r}$ and $\Gamma^{[2],pt}_{r,0} = \zeta^{[2]}_{r}$ where $\zeta^{[0]}_{r}$ and $\zeta^{[2]}_{r}$ are the coefficients of the asymptotic series of $\zeta^{[0]}$ and $\zeta^{[2]}$ around $g=0$ respectively. $\Phi_0$ and $\Phi_2$ turn out to be modified Borel type summations of those series. \\ 
As an application we derive the rising edge behaviour of $\nu_0$ and $\nu_2$ from the large order estimates of Lipatov \citep{lipatov}. It turns out to be of the form of a Gamma distribution with parameters known numerically. 
\end{abstract}

\begin{keyword}[class=MSC]
\kwd[Primary ]{60J65}
\kwd{60J10}
\kwd[; secondary ]{60E10}
\kwd{60B12}
\end{keyword}

\begin{keyword}
\kwd{ Multiple point range of a random walk, $\phi^4$-theory, quantum field theory, proper functions, Intersection Local Time, Range of a random walk, multiple points, Brownian motion }
\end{keyword}

\end{frontmatter}

\noindent

\noindent
\section{Introduction}\label{intro}
The distribution of the range and the multiple point range of the planar random walk for large length has been an important subject of mathematical research over the last 70 years \citep{Tek}. The first moment of the range for non restricted walks has been calculated by Flatto in \citep{Fla}, the second moment by Jain and Pruitt in \cite{jain} the third moment for the multiple point range in \citep[eq. 1.2]{dhoefm} and the fourth moment numerically in \citep[eq. 1.3]{dhoefm}. The first three moments of the multiple point range for the closed random walk have also been calculated in \citep[eq. 1.7, 1.8]{dhoefm}. The leading behaviour (for large length) of the distribution of the appropriately renormalized and rescaled range of a non restricted planar random walk and that of the renormalized intersection local time of the Brownian motion in two dimensions are proportional to each other with a negative real constant of proportionality as established by Le Gall \citep{legall}. This has been extended to a comparable relationship for the multiple point range by Hamana \citep{hamana_ann}. This fundamental distribution has been studied by Bass and Chen extensively \cite{bass_chen} and the tails of the distribution have been calculated to decrease exponentially and rise double exponentially. The rate of decrease has been linked to an infimum of a Gagliardo-Nirenberg inequality which is tightly related to the infimum of the Lagrangian of a planar $\phi^4$-theory \cite{bass_chen}.\\ 
A while ago I had established an exact relationship between the moments of the distribution of the multiple point range for the nonrestricted and the closed planar random walk in the limit of large length with certain integrals of the perturbation series of the $\phi^4$-theory. This perturbation series has been conjectured to be an asymptotic series which should be Borel summable to give it a more precise meaning beyond perturbation theory namely as an integral transform of the Borel sums. 
This idea has first been established long ago by Bender and Wu for the onedimensional case \cite{bender_wu}. It has also been given substantial support by the calculations of Lipatov \citep{lipatov} about large orders of perturbation theory. Yet the methodology has not been established in a rigorous mathematical sense. Our newly established relationship to well defined random walk distributions now points to a natural resumming of the perturbation theory as the characteristic function of the distributions mentioned before which leads to its own type of integral transform.\\
In this article we elaborate this idea to the  following exact relationship:\\
\begin{theorem}\label{th_main}
For $k \in \Z$, real $b > 0$ and $g \in \C\setminus\{0\}$ we define the family of integral kernels 
\begin{equation}\label{eq_dtheta}
\vartheta_k(g,b) := \exp\left(-b \cdot \left(k + z(g) + i\cdot \frac{\ln(b)}{2\pi}  \right)\right)
\end{equation}
with 
\begin{equation}
z(g) := -\frac{1}{4} +  \frac{i}{g} - \frac{i \cdot Log(g)}{2\pi}
\end{equation}
where $Log$ is the principal branch of the logarithm $-\pi < \Im(Log(g)) \leq \pi$.
There exist random variables $\nu_0$ and $\nu_2$ with Borel measures on $\R$ the moments of which equal the leading order for large length of those of the rescaled and renormalized multiple point range of the closed ($\nu_0$) and the  non restricted ($\nu_2$) planar random walk  (for the moments see \citep[eq. 8.56, 8.57]{dhoefc}, \citep[Theorem 1.1]{dhoefc} and \citep[Theorem 3.5]{hamana_ann}). Let us define the open set $R_k = \{g \in \C\setminus\{0\}: \Re(z(g)) + k > 0\}$  For $g \in R_k$  we can then define the branches of the holomorphic functions
\begin{equation}
L_0(k,g) := \int_0^{\infty} db \cdot E\left(e^{i\cdot b \cdot \nu_0}\right) \cdot e^{-\frac{i \cdot \gamma \cdot b}{2\pi}} \cdot \vartheta_k(g,b)
\end{equation}
and 
\begin{equation}
L_2(k,g) := \int_0^{\infty} db\cdot b  \cdot E\left(e^{i\cdot b \cdot \nu_2}\right) \cdot e^{\frac{i \cdot (1 - \gamma) \cdot b}{2\pi}} \cdot \vartheta_k(g,b)
\end{equation}
($\gamma$ the Euler constant and $E(.)$ the expectation value).
Deriving from them the branches of the holomorphic function 
\begin{equation}\label{eq_zeta0d}
\zeta^{[0]}(k,g) := -\frac{i\cdot g}{4\pi}LS_g\left(\rho(g)\cdot LS_g\left(\rho(g) L_0(k,g) + \frac{1}{g}\right)\right)
\end{equation}
where $LS_g$ denotes taking the local antiderivative and 
\begin{equation}
\rho(g) := \frac{d}{d g} z(g) = -\frac{i}{g^2} \cdot ( 1 + \frac{g}{2\pi})
\end{equation}
and (excluding the discrete set of points of $R_k$ where $LS_g\left(\rho(g)L_2(k,g)\right) = 0$)
\begin{equation}\label{eq_zeta2d}
\zeta^{[2]}(k,g) := \frac{ig}{LS_g\left(\rho(g)L_2(k,g)\right)} 
\end{equation}
The antiderivatives can be made unique by the constraints $\zeta^{[0]}(0,g) \sim 0$ and $\zeta^{[0]'}(0,g) \sim 0$ and $\zeta^{[2]}(0,g) \sim 1$ for $g \sim 0$, such a choice is possible.
Then for a given natural number $M$ the functions $\zeta^{[0]}(0,g)$ and $\zeta^{[2]}(0,g)$ have asymptotic expansions uniform on each sector $\left| \arg(g) - \frac{\pi}{2} \right| < \frac{\pi}{2} - \xi$ with $\xi \in \left(0,\frac{\pi}{2}\right)$
\begin{equation}
\zeta^{[i]}(0,g) = \sum_{r=0}^M \zeta^{[i]}_r \cdot (-g)^r + O(g^{M + 1})
\end{equation}
for $i = 0,2$ around $g=0$ with $\zeta^{[0]}_0 = \zeta^{[0]}_1 = \zeta^{[2]}_1 = 0$ and $\zeta^{[2]}_0 = 1$.\\ Let us on the other hand  consider the standard perturbation theory of a two dimensional Euclidian complex $m$-component quantum field theory with Lagrangian 
\begin{equation}\label{eq_lagra}
\mathcal{L}(x) = \partial_{\mu}\Phi^{\dagger}(x)\partial^{\mu}\Phi(x) + \lambda_1 \cdot \Phi(x)^{\dagger}\Phi(x) + \lambda_2 \cdot \left(\Phi^{\dagger}(x)\Phi(x)\right)^{2} 
\end{equation}
specifically the tadpole free perturbation expansions of its free energy $\Gamma^{[0],pt}$, as the logarithm of the partition function (minus the free energy contribution of the interaction free theory).  
\begin{equation}\label{eq_g0}
\Gamma^{[0],pt}(\lambda) := \sum_{r=2}^{\infty} \Gamma_{r}^{[0],pt}(\lambda)
\end{equation}
(where $\lambda = (\lambda_1,\lambda_2)$)
as the formal sum of the so called one particle irreducible vacuum diagrams with r vertices.\\
Next we consider the proper two point function as defined in \cite[eq. 4.33, 4.34]{phi4}
\begin{equation}\label{eq_g1}
\Gamma^{[2]}(\lambda,P = 0)_{\alpha,\beta} := \delta_{\alpha,\beta}  \cdot \left( \sum_{r=0}^{\infty} \Gamma_{r}^{[2],pt}(\lambda) \right)
\end{equation} 
(the external momentum $P$ is assumed to be $0$) where the sum again is a formal sum.\\ 
The coefficients $\Gamma_{r}^{[i],pt}$ are polynomials in $m$ of degree $\leq r$ 
\begin{equation}
 \Gamma_{r}^{[0],pt}(\lambda) = \lambda_1 \cdot \left(-\frac{\lambda_2}{\lambda_1}\right)^r \cdot \left(\sum_{q=0}^{r} \Gamma_{r,q}^{[0],pt}\cdot m^q  \right)
\end{equation}
and
\begin{equation}
 \Gamma_{r}^{[2],pt}(\lambda) =  \lambda_1 \cdot \left(-\frac{\lambda_2}{\lambda_1}\right)^r  \cdot \left(\sum_{q=0}^{r} \Gamma_{r,q}^{[2],pt}\cdot m^q \right)
\end{equation}
and $\Gamma_{r,q}^{[i],pt}$ (in two dimensions) are well defined mathematically in terms of absolute converging (Feynmann) integrals. Additionaly vacuum diagrams have   $\Gamma_{r,0}^{[0],pt} = 0 $ for any $r \geq 0$.\\
Then for all $r \geq 0$  
\begin{equation}\label{eq_gamma0}
\zeta^{[0]}_r = \Gamma_{r,1}^{[0],pt}
\end{equation}
and for all $r \geq 0$
\begin{equation}\label{eq_gamma2}
\zeta^{[2]}_r = \Gamma_{r,0}^{[2],pt}
\end{equation}
This means that the mathematically well defined holomorphic/meromorphic function branches $\zeta^{[0]}(k,g)$ and $\zeta^{[2]}(k,g)$ with counter clockwise continuity are natural extensions of the perturbation theory of the corresponding proper functions of this quantum field theory for $m=0$ and that standard perturbation expansion in this case is an asymptotic expansion uniform on appropriate sectors in $\C$.  
\end{theorem}
To prove this theorem the paper is organized in the following way: In the second section we define appropriate classes of graphs, adjacency matrices and permuations. In the third section we then formulate perturbation theory of an $m$ component general complex Euclidian field theory in these terms. The expansion coefficients $\Gamma_{r}^{[0],pt}(\lambda)$ and $\Gamma_{r}^{[2],pt}(\lambda)$ are polynomials in $m$. For $\Gamma_{r}^{[0],pt}(\lambda)$ the $q^{th}$ coefficients of this polynomial is related to the number of configurations of $q$ edge disjoint Eulerian paths on decompositions of Eulerian graphs and can be related to the moments of the multiple point range of the closed planar random walk for $q=1$. The same procedure is also applied to the proper two point function $\Gamma_{r}^{[2],pt}(\lambda)$. Here the $q^{th}$ coefficient is related to the number of configurations with one semieulerian path and $q$ Eulerian paths on decompositions of semieulerian graphs. Different parts of the so called self energy ordered according to graph types are analysed and allow a connection to the moments of the multiple point range of the unrestricted and the closed planar random walk for $q=0$. In the fourth section we derive basic formulas for the integral transformation with the integral kernels $\vartheta_0(g,b)$ from \eqref{eq_dtheta} with the goal to establish the uniform asymptotic expansion needed. Combining these results with the former formulas (\citep[eq. 8.56, 8.57]{dhoefc}, \citep[Theorem 1.1]{dhoefc}) for moments of the distribution of the multiple point range of the planar random walk the above theorem is easily seen to be true in the fifth section. From this relationship we can then in the sixth section connect the rising edge behaviour of the distributions to the high order behaviour of the planar $\phi^4$ theory as given by Litpatov \citep{lipatov} and Suslov \citep{Suslov_1997}. Assuming the validity of their calculations, we find the rising edge behaviour to be like that of a Gamma distribution. 

\section{The foundations: graphs, matrices and permutations}

In this paper we will construct a bridge between standard perturbation theory and multiple points of random walks. In standard perturbation theory the objects studied typically are sums of so called Feynman integrals. The summations are taken over the set of all graphs of a specific kind where each graph contribution is multiplied with a so called weight factor which depends on the number $m$ of components of the underlying field. To construct the bridge we will on a first level define sets of relevant graphs. On a second level we will refine graphs by numbering their vertices and looking at their adjacency matrices. It is on this level that results about random walks were formulated in an earlier article \citep{dhoefc}. On a third level we will refine graphs by numbering both vertices and edges and look at the underlying structure of permutations. It is on this third level that we will make the obvious connection to perturbation theory of quantum field theory based on the Wick/Isserli theorem. The other two levels still are important for the derivation of the formulas we need. \\
\subsection{Graphs}
As we will deal with complex quantum field theories we will in this article always work with directed loopfree finite multigraphs $G=(E,V,J)$ where $E$ is the finite set of edges and $V$ is the finite set of vertices. The mapping $J$ assigns two verices to each $e \in E$: the starting vertex $st(e) := J(e)_1 \in V$  and the ending vertex $end(e) := J(e)_2$. Loopfreeness means $J(e)_1 \neq J(e)_2$ for each $e \in E$.  For $v \in V$ we denote by $outd(v)$ the number of edges $e$ such that $J(e)_1 = v$ and by $ind(v)$ the number of edges $e$ such that $J(e)_2 = v$.
\begin{definition}
For nonzero $q \in \N$ let us denote the set of all directed loopfree finite multigraphs with $q$ vertices by $Gr(q)$. 
\end{definition}
\begin{definition}
A trail $T$ in a graph $G \in Gr(q)$ denotes a sequence of edges $(e_1,\ldots,e_n)$ such that $e_i \in E$ for $i=1,\ldots,n$ and $end(e_i) = st(e_{i+1})$ for $i=1,\ldots,n-1$. The vertex $st(e_1)$ is called the starting vertex $st(T)$ of $T$ and $end(e_n)$ the ending vertex $end(T)$. $T$ is said to go over the vertex $v$ iff $v = end(e_i)$ for an $1 \leq i < n$. A graph $G \in Gr(q)$ is called connected iff for any two vertices $v_1, v_2 \in V$ with $v_1 \neq v_2$ there is a trail $T$ such that $v_1 = st(T)$ and $v_2 = end(T)$. We denote the subset of graphs in $Gr(q)$ which are connected by $GC(q)$. 
\end{definition}
\begin{definition} Edges $e_1,e_2$ of a a graph $G$ are said to be parallel iff $st(e_1) = st(e_2)$ and $end(e_1) = end(e_2)$. We denote the set of equivalence classes of edges under parallelism by $\left[E(G)\right]$. 
\end{definition}
\begin{definition}
A trail $T$ is called closed iff $st(T) = end(T)$. $T$ is called Eulerian iff $T$ contains all edges of $G$ and is closed. $T$ is called semieulerian iff $T$ is not closed but contains all edges of $G$. A graph $G \in Gr(q)$ is called Eulerian if it has a Eulerian trail, it is called semieulerian if it has a semieulerian trail.  
For $G \in GC(q)$ we denote by $EC(G)$ the so called edge connectivity, i.e. the minimal number of edges whose subtraction leaves the remaining graph unconnected. For $G \in GC(q)$ we call a vertex $v$ a cutvertex iff there are vertices $v_1 \neq v$ and $v_2 \neq v$ such that any trail from $v_1$ to $v_2$ must go over $v$. 
\end{definition}
\begin{definition}
Let $G \in Gr(q)$ be a graph. A vertex $v \in V$ is called balanced if $ind(v) = outd(v)$. The subset of graphs in $GC(q)$ which have only balanced vertices is denoted by $GBa(q)$. For $w \geq 1$ we denote by $GBalp(q,w)$ the set of graphs in $GBa(q)$ which have $ind(v) \geq w$ for all vertices $v \in V$.  
\end{definition}
We will in the sequel define sets of graphs corresponding to the $r^{th}$ contribution of the so called $p$ point functions of quantum field theory: If nothing else is stated we will assume that the integers $r>0$ and $p \geq 0$: 
\begin{definition}
Let us consider graphs $G \in GC(r + 2\cdot p)$ such that $V = V_{o,1} \cup V_{o,2} \cup V_i$ with the outer outgoing vertices
 $v \in V_{o,1} \Rightarrow ind(v) = 0, outd(v) = 1$ and the outer ingoing vertices $\tilde{v} \in V_{o,2}  \Rightarrow ind(\tilde{v}) = 1, outd(\tilde{v}) = 0$ and the inner vertices $\bar{v} \in V_i \Rightarrow ind(\bar{v}) = outd(\bar{v})$ (i.e. they are balanced). Let in addition $\#V_{o,1} = \#V_{o,2} = p$. We denote the set of these graphs by $PGF(r,p,1)$. 
\end{definition}
\begin{definition}
Let us now consider graphs $G \in PGF(r,p,1)$ such that additionally $\bar{v} \in V_i \Rightarrow ind(\bar{v}) \geq w$ with a $w \geq 1$. We call the set of these graphs $PGF(r,p,w)$. 
\end{definition}
\begin{definition}
Let us consider graphs $G \in GC(r + p)$ such that $V = V_{o}\cup V_i$ with outer vertices
 $v \in V_{o} \Rightarrow ind(v) = 1, outd(v) = 1$ and inner vertices $\bar{v} \in V_i \Rightarrow ind(\bar{v}) = outd(\bar{v})$. We denote the set of those graphs by $GF(r,p,1)$. We denote the set of graphs which additionally fulfil  $\bar{v} \in V_i \Rightarrow ind(\bar{v}) \geq w$ with a $w \geq 1$ by $GF(r,p,w)$.
\end{definition}
In complex quantum field theory with one field with polynomial selfinteraction it is the set $PGF(r,p,w)$ with $w = 2$ which is the basic set of graphs which contributes to a $p$ point function of $r$th order. \\
We will now also define useful simple operations on graphs. 
\begin{definition}\label{de_prret}
For $G \in GF(r,p,1)$ with $p > 0$ we define the graph $Pr(G) \in PGF(r,p,1)$ in the following way: We define $V_{o,1} := V_o$ add a $p$ element vertex set $V_{o,2}$ and in $Pr(G)$ connect each edge $e$ which had $end(e) \in V_o$ with a different vertex $ \overline{end(e)} \in V_{o,2}$. The relationship $end(e) \mapsto \overline{end(e)}$ then is a bijection between $V_{o,1}$ and $V_{o,2}$, which we denote by $vp$. \\
\\ 
On the other hand: For each $G \in PGF(r,p,1)$ together with a bijection 
\begin{align*}\label{eq_vpdf}
vp :& V_{o,2} & \mapsto & V_{o,1} \\
   & v & \mapsto & vp(v) 
\end{align*}
we define the graph $Ret(G,vp) \in GF(r,p,1)$ by connecting each edge $e$ with $ end(e) \in V_{o,2} $ with $vp(end(e))$ instead, subtract $V_{o,2}$ and set $V_o := V_{o,1}$. For $p=1$ the bijection $vp$ is trivial and therefore we will write $Ret(G)$ in this case.\\
For  $G \in GF(r,0,1)$ we define $Pr(G) := G$ and $Ret(G) := G$. This definition for $p=0$ is important to keep later formulas consistent and meaningful.
\end{definition}
\begin{definition}
Let $G \in PGF(r,1,1)$. We denote the one vertex in $V_{o,1}$ by $st(G)$ and the one vertex in  $V_{o,2}$ by $end(G)$ and the one edge $e$ with $st(e) = st(G)$ with $ank(G)$ and the one edge $\tilde{e}$ with $end(\tilde{e}) = end(G)$ with $bnk(G)$.
\end{definition}
\begin{definition}\label{de_amp}
For $G \in PGF(r,p,1)$ with $p > 0$ we define the amputated graph $Amp(G) \in GC(r)$ by subtracting $V_{o,1}$ and $V_{o,2}$ and subtracting all edges with $end(e) \in V_{o,2}$ or $st(e) \in V_{o,1}$. For $p=0$ we define $Amp(G) := G$.\\
For $G \in GF(r,p,1)$ with $p > 0$ we define the graph $Amp(G) \in GC(r)$ by subtracting $V_{o}$ and subtracting all edges with $end(e) \in V_{o}$ or $st(e) \in V_{o}$. For $p=0$ we define $Amp(G) := G$. This definition for $p=0$ is important to keep later formulas consistent and meaningful.
\end{definition}
\begin{definition}\label{de_psfpef}
Let $G \in GF(r,1,1)$. It then has one vertex $v \in V_o$ which we denote by $out(G)$. We denote by $ank(G) \in E(G)$ the edge such that $st(ank(G)) = v$ and $bnk(G) \in E(G)$ the edge such that $end(bnk(G)) = v$. A graph $G \in GF(r,1,1)$ is called closable iff $end(ank(G)) \neq st(bnk(G))$ otherwise unclosable. We can therefore define the two disjoint sets $GFC(r,1,1)$ as the set of closable graphs and $GFU(r,1,1)$ as the set of unclosable graphs.  
\begin{equation}\label{eq_sfe}
GF(r,1,1) = GFC(r,1,1) \dot{\cup} GFU(r,1,1)
\end{equation}
In the same way $GFC(r,1,w)$ and  $GFU(r,1,w)$ denote the corresponding subsets of $GF(r,1,w)$ and $PGFC(r,1,w)$ and $PGFU(r,1,w)$ the corresponding subsets of $PGF(r,1,w)$.
\end{definition}
\begin{definition}
We denote by $GFE(r,0,1)$ the set of pairs $(G,e)$ where $G \in GF(r,0,1)$ and $e \in \left[E(G)\right]$.
\end{definition}
We now come to an important one to one relationship between $GFE(r,0,1)$ and $GFC(r,1,1)$
\begin{definition}\label{de_divcl}
Let $G \in GF(r,0,1)$ and $e$ be one of its edges. 
We define the graph $Div(e,G) \in GFC(r,1,1)$ by adding to $G$ the set $V_o := \{v\}$ and replacing $e$ by two edges $e_1, e_2$ such that $st(e_1) := st(e)$, $end(e_2) = end(e)$ and $st(e_2) = end(e_1) = v$.  \\
Let on the other hand $G \in GFC(r,1,1)$ be given. We define the graph $Cl(G) \in GF(r,0,1)$ by subtracting $V_o$ and replacing the pairs of edges $(bnk(G),ank(G))$ by one edge $Cle(G)$ such that $st(Cle(G)) := st(bnk(G))$ and $end(Cle(G)) := end(ank(G))$. \\
We notice that for $G \in GFC(r,1,1)$
\begin{equation} 
Cl(Div(e,G)) = G
\end{equation}
up to isomorphy of graphs and 
$Cle(Div(e,G))$
is either $e$ or one of the edges in $G$ which is paralell to $e$.
We also note
\begin{equation}
Div(Cle(H),Cl(H)) = H
\end{equation}
up to isomorphy of graphs.
$Div$ and the pair $(Cle,Cl)$ therefore constitute a one to one relationship and its inverse between $GFC(r,1,1)$ and $GFE(r,0,1)$. 
\begin{definition}
For $G \in GF(r,0,1)$ we denote the set $DG(G) := \{H \in GFC(r,1,1): Cl(H) = G\}$. We then note 
\begin{equation}\label{eq_bigcupb}
GFC(r,1,1) = \dot{\bigcup}_{G \in GF(r,0,1)} DG(G)
\end{equation}
\begin{equation}\label{eq_bigcupw}
GFC(r,1,w) = \dot{\bigcup}_{G \in GF(r,0,w)} DG(G)
\end{equation}
\end{definition}

\end{definition}
\begin{definition}
For $G \in GF(r,p,1)$ we define 
\begin{equation}\label{eq_sumg}
Sum(G) := \sum_{v \in V_i(G)} ind(v)
\end{equation}
\end{definition}

We now give a short definition of the so called Feynman integral which is the main ingredient of the perturbation theory of quantum field theory.

\begin{definition}\label{de_feyg}
For any connected loopfree finite multigraph (directed or non\-di\-rec\-ted) $G$ the (fully massive bosonic) Feynman integral  $I_G$ in $d$ dimensions with external momenta zero is defined  \citep[eq. 8.5]{itzykson} for $\Re(1 - d/2) \geq 0$ as
\begin{equation}\label{eq_intig}
I_G(d) := \frac{1}{(4\cdot \pi)^{L(G)\cdot \frac{d}{2}}} \Gamma_G\left(1-\frac{d}{2}\right)
\end{equation}
where  
\begin{equation}
L(G) := \#E(G) - \#V(G) + 1
\end{equation}
is the so called number of loops of $G$ and $E(G)$ and $V(G)$ the set of edges and vertices of $G$ respectively and $\#$ denotes the number of elements. The generalized gamma function $\Gamma_G$ of a connected loopfree multigraph $G$ is defined for $\Re(n) \geq 0$ by the integral
\begin{equation}\label{eq_intgamma}
\Gamma_G(n) := \int \left(\prod_{e \in E(G)} d\alpha_e \right) \cdot \exp\left( - \sum_{e \in E(G)} \alpha_e \right) \cdot \mathcal{P}_G(\alpha)^{n-1}
\end{equation}
which is holomorphic in $n$.
The integration over the real variables $\alpha_e$ for each edge $e \in E(G)$ is performed over the interval $[0,\infty]$ and $\mathcal{P}_G(\alpha)$ is the so called Kirchhoff-Symanzik polynomial defined by
\begin{equation}\label{eq_poly}
\mathcal{P}_G(\alpha) := \sum_{T \in sp(G)} \left( \prod_{e \in E(G \setminus T)} \alpha_e \right)
\end{equation}
where $sp(G)$ is the set of spanning trees of $G$.
\end{definition}
\begin{definition}
For any connected loopfree finite multigraph $G$ (directed or nondirected) we define the so called mass factor 
\begin{equation}
FM_G(\lambda_1) :=  \frac{\lambda_1^{L(G)\cdot \frac{d}{2}}}{\lambda_1^{\#E(G)}}
\end{equation}
where $\lambda_1 \in \C\setminus \{0\}$ is called square of the \lq\lq mass\rq\rq\  by physicists.
\end{definition}
\begin{lemma}
For $G \in GF(r,0,1)$ 
\begin{equation}\label{eq_divint}
\sum_{e \in E(G)} I_{Div(e,G)}(d) =  \left( \#E(G) - L(G)\cdot \frac{d}{2} \right) \cdot I_G(d)
\end{equation}
\end{lemma}
\begin{proof}
We note from \citep[eq. 6.94, 8.4, 8.5]{itzykson} that
\begin{equation} 
 FM_G(\lambda_1) \cdot I_G(d)
\end{equation} 
is an integral over so called \lq\lq Feynman parameters\rq\rq\  but can also be written as an integral in \lq\lq momentum space\rq\rq\  with the integrand being a product over all edges $ed$ of the so called propagator
\begin{equation}
\frac{1}{p_{ed}^2 + \lambda_1}
\end{equation} 
where each edge $ed \in E(G)$ has a momentum $p_{ed}$ associated with it (and of course the obligatory constraints on these \lq\lq momentum variables\rq\rq\  in each vertex). 
Now the integrand for $I_{Div(e,G)}(d)$ in these variables is identical to the one of $I_G(d)$ only the factor
\begin{equation}
\frac{1}{p_e^2 + \lambda_1}
\end{equation} 
is replaced by the factor 
\begin{equation}
\frac{1}{p_{e_1}^2 + \lambda_1} \cdot \frac{1}{p_{e_2}^2 + \lambda_1}
\end{equation}
with the additional constraint $p_{e_1} = p_{e_2}$ for the edges $e_1,e_2$ in the outer vertex of $Div(e,G)$. So by differentiating in $\lambda_1$ in the \lq\lq momentum space\rq\rq\  integrals we find 
\begin{equation}\label{eq_difint}
\sum_{e \in E(G)} FM_{Div(e,G)}(\lambda_1) \cdot I_{Div(e,G)}(d) =  \left( -\frac{\partial}{\partial \lambda_1} \right) FM_G(\lambda_1) \cdot I_G(d) 
\end{equation}  
which is equivalent to equation \eqref{eq_divint}

\end{proof}

\subsection{Matrices}
We now turn to suitable sets of integer matrices. They will turn out to be a refinement of $Gr(q)$ and its subsets defined before. 
\begin{definition}
For nonzero $q \in \N$ let us denote the set of all $q \times q$  matrices $F$ with nonnegative integer $F_{i,j} \in \N_0$ and zero diagonal $F_{i,i} = 0; \forall i=1,\ldots,r$ by $Ma(q)$. 
\end{definition}
\begin{definition}\label{de_bal}
A matrix $F \in Ma(q)$ is called balanced, i.e. $F \in MBal(q)$ iff there exists nonnegative integers $dias(F)_k \in \N_0$ such that for $k=1,\ldots,q$
\begin{equation}
dias(F)_k = \sum_{i=1}^q F_{i,k} = \sum_{i=1}^q F_{k,i}
\end{equation}
For $F \in MBal(q)$ we define  the $q \times q$ matrix $Diag(F)$  by 
\begin{equation}
Diag(F)_{i,k} := \delta_{i,k}\cdot dias(F)_k
\end{equation} 
where $\delta$ is the Kronecker-Symbol.
Let us further define the set $MBaln(q) := MBal(q)\setminus\{ 0_{q\times q}\}$.
\end{definition}
\begin{definition}\label{de_balpu}
For integers $w > 0$, $r > 0$ and $p \geq 0$ we also define the set $MF(r,p,w)$: \\ 
$F \in MF(r,p,w)$ iff 
$F \in MBaln(r + p)$ and 
\begin{equation}
dias(F)_k  \geq w
\end{equation}
for all $k = 1,\ldots,r$ and
\begin{equation}
dias(F)_k = 1
\end{equation}
for all $k=r+1,\ldots,r+p$ and
\begin{equation}\label{eq_cofmf}
cof(Diag(F) - F) \neq 0
\end{equation}
where $cof$ denotes the cofactor. \\
For $F \in MF(r,p,w)$ we define the vector 
$dia(F) := (dias(F)_1,\ldots,dias(F)_r)$ and 
\begin{equation}
Sum(F) := \sum_{i=1}^r dias(F)_i
\end{equation}
. 
\end{definition}

\begin{definition}\label{de_gr}
For a matrix $F \in MBaln(r + p)$ and a permutation $\sigma \in S_{r + p}$ we define  $F^{\sigma}$ by 
\begin{equation}\label{eq_fsigma}
F^{\sigma}_{i,j} := F_{\sigma(i),\sigma(j)}
\end{equation}
We also define the subgroup $S_r(p) \subset S_{r+p}$ as the set of permutations $\sigma$ with $\sigma(r+k) > r $ for $k=1,\ldots,p$. Then $S_r(p)$ is isomorphous to $S_r \times S_p$. 
For $F \in MF(r,p,w)$ we then define the set
\begin{equation}
Fix(F) := \{\sigma \in S_{r}(p): F^{\sigma} = F\}
\end{equation}
and the set 
\begin{equation}
WM(F) := \{F^{\sigma}; \sigma \in S_r(p)\}
\end{equation}
Then $Fix(F)$ is a subgroup of $S_r(p)$. We define a symmetry factor 
\begin{equation}
Syf(F) := \#Fix(F)
\end{equation}
 and the number of cosets 
\begin{equation}
gr(F) := \#WM(F)
\end{equation}
According to Lagranges theorem we then have 
\begin{equation}
gr(F)\cdot Syf(F) = r!\cdot p!
\end{equation}
\end{definition}
\begin{definition}\label{de_grpheu}
Let $G \in Gr(q)$. Then its adjacency matrix $Adj(G)$ is uniquely defined up to a bijective mapping $map_G$ of its vertex set $V$ to the set of dimensions $\{1,\ldots,q\}$. We call the set of these bijective vertex mappings $BV(G)$. \\ 
For a graph $G \in GF(r,p,w)$ we define such a mapping $map_G\in BV(G)$ to be admissible, i.e. $map_G \in BVA(G)$ iff any vertex $v \in V_o$ is mapped onto a dimension $map_G(v) > r$. 
\end{definition}
\begin{definition}\label{de_graph}
For a matrix $F \in Ma(q)$ we can define a loopfree multiple directed graph $G(F) \in Gr(q)$ in the obvious way: the set of vertices $V := \{1,\ldots,q\}$ a set of edges $E := \{e_{i,j}^{[k]}; 1 \leq i,j \leq q;  1 \leq k \leq F_{i,j} \}$ and the direction function $ st\left(e_{i,j}^{[k]} \right) = i$ and $ end\left(e_{i,j}^{[k]} \right) = j$ independant of $k$. With this definition $F$ is the adjacency matrix of $G(F)$ if the vertex $i$ is mapped onto the dimension $i$ of the $q \times q$ matrix. We call the correponding map $map_{G(F)}^{[c]} \in BV(G(F))$ the canonical map of $G(F)$ and under this map obviously. 
\begin{equation}
F  = Adj(G(F))
\end{equation}
\end{definition}
\begin{remark}
Let $G \in GF(r,p,w)$ and $map_{G,1} \in BVA(G)$ and $map_{G,2} \in BVA(G)$ be two different mappings of the vertices onto $\{1,\ldots,r + p\}$ and $Adj(G)_1$ and $Adj(G)_2$ the two adjacency matrices. Then there is a $\sigma \in S_r(p)$ such that $map_{G,1} = map_{G,2} \circ \sigma$ i.e. $Adj(G)_1 = Adj(G)_2^{\sigma}$. Any mapping $map_{G,1}\circ \sigma \in BVA(G)$ for any $\sigma \in S_r(p)$. 
\end{remark}
\begin{lemma}
Let $G \in GF(r,p,w)$. Then under any $map_G \in BVA(G)$ the matrix $Adj(G) \in MF(r,p,w)$. Let on the other hand $F \in MF(r,p,w)$. Then $G(F) \in GF(r,p,w)$. 
\end{lemma}
\begin{proof}
Let $G \in GF(r,p,w)$. Then $Adj(G)$ is an integer and balanced matrix with $0$ diagonal by construction. For $v \in V$ we have $dias(Adj(G))_{map_G(v)} = ind(v)$ and $dias(Adj(G))_k \geq w$ for $1 \leq k \leq r$ and $dias(Adj(G))_k = 1$ for $r+1 \leq k \leq r+p$. Moreover because $G$ is connected it is Eulerian and therefore $cof(Diag(Adj(G)) - Adj(G)) \neq 0$ and therefore $Adj(G) \in MF(r,p,w)$ \\
Let on the other hand $F \in MF(r,p,w)$. Then $G(F)$ is a balanced graph. It is connected because the number of its Euler trails is not zero because of equation \eqref{eq_cofmf} and for vertex $k$ we have $ind(k) = outd(k) = dias(F)_k \geq w$ for $1 \leq k \leq r$ and $ind(k) = outd(k) = dias(F)_k = 1$ for $r+1 \leq k \leq r+p$. But therefore $G(F) \in GF(r,p,w)$. 
\end{proof}
\begin{definition}
For a matrix $F \in Ma(q)$ we define a multiple directed graph $GK(F)$ to be the graph $G(F)$ without possible isolated vertices of $G(F)$. 
\end{definition}
\begin{definition}
For a given admissible mapping $map_G \in BVA(G)$ of the graph $G \in GF(r,p,w)$ onto matrix dimensions we define the symmetry factor $Syf(G)$ of such a graph by 
\begin{equation}
Syf(G) := Syf(Adj(G))
\end{equation}
It is clear that $Syf(G)$ does not depend on the specific mapping $map_G \in BVA(G)$
We now note that there are exactly
\begin{equation}\label{eq_syfg}
gr(G) := \frac{r!\cdot p!}{Syf(G)} 
\end{equation}
different matrices $F \in MF(r,p,w)$ such that there is a bijective mapping $map_G \in BVA(G)$ resulting in $Adj(G) = F$.  
\end{definition}
\begin{remark} For $G \in GF(r,p,w) $ the number
$Syf(G)$ is also the number of those graph automorphisms of $G$ which operate on the vertices and map outer vertices onto outer vertices and inner vertices onto inner vertices. 
\end{remark} 
\begin{definition}
Let $F \in MF(r,p,w)$ with $p > 0$. We then define the $r \times r$ matrix $Amp(F)$ by $Amp(F)_{i,j} := F_{i,j}$ for $1 \leq i,j \leq r$. Then $Amp(F) = Adj(Amp(G(F)))$. ($Adj$ here is taken under the canonical mapping). We also define $Pr(F) := Adj(Pr(G(F)))$. In this case $Adj$ is taken under the canonical mapping for the vertices $ v \in V_{o,1} \cup V_i$. Additionaly we map $ v \in V_{o,2}$ onto $r + p +k$ iff $vp^{-1}(v)$ was mapped onto $r + k$ by the canonical mapping.
\end{definition}
\begin{definition}
Let $F \in MF(r,1,w)$. Then, because $dias(F)_{r+1} = 1$ we must have $F_{\alpha,r+1} = \delta_{\alpha,i}$ and $F_{r+1,\beta} = \delta_{j,\beta}$ for numbers $1 \leq i,j \leq r$. We denote $bnke(F) := i$ and $anke(F) := j$ and $Cle(F) := (bnke(F),anke(F))$. A matrix $F \in MF(r,1,w)$ is called closeable iff $bnke(F) \neq anke(F)$ and unclosable otherwise. We denote the subset of closeable matrices in  $MF(r,1,w)$ by $MFC(r,1,w)$ and the set of uncloseable matrices by $MFU(r,1,w)$.    
\end{definition}
\begin{definition}
Let $F \in MFC(r,1,w)$. We then define the $r \times r$ matrix $Cl(F)$ by 
\begin{equation}
Cl(F)_{\alpha,\beta} = F_{\alpha,\beta} + \delta_{\alpha,bnke(F)}\cdot \delta_{\beta,anke(F)}
\end{equation}
for $1 \leq \alpha,\beta \leq r$. Then $Cl(F) = Adj(Cl(G(F)))$ under the canonical mapping and therefore $Cl(F) \in MF(r,0,w)$.
\end{definition}
\begin{definition}
Let $F \in MF(r,0,w)$ and $1\leq i,j \leq r$ and $i\neq j$ and $F_{i,j} > 0$.
Then we define the $(r + 1) \times (r + 1)$ matrix $J:=Div((i,j),F)$ by $J_{\alpha,\beta} := F_{\alpha,\beta}$ iff $1 \leq \alpha,\beta \leq r$ and $i \neq \alpha$ or $j \neq \beta$. $J_{\alpha,r+1} := \delta_{i,\alpha}$ and $J_{r+1,\beta} = \delta_{j,\beta}$ and $J_{i,j} := F_{i,j} - 1$. With this definition $Div((i,j),F) \in MFC(r,1,w)$ and $Div((i,j),F) = Adj(Div(e,G(F))$ 
where $e$ is any edge $e_{i,j}^{[k]}$ and $Adj$ is taken under the canonical mapping of $G(F)$ for $v \in V_i$ and the additional vertex is mapped onto $r + 1$. We notice that $Cl(Div((i,j),F)) = F$ for any $F \in MF(r,0,w)$ and any $(i,j)$ such that $F_{i,j} \neq 0$. Moreover $Div(Cle(J)),Cl(J)) = J$ for any $J \in MFC(r,1,w)$. 
\end{definition}
\begin{definition}
Let $F \in MFC(r,0,w)$. We define the set 
\begin{equation}
DM(F) := \{H \in MF(r,1,w): Cl(H) = F\}
\end{equation}
\end{definition}
\begin{definition}
For integer $q \geq 1$ let there be numbers $\lambda_q \in \C$ 
of which only finitely many are different from $0$.  Let us denote the vector $\lambda = (\lambda_1,\lambda_2,\ldots )$. Let $w > 1$. For $F \in MF(r,p,w)$ we define 
\begin{equation}
\Pi_{F}(\lambda) := \prod_{k=1}^r \left(-\lambda_{dias(F)_k}\right)
\end{equation}
By the same token for $G \in GF(r,p,w)$ or $G \in PGF(r,p,w)$ we define 
\begin{equation}
\Pi_{G}(\lambda) := \prod_{v \in V_i} \left(-\lambda_{ind(v)} \right)
\end{equation}
Therefore for $G \in GF(r,p,w)$
\begin{equation}
\Pi_{Pr(G)}(\lambda) = \Pi_{G}(\lambda)
\end{equation}
and 
\begin{equation}
\Pi_{G}(\lambda) = \Pi_{Adj(G)}(\lambda)
\end{equation}

\end{definition}

\subsection{Permutations}
In a further step of refinement we will now define suitable permutations which correspond to the matrices and graphs we have defined before. 

\begin{definition}\label{de_tup}
For a natural number $w \in \N$ with $w \geq 1$ we define $Tup(r,w)$ as the set of all ordered $r$ tuples $(d_1,\ldots,d_r)$ with $d_k \in \N$ and $d_k \geq w$ for all $k=1,\ldots,r$\\ 
For $w > 1$ and $tu := (d_1,\ldots,d_r) \in Tup(r,w)$ we define $Sum(tu) := \sum_{k=1}^r d_k$
and
\begin{equation}
\Pi_{tu}(\lambda) := \prod_{k=1}^r \left(-\lambda_{d_k}\right)
\end{equation}
\end{definition}

\begin{definition}\label{de_qaqb}
For $tu :=  (d_1,\ldots,d_r) \in Tup(r,w)$ and $p \geq 0$ we define the integers $D_0(tu) :=  0$ and $D_i(tu) := D_{i - 1}(tu) + d_i$ for $i=1,\ldots,r$ with $R := Sum(tu)$ and $D_i(tu) :=R - r + i$ for $r < i \leq r+ 2\cdot p$. \\
We also define the functions 
\begin{equation}
qa_{tu}: \{1,\ldots,R + p\} \rightarrow \{1,\ldots,r + p\}
\end{equation} and 
\begin{equation}
qb_{tu}: \{1,\ldots,R + 2\cdot p\} \rightarrow \{0,\ldots,r + 2\cdot p\}
\end{equation}
 defined for $j = 1,\ldots,R$ by  
\begin{equation}
qa_{tu}(j) = qb_{tu}(j) = u \Leftrightarrow D_{u-1}(tu) < j \leq D_u(tu)
\end{equation}
and by
$qa_{tu}(j) := r+ j - R $ and $qb_{tu}(j) := r + p + j-R$ for $R < j \leq R + p$ \\ and by $qb_{tu}(j) := 0$ for $R + p < j \leq R + 2\cdot p$. \\
\end{definition}
\begin{definition}
Let $tu = (d_1,\ldots,d_r) \in Tup(r,2)$, $R = Sum(tu)$ and an integer $p \geq 0$ be given and let $qa_{tu}$ be the corresponding function defined in definition \ref{de_qaqb}. Then a permutation $\sigma \in S_{R+p}$ is called loopfree iff $qa_{tu}(k) \neq qa_{tu}(\sigma(k))$ for $k = 1,\ldots,R+p$. We call the set of these loopfree permutations $SY(tu,r,p)$
\end{definition}
\begin{definition}
For $\sigma \in S_{R + p}$ we define the extension $Ext(\sigma) = \bar{\sigma} \in S_{R + 2 \cdot p}$ to be $\bar{\sigma}(k) := \sigma(k)$ for $1 \leq k \leq R + p $ and $\bar{\sigma}(k) := k$ for $R + p < k \leq R + 2\cdot p$. 
\end{definition}
\begin{definition}
Let $tu \in Tup(r,w)$ and $p \geq 0$ and 
 let $\sigma \in SY(tu,r,p)$. We define the following integer matrices:
The $(r+p)\times (r+p)$ matrix
\begin{equation}
Adj(\sigma,tu,p)_{i,j} := \sum_{k=D_{i-1}(tu)+1}^{D_i(tu)} \delta_{qa_{tu}(\sigma(k)),j}
\end{equation}  
and the $(r + 2p) \times (r + 2p)$ matrix 
\begin{equation}
Adjo(\sigma,tu,p)_{i,j} := \sum_{k=D_{i-1}(tu) + 1}^{D_i(tu)} \delta_{qb_{tu}(Ext(\sigma)(k)),j}
\end{equation}  
\end{definition}
\begin{definition}\label{de_adcon}
	Given the data $tu \in Tup(r,2)$ and $p \geq 0$ as before we call a permutation $\sigma \in SY(tu,r,p)$ connected iff $Adj(\sigma,tu,p) \in MF(r,p,1) $. We call  $\sigma$ admissible if either $p=0$ or if $\sigma$ contains $p$ element disjoint cycles $cyc_1,\ldots,cyc_p$ such that $cyc_k(R+k) \neq R + k$ (for $p=1$ of course admissibility is equivalent to loopfreeness). We call the set of permutations which are loopfree, admissible and connected $SF(tu,r,p)$. 
\end{definition}
\begin{remark}
With this definition it is easy to see that for $tu = (d_1,\ldots,d_r)$ and $\sigma \in SF(tu,r,p)$
\begin{equation}
dias(Adj(\sigma,tu,p))_q = d_q
\end{equation}
for $q = 1,\ldots,r$ and 
\begin{equation}
dias(Adj(\sigma,tu,p))_q = 1
\end{equation}
for $q = r+1,\ldots, r+p$.
\end{remark}
\begin{definition}\label{de_gsigma}
Let $tu \in Tup(r,w)$ and $p \geq 0$ an integer. For $\sigma \in SF(tu,r,p)$ we can define the graph $G(\sigma)$ in the following way: \\
The vertex sets $V_i := \{1,\ldots,r\}$ and $V_o := \{r+1,\ldots,r+p\}$, the set of egdes $E := \{1,\ldots,Sum(tu) + p \}$ and $J(b)_1 := qa_{tu}(b)$ and $J(b)_2:= qa_{tu}(\sigma(b))$. With these definitions $G(\sigma) \in GF(r,p,w)$ and $Adj(\sigma,tu,p)$ is an adjacency matrix of $G(\sigma)$ with the obvious mapping of the vertices and $Adjo(\sigma,tu,p) = Adj(Pr(G(\sigma)))$ with the obvious mapping.
\end{definition}
So the permutations $\sigma \in SF(tu,r,p)$ are a refinement of both the graphs and the adjacency matrices belonging to them. We therefore now work on a relationship between matrices $F \in MF(r,p,w)$ and permutations $\sigma \in SF(tu,r,p)$ such that $Adj(\sigma,tu,p) = F$. 

For this purpose we define the so called weight polynomial which plays a fundamental role in quantum field theory. This paper will give a fundamentally new formulation of it to connect the Wick/Isserli evaluation of moments of quantum field theory to random walk results. For the standard formulation (differing by a global factor) see e.g. \citep[ch. 6]{phi4}.
\begin{definition}
For $\sigma \in SF(tu,r,p)$ we define the so called weight function $wei_\sigma$ as the monomial
\begin{equation}
wei_\sigma(m,p) := m^{cycn(\sigma)-p}
\end{equation}
where $cycn(\sigma)$ is the number of cycles of $\sigma$ and $m$ is a general complex variable. \\
For $F \in MF(r,p,w)$ we define the weight set
\begin{equation} 
WS(F) := \{\sigma \in SF(dia(F),r,p); Adj(\sigma,dia(F),p) = F\}
\end{equation}
and the weight function as the polynomial
\begin{equation}
wei_F(m,p) := \sum_{\sigma \in WS(F)} wei_{\sigma}(m,p)
\end{equation} 
\end{definition}
We will now work on a different formulation of $wei_{F}(m,p)$ in terms of partitions of $F$, avoiding any reference to permutations. This will result in the major connection point to results of random walks. 
\begin{definition}\label{de_komp}
For a matrix $F \in MBaln(q)$ we define the following quantities: \\  
Let $komp(F)$ be the matrix which results from $F$ by deleting all columns and rows with an index $k$ such that $dias(F)_k=0$. If $komp(F)$ is a $n\times n$ matrix we define $red(F) := n$ \\ 
We also define  
\begin{equation}
Mud(F) := \prod_{1 \leq k \leq r} dias(F)_k!
\end{equation}
and 
\begin{equation}
Co(F) = cof\left(Diag\left(komp\left(F\right)\right) - komp\left(F\right)\right)
\end{equation}
\begin{equation}
Eul(F) := \left(\prod_{1\leq k \leq red(F) } \left(dias(komp(F))_k-1\right)!\right) 
\cdot Co(F)
\end{equation}
where $cof$ is the cofactor of the $red(F) \times red(F) $ matrix in brackets behind it. We further define
\begin{equation}
Mult(F) := \prod_{1\leq i,j \leq r} F_{i,j}!
\end{equation}
\end{definition}
\begin{lemma}\label{le_coneul}
Let $F \in MBaln(q)$ and $G$ a graph with $Adj(G) = komp(F)$. $G$ is connected iff $Eul(F) \neq 0$.  
\end{lemma}
\begin{proof}
Let $G$ be connected. Then $Eul(F)$ is the number of Euler trails on $G$ \cite{aar,tutte} which is nonzero because $G$ as a balanced and connnected finite graph is Eulerian.  Let on the other hand $G$ be a graph which is not connected. Then because $G$ is a balanced graph and nonempty it must have connection components which we call $G_1,\ldots,G_k$. They are balanced and Eulerian.  Therefore $Diag(komp(Adj(G_j))) - komp(Adj(G_j))$ has a kernel with dimension $1$. Therefore the matrix $Diag(komp(Adj(G))) - komp(Adj(G))$ has a kernel with dimension $k > 1$. Therefore $cof(Diag(komp(F)) - komp(F)) = 0$.
\end{proof}
\begin{definition}
Let $F \in MBaln(q)$. We call $F$ connected iff $Eul(F) \neq 0$. We call the set of connected matrices $MC(q)$
\end{definition}
\begin{definition}\label{de_bpp}
For a matrix $F \in MF(r,p,1)$ we define a $p$ admissible balanced $s$-partition $\tau \in Partc(F,p,s)$ to be any ordered tupel 
\begin{equation}
\tau = (\tau(1),\ldots,\tau(s))
\end{equation}
of matrices $\tau(\alpha) \in MBaln(r+p)$ with 
\begin{equation}\label{eq_admiss}
\sum_{k=r+1}^{r+p} dias(\tau(\alpha))_k \leq 1
\end{equation}
 for $\alpha=1,\ldots,s$ such that
\begin{equation}\label{eq_tausum}
F = \sum_{1 \leq i \leq s} \tau(\alpha)
\end{equation} 
We define $Part(F,p,s)$ to be the set of equivalence classes of $Partc(F,p,s)$ under permutation in the $s$-tupel and identify any representative of the class with it in the sequel. Because of equation \eqref{eq_admiss} it is trivial to note that a $p$-admissible partition must contain at least $p$ elements.
\end{definition}
\begin{definition}
For $\tau \in Partc(F,p,s)$ and $\tau = (\tau(1),\ldots,\tau(s))$ choose the matrices $\tau^{(1)}, \ldots,\tau^{(q)}$ to be such that \\
$ i \neq j \Rightarrow \tau^{(i)} \neq \tau^{(j)} $ and \\
for each $1 \leq i \leq s$ there is a number $k$ such that $\tau(i) = \tau^{(k)}$. \\ Let $m_k$ be the number 
\begin{equation}
m_k := \#\{\tau(i): i=1\ldots,s \wedge \tau(i) = \tau^{(k)}\}
\end{equation}
Then we define 
\begin{equation}
Sym(\tau) := \prod_{k=1}^{q} m_k!
\end{equation}
$Sym(\tau)$ does not depend on the representative of $\tau$ in $Partc(F,p,s)$ and so it is also well defined on $Part(F,p,s)$. 
\end{definition}
\begin{definition}
Let $F \in MF(r,p,w)$ and $\tau \in Partc(F,p,s)$ and 
$tu = dia(F)$.
We define a selection $g$ of $\tau$ to be any function 
\begin{equation}
g: \{1,\ldots,R + p\} \mapsto \{1,\ldots,s\} \times \{1,\ldots,r + p\}
\end{equation} 
which fulfils the following conditions: 
\begin{equation}
\tau(\alpha)_{i,k} = \#\left(g^{-1}(\alpha,k) \cap \{D_{i-1}(tu) + 1,\ldots,D_i(tu)\} \right)
\end{equation}
for $1 \leq i,k \leq r + p$ and $1 \leq \alpha \leq s$. \\
We denote the set of selections of $\tau$ by $Sel(\tau)$. By simple combinatorics we note that 
\begin{equation}\label{eq_selnum}
\#Sel(\tau) = \frac{Mud(F)}{\prod^s_{\alpha = 1}Mult(\tau(\alpha))} 
\end{equation}
\end{definition}
\begin{definition}\label{de_canong}
Let $F \in MF(r,p,w)$ and $\tau \in Partc(F,p,s)$ and $tu = dia(F)$.
Let $g \in Sel(\tau)$. We define the $g$-canonical representation of graphs $(G_1,\ldots,G_s)$ with $G_\alpha = (V_\alpha, E_\alpha, J_\alpha)$ by defining $i \in V_\alpha$ iff 
\begin{equation}
\sum_{k=1}^{r+p} \tau(\alpha)_{i,k} \neq 0
\end{equation}
and $b \in E_\alpha$ iff 
\begin{equation}
b \in \bigcup_{k=1}^{r+p} g^{-1}(\alpha,k)
\end{equation}
and $J_\alpha(b)_1 = qa_{tu}(b)$
and  $J_\alpha(b)_2 = g(b)_2$ (where $g(.)_2$ denotes the second component of $g$).\\ 
With this representation $Adj(G_\alpha) = komp(\tau(\alpha))$ and $G_{\alpha} = GK(\tau(\alpha))$ up to isomorphism of graphs.  
\end{definition}
\begin{definition}\label{de_eulcol}
With the same data as in definition \ref{de_canong} let $(G_1,\ldots,G_s)$ be a $g$-canonical representation. Then we denote the set of collections of Euler trails $(T_1,\ldots,T_s)$ where $T_\alpha$ is an Euler trail on $G_\alpha$ by $Col(g,\tau)$. Then we know that for a given $\tau \in Partc(F,p)$ 
\begin{equation}\label{eq_eulcol}
\#Col(g,\tau) = \prod_{\alpha=1}^s Eul(\tau(\alpha))
\end{equation}  
(according to lemma \ref{le_coneul} this is even true if $Col(g,\tau) = \emptyset$).
Each Euler trail $T_\alpha$ can in this representation be written by a sequence of integers
\begin{equation} 
T_\alpha = (b_1(\alpha),\ldots,b_{\#E_\alpha}(\alpha))
\end{equation}
which denotes the sequence of edges in the canonical representation. The sequence is unique up to a cyclic permutation. 
\end{definition}
\begin{lemma}\label{le_sigcol}
Let $F \in MF(r,p,w)$ and $tu = dia(F)$.
and $\sigma \in WS(F)$. Then $\sigma$ uniquely induces a partition $\tau \in Partc(F,p,cycn(\sigma))$ up to a permutation of the $\tau(\alpha)$ and it uniquely induces a selection $g \in Sel(\tau)$ up to one of the $Sim(\tau)$ permutations of those indices $\beta$ which belong to identical matrices $\tau(\beta)$  in the partition and a collection of Euler trails $col \in Col(g,\tau)$. We denote the equivalence class (with respect to permutations of the $\tau(\alpha)$ and permutations of those indices $\beta$ which belong to identical matrices $\tau(\beta)$ in the partition) of the three elements induced 
\begin{equation}
endu(\sigma) := [(\tau,g,col)]
\end{equation} 
\end{lemma}
\begin{proof}
We write $\sigma$ as a product of element disjoint cycles
\begin{equation}\label{eq_cycsig}
\sigma = \prod_{\alpha = 1}^{cycn(\sigma)} cyc_\alpha
\end{equation}
where $s := cycn(\sigma)$ is the number of cycles of $\sigma$ and the cycles $cyc_\alpha$ are uniquely defined by $\sigma$ up to permutation of their indices and each cycle can be written in the form 
\begin{equation}
cyc_\alpha = (b_1(\alpha),\ldots,b_{n(\alpha)}(\alpha))
\end{equation} 
such that $\sigma(b_h(\alpha)) = b_{h + 1} (\alpha)$ for $1 \leq h < n(\alpha)$ and $\sigma(b_{n(\alpha)}(\alpha)) = b_1(\alpha)$. \\
Because the cycles are element disjoint and $\sigma(b) \neq b$ because of loopfreeness, we can for each $b \in \{1,\ldots,R+p\}$ define
\begin{equation}
endes(b):= \alpha
\end{equation}
iff $b$ is an element of $cyc_\alpha$.
Now for $1 \leq i,k \leq r+p$ we define the matrices 
\begin{equation}
\tau(\alpha)_{i,k} := \#\{b_h(\alpha): qa_{tu}(b_h(\alpha)) =  i \wedge qa_{tu}(\sigma(b_h(\alpha)) = k\ \wedge 1 \leq h \leq n(\alpha)\}
\end{equation}
With this definition $\tau(\alpha) \in MBaln(r + p)$ because of the loopfreeness condition for $\sigma$ and the cyclic nature of $cyc_\alpha$. Moreover $\tau := (\tau(1),\ldots,\tau(s))$ fulfils 
equation \eqref{eq_tausum} because of the definition of $Adj(\sigma,tu,p)$. Because of the admissibility of $\sigma$ assuming that indexes $1 \leq k \leq p$ are such that $cyc_k(R + k) \neq R + k$ we have $dias(\tau(k))_{r+n} = \delta_{k,n}$ for $1 \leq k \leq p$ and $dias(\tau(k))_{r+n} = 0$ for $k > p$ and  $1 \leq n \leq p$. Therefore equation \eqref{eq_admiss} is fulfiled. Therefore $\tau \in Partc(Adj(\sigma,tu,p),p,s)$.\\
Defining the function $g$ by 
\begin{equation}
g(b) := (endes(b),qa_{tu}(\sigma(b)))
\end{equation}
we have constructed a $g \in Sel(\tau)$. For $b$ with $endes(b) = \alpha$ the sequence of edges $(b,\sigma(b),\sigma(\sigma(b)),\ldots)$ is by construction identical to a sequence which constitutes a cyclical permutation $cyc_\alpha$. As a sequence of edges in $G_\alpha$ from the $g$-canonical representation of $\tau$ it  is a Eulerian trail on $G_{\alpha}$ which is uniquely defined up to cyclic permutation. The indices $\alpha$ given to the cycles in equation \eqref{eq_cycsig} uniquely determine those of the partition and the first component of $g$. Any of the $cycn(\sigma)!$ permutations of those indices belong to the same $\sigma$. But any of those index permuations leads to a permutation of the indices of the partition (which may or may not lead to a different partition) and definitely leads to a different selection $g$. So $\sigma$ fixes the partition up to a permutation of its elements and for a fixed partition the selection $g$ up to one of the $Sim(\tau)$ permutations of the indices $\beta$ which leave the partition unchanged. 
\end{proof}
\begin{lemma}\label{le_colsig}
Let $F \in MF(r,p,w)$ and $\tau \in Partc(F,p,s)$ and $tu = dia(F)$. Let a $g \in Sel(\tau)$ be given and let $Col(g,\tau) \neq \emptyset$. Let therefore a collection of Euler trails $col \in Col(g,\tau)$ be given. Then there is a uniquely defined $\sigma \in SF(tu,r,p)$ such that $Adj(\sigma,tu,p) = F$ and $endu(\sigma) = [(\tau,g,col)]$ 
\end{lemma}
\begin{proof}
On the $g$-canonical graphs $(G_1,\ldots,G_s)$ the Euler trails $(T_1,\ldots,T_s)$ define element disjoint sequences $(b_1(\alpha),\ldots,b_{n(\alpha)}(\alpha))$ up to cyclic permutation. Each such sequence can be interpreted as a cyclic permutation $cyc_\alpha$ in the form of lemma \ref{le_sigcol}.  We can then define 
\begin{equation}
\sigma := \prod_{\alpha = 1}^{s} cyc_\alpha
\end{equation}
Because of the admissibility condition for $\tau$  in equation \eqref{eq_admiss} the cycles in $\sigma$ fulfil (after a possible reordering of their indices) $cyc_{k}(R+ k) \neq R + k$ and because of equation \eqref{eq_tausum} $Adj(\sigma,tu,p) = F$ and therefore $\sigma \in SF(tu,r,p)$.
With the proof of lemma \ref{le_sigcol} we see that $\sigma$ by construction fulfils  $endu(\sigma) = [(\tau,g,col)]$.
\end{proof}
\begin{theorem}\label{th_weifp}
For a matrix $F \in MF(r,p,1)$ and $j \geq 1$ we define the following quantities 
\begin{equation}\label{eq_afpc}
a_{F,p,j} :=   \frac{1}{j!} \cdot 
\left( \sum_{\tau \in Partc(F,p,j)}  \varphi(\tau) \right)
\end{equation}
with
\begin{equation}\label{eq_weifptau}
\varphi(\tau) :=  Mud\left(\sum_{1 \leq \alpha \leq j} \tau(\alpha) \right) \cdot \chi(\tau)
\end{equation}
with 
\begin{equation}
\chi(\tau) := \prod_{1 \leq \alpha \leq j} \frac{Eul(\tau(\alpha))}{Mult(\tau(\alpha))}
\end{equation}
Then
\begin{equation}\label{eq_afp}
wei_F(m,p) := \sum_{j \geq 1} a_{F,p,j}\cdot m^{j -p}
\end{equation}
\end{theorem} 
\begin{proof}
According to lemmata \ref{le_sigcol} and \ref{le_colsig} we know that $\sigma \in WS(F)$ induces a uniquely defined class $endu(\sigma) = [(\tau,g,col)]$. So to determine the $j$-th coefficient of $wei_F(m,p)$ as a polynomial in $m$ we look at a given $\tau \in Partc(F,p,j)$. In \eqref{eq_selnum} we have given the number of selections $g \in Sel(\tau)$. In \eqref{eq_eulcol} the number of collections $col \in Col(g,\tau)$ for a given $\tau$ and a given $g$ are given. As we have seen in the proof of lemma \ref{le_sigcol} the $j!$ permutations of the cycles of a given $\sigma$ lead to $j!$ different $(\tau,g,col)$ but belong to the same $\sigma$. Summing over all $\sigma \in WS(F)$ is therefore equivalent to summing over all $(\tau,g,col)$  and dividing by $j!$. Putting all this together we reach \eqref{eq_afp} with the coefficients given in \eqref{eq_afpc}.
\end{proof}
\begin{corollary}
With the proof of lemma \ref{le_sigcol} can also write alternatively
\begin{equation}\label{eq_weifp}
a_{F,p,j} := \sum_{\tau \in Part(F,p,j)} \psi(\tau)
\end{equation}
where 
\begin{equation}
\psi(\tau) := \frac{\varphi(\tau)}{Sym(\tau)}
\end{equation}
\end{corollary}
\begin{remark}
It is clear that for $F \in MF(r,p,1)$  
\begin{equation}\label{eq_weigp}
wei_F(m,p) = wei_{F^{\sigma}}(m,p)
\end{equation}
for all $\sigma \in S_r(p)$ (see definition \ref{de_gr}) because the the permuation of the indices can be applied to each element of a partition.
\end{remark}
\begin{lemma}\label{le_surj}
Let $F \in MF(r,p,w)$ with $p=0$ or $p=1$. Then $WS(F)\neq \emptyset$.
\end{lemma}
\begin{proof}
For $F \in MF(r,p,w)$ with $p=0$ or $p=1$ we know that the partition $\tau = (F)$ fulfils $\tau \in Part(F,p,1)$ and therefore $Eul(\tau) \neq 0$. With these data we know that $\#Sel(\tau) \neq 0$ according to equation \eqref{eq_selnum}.  But then because of lemma \ref{le_colsig} we have proved the existence of a $\sigma \in WS(F)$. 
\end{proof}



\section{The Connection}
In this section we will bring the basic concepts of the last section together with standard perturbation theory. We will analyze the coefficients of the so called proper functions $\Gamma^{[0]}$ and $\Gamma^{[2]}$ to make the crucial connections to the asymptotic  moments of the multiple point range of the closed and the unrestricted planar random walk. 
\subsection{Perturbation Theory}
In this paper we will consider a general massive  field theory with complex fields $\Phi_{\alpha}(x)$ with ($\alpha = 1,\ldots,m$) and a Lagrangian density
\begin{equation}\label{eq_lagra1}
\mathcal{L}(x) = \mathcal{L}_{free}(x) + \mathcal{L}_{int}(x) 
\end{equation}
with 
\begin{equation}
\mathcal{L}_{free}(x) := \partial_{\mu}\Phi^{\dagger}(x)\partial^{\mu}\Phi(x) + \lambda_1\cdot\left(\Phi(x)^{\dagger}\Phi(x)\right)
\end{equation}
\begin{equation}\label{eq_lagra2}
\mathcal{L}_{int}(x) := \sum_{w > 1} \lambda_w \cdot \left(\Phi^{\dagger}(x)\Phi(x)\right)^{2w}
\end{equation}
(where $\lambda_1 \in \C\setminus\{0\}$ and for $w > 1$ at least one but at most finitely many of the $\lambda_w \in \C$ are nonzero) in $d$ dimensions. Perturbation theory in the $r^{th}$ order is mathematically well defined in terms of Feynman integrals and weight factors. To motivate them 
let us define 
\begin{equation}
\ell_{int,w}(x) := 
\left(\Phi^{\dagger}(x)\Phi(x)\right)^{2w}
\end{equation}
Then for $r \geq 0$ the $r$th expansion term of the standard exponential of field theory  (see e.g. \citep[ch. 6]{itzykson}, \citep[ch. 2 - 5]{phi4}) is given by 
\begin{equation}
Exp_r := \frac{(-1)^r \cdot \prod_{i=1}^{r} \mathcal{L}_{int}(x_i)}{r!}
\end{equation}
Using definition \eqref{de_tup} we then notice by simple algebra
\begin{equation}
Exp_r := \frac{1}{r!}\sum_{(d_1,\ldots,d_r) \in Tup(r,2)} \left(\prod_{1 \leq k \leq r}  \left(-\lambda_{d_k} \right)\right) \cdot\left(\prod_{1 \leq k \leq r} \ell_{int,d_k}(x_k) \right)
\end{equation} 
In perturbation theory the following objects (called $r^{th}$ order contribution to the $2p$-point Greensfunctions) are studied: 
\begin{multline}
G^{[2p],pt}_{r}(y_1,z_1\ldots,y_{p},z_p)_{\beta_1,\gamma_1,\ldots,\beta_p,\gamma_p} := \\ 
\left\langle\left\langle  \Phi^{*}_{\beta_1}\left(y_1 \right) \cdot \Phi_{\gamma_1}\left(z_1\right) \cdot  \ldots \cdot \Phi^{*}_{\beta_{p}}\left(y_p\right)\cdot \Phi_{\gamma_{p}}\left(z_{p}\right) \cdot Exp_r \right\rangle_{WI} \right\rangle_{I}
\end{multline} 
where $\left\langle \right\rangle_{WI}$ stands for the Wick/Isserli contractions using the Gaussian distribution of the free field theory represented by the $\mathcal{L}_{free}(x)$ and $\left\langle \right\rangle_{I}$ stands for appropriate integrations over a subset of the variables $x_1,\ldots ,x_r$: For each connection component one variable in it is set to $0$ and we integrate the others over $\R^d$.\\ Because of the form of the free Lagrangian we know that $G^{[2p],pt}_{r}$ has a specific behaviour under permutatios $\pi \in S_p$. 
\begin{multline}\label{eq_inddelta}
G^{[2p],pt}_{r}(y_1,z_1\ldots,y_{p},z_p)_{\beta_1,\gamma_1,\ldots,\beta_p,\gamma_p} = \\ G^{[2p],pt,par}_{r}(y_1,z_1\ldots,y_{p},z_p)\cdot \left( \sum_{\pi \in S_p} \prod_{i=1}^p\delta_{\beta_i,\gamma_{\pi(i)}}\right)
\end{multline}
For general $m$ it therefore suffices to consider the case where the following conditions are met: 
\begin{equation}\label{eq_bg}
\beta_j = \gamma_j
\end{equation}
\begin{equation}\label{eq_bb}
 \beta_j \neq \beta_k \Leftarrow j \neq k
\end{equation}
for all $j,k=1,\ldots,p$  in the sequel.
For a given tupel $tu := (d_1,\ldots,d_r) \in Tup(r,2)$ with $Sum(tu) = R$ let us define 
$\alpha_{R+k} := \beta_k$ and $x_{r+k} := y_k$ and $x_{r+p+k} := z_k$ for $k=1,\ldots,p$.
We now note that we can write 
\begin{multline}\label{eq_phialp}
 \Phi^{*}_{\beta_1}\left(y_1 \right) \cdot \Phi_{\beta_1}\left(z_1\right) \cdot  \ldots \cdot \Phi^{*}_{\beta_{p}}\left(y_{p}\right)\cdot \Phi_{\beta_{p}}\left(z_p\right) \cdot \prod_{1 \leq k \leq r} 
\ell_{int,d_k}(x_k) = \\
\sum_{1 \leq \alpha_{q} \leq m} \left(\prod_{1 \leq j \leq R + p} \Phi^{*}_{\alpha_j}(x_{qa_{tu}(j)})\Phi_{\alpha_j}(x_{qb_{tu}(j)}) \right)
\end{multline}
where the variables $\alpha_q$ for $1 \leq q \leq R$ over which the right hand side of \eqref{eq_phialp} is summed belong to the pairs $\Phi^{*}_{\alpha_q}(x_q)\cdot \Phi_{\alpha_q}(x_q)$ in $\prod_{1 \leq k \leq r} 
\ell_{int,d_k}(x_k)$. We have used the functions $qa_{tu}$ and $qb_{tu}$ given in definition \ref{de_qaqb}.
So we have an expression which is a sum over monomials in $\Phi^{*}\Phi$ of degree $R + p$. The Wick/Isserli contractions of these  monomials can now be expressed in terms of permutations  $\sigma \in S_{R + p}$ \cite[eq. 1.22]{zinn_justin}: 
\begin{multline}\label{eq_wi}
\left\langle \prod_{1 \leq j \leq R + p} \Phi^{*}_{\alpha_j}(x_{qa_{tu}(j)})\Phi_{\alpha_j}(x_{qb_{tu}(j)}) \right\rangle_{WI} = \\ \sum_{\sigma \in S_{R+p}} \prod_{1 \leq j \leq R + p} \left\langle \Phi^{*}_{\alpha_j}(x_{qa_{tu}(j)}) \Phi_{\alpha_{\sigma(j)}}(x_{qb_{tu}(\sigma(j))}) \right\rangle_{WI}
\end{multline}
where $\left\langle \right\rangle_{WI}$ stands for the Wick/Isserli contractions. So we have reduced the higher moments to sums of products of second moments. For the second moments we know from e.g. \cite[eq. 25]{psp} that for $x_i \neq x_j$
\begin{multline}\label{eq_propp}
\left\langle \Phi(x_i)_\alpha \Phi^{*}(x_j)_\beta \right\rangle_{WI} = \delta_{\alpha,\beta} \cdot Gre(d,\left| x_i - x_j\right|) :=\\ 
\delta_{\alpha,\beta}\cdot \left( \frac{\lambda_1}{u^2}\right)^{\frac{d-2}{4}} \cdot \frac{1}{(2\pi)^{\frac{d}{2}}} \cdot K_{\frac{d-2}{2}}\left( \sqrt{\lambda_1\cdot u^2}\right)\Biggr|_{u = \left| x_i - x_j\right|}
\end{multline}
where $K_{\nu}$ here denotes the modified Bessel function as in \cite[ch.9.6]{danraf}.\\
In this article we are only interested in those permutations $\sigma$ which are admissible and lead to connected loopfree multigraphs because they turn out to be the building blocks of quantum field theory.
Let us discuss shortly the reason for the above restriction:\\
Loopfreeness (no tadpoles) is a widespread demand in quantum field theory which is somewhat ad hoc \citep[p. 271]{itzykson} and related to a change of the parameter $\lambda_1$ which physicists call mass renormalization. In this article in conjunction with \citep[section 7.3]{dhoefc} a mathematical reason is given for the exclusion of tadpole diagrams at this level. Graphs with tadpoles in the evaluation of the moments of the distribution of the multiple point range  are related to factors of the first moment in higher noncentralized moments. Their subtraction (which mathematicians call renormalization) to get to the centralized moments leads to the factors related to the Borel transform to be discussed in section 4 and 5 of this article. This crucial information is lost in standard quantum field theory. The standard Borel transform used to resum the perturbation expansion comes from a general reasoning about the growth rate of the perturbation series (see e.g. \citep[ch. 9.4.2]{itzykson}). This article puts together those two loose ends.\\   
Connected and unconnected graphs on the other hand are related to each other by exponentiation \citep[section 3.3]{phi4}.\\ 
Admissibility is related to the factor $\delta_{\alpha,\beta}$ in the propagators \eqref{eq_propp} and the conditions in equations  \eqref{eq_bg} and \eqref{eq_bb}. Because of these equations the permutations $\sigma$ contributing to  $G^{[2p],pt,par}_{r,c}$ must induce $p$ seperate edge trails on $G(\sigma)$ each one containing a different outer vertex $r+k$ for $k=1,\ldots,p$.\\ With definition \ref{de_adcon} 
the above three conditions --loopfreeness, connectedness and admissibility-- are equivalent to $\sigma \in SF(tu,r,p)$  \\
\begin{definition}
We therefore for $p \geq 0$ define the $r^{th}$ contribution to the so called connected (subscript $c$) $2p$ Greensfunction
for $r=0$ by
\begin{equation}
G^{[2],pt,par}_{0,c}(x_i,x_j) = Gre\left(d,\left|x_i - x_j\right|\right)
\end{equation}
and  for $r > 0$ by
\begin{multline}\label{eq_g2px}
G^{[2p],pt,par}_{r,c}(y_1,z_1,\ldots,y_p,z_p) := \\ 
\frac{1}{r!}\Biggl(\sum_{tu \in Tup(r,2)}\Biggl(\sum_{\sigma \in SF(tu,r,p)}\Biggl(\sum_{1 \leq \alpha_q \leq m} 
 \\
\Pi_{tu}(\lambda) \cdot \left\langle \prod_{1 \leq j \leq R + p} \left\langle \Phi^{*}_{\alpha_j}(x_{qa_{tu}(j)}) \Phi_{\alpha_{\sigma(j)}}(x_{qb_{tu}(\sigma(j))}) \right\rangle_{WI}\right\rangle_{I}\Biggr)\Biggr)\Biggr)
\end{multline} 
where $R = Sum(tu)$ and the integration $\langle \rangle_I$ is over the internal space variables $x_1,\ldots,x_R$ for $p > 0$ and over $x_2,\ldots,x_R$ and $x_1 :=0$ for $p = 0$  and the summation over $\alpha_q$ is for all $1\leq q \leq R$. For $p=0$ of course no variables $y_i,z_i$ exist. 
\end{definition}
Now for each $p>0$ the function  $G^{[2p],pt,par}_{r,c}(P,\lambda)$ is defined as the Fourier transform of
\begin{equation}
G^{[2p],pt,par}_{r,c}(y_1,z_1\ldots,y_{p},z_p)
\end{equation} 
where the so called \lq\lq external momenta\rq\rq $P$ denotes collectively the variables conjugated to $y_1,z_1,\ldots,y_p,z_p$.  
In this article we will only work with $P=0$ i.e. all \lq\lq external momenta\rq\rq\ are set to $0$ and so the Fourier integral reduces to an integral over the external \lq\lq space\rq\rq\ variables. In the sequel we will generally drop the dependency on $P$.
\begin{definition}
For $r > 0$ and $p>0$ we can now define 
the $r^{th}$ order contribution to the connected $2p$-point function with zero external momenta $P$ by
\begin{multline} \label{eq_gcst}
G^{[2p],pt,par}_{r,c}(P = 0,\lambda) := \\ \int dz_1 \int \prod_{k=2}^{p} d^dy_k\cdot d^dz_k \cdot G^{[2p],pt,par}_{r,c}(y_1,z_1\ldots,y_{p},z_p)\Biggr|_{y_1 = 0}
\end{multline}
For $r>0$ and $p=0$ we define $G_{r,c}^{[0],pt,par}(\lambda)$ to be the left hand side of equation \eqref{eq_g2px} for $p=0$. For  $r = 0$ and $p=0$ we define $G_{0,c}^{[0],pt,par}(\lambda) = 0$
\end{definition}
To evaluate $G^{[2p],pt,par}_{r,c}(\lambda)$ we define: 
\begin{definition}
Let $r > 0$ and $tu \in Tup(r,2)$ and $\sigma \in SF(tu,r,p)$ and $cycn(\sigma)$ denote the number of its cycles. Then we define the so called Feynman contribution $FC$ of this permutation $\sigma$ 
\begin{multline}
FC(\sigma,tu,p,\lambda) := \\ \Pi_{tu}(\lambda) \cdot m^{cycn(\sigma) - p} \cdot \int \prod_{k=2}^{r + 2p} d^dx_k  \prod_{1 \leq i \neq j \leq r + 2p} Gre(d,\left|x_i - x_j\right|)^{Adjo(\sigma,tu,p)_{i,j}}\Biggr|_{x_1=0}
\end{multline}
\end{definition}

As is well known from the so called Feynman parametrization \cite[ch.8]{itzykson} 
\begin{multline}\label{eq_igcst}
\int \prod_{k=2}^{r + 2p} d^dx_k  \prod_{1 \leq i,j \leq r + 2p} Gre(d,\left|x_i - x_j\right|)^{Adjo(\sigma,tu,p)_{i,j}}\Biggr|_{x_1=0} = \\ FM_{Pr(G)}(\lambda_1) \cdot I_{Pr(G)}(d)
\end{multline}
(with $I_{Pr(G)}(d)$ from definition \ref{de_feyg}) where the graph $G:= G(\sigma)$ was defined in \ref{de_gsigma}. We have used that $Adjo(\sigma,tu,p) = Adj(Pr(G(\sigma)))$. So we can write:
\begin{equation}
FC(\sigma,tu,p,\lambda) := \\ \Pi_{tu}(\lambda) \cdot wei_{\sigma}(m,p) \cdot FM_{Pr(G(\sigma))}(\lambda_1) \cdot I_{Pr(G(\sigma))}(d)
\end{equation}
Realizing that for a given $\sigma \in SF(tu,r,p)$ it is exactly the $cycn(\sigma) - p $ cycles of $\sigma$ which do not contain an $R + k$ for $1 \leq k \leq p$ in it which contribute a factor $m$ when the summation over the $\alpha_q$ in equation \eqref{eq_g2px} is done we find
\begin{equation}\label{eq_g2psig}
G^{[2p],pt,par}_{r,c}(\lambda) = 
  \frac{1}{r!}\left(\sum_{tu \in Tup(r,2)} \left(
\sum_{\sigma \in SF(tu,r,p)} FC(\sigma,tu,p,\lambda) \right) \right)
\end{equation}
With theorem \ref{th_weifp} we have found a way of writing the $r^{th}$ order contribution of the $2p$ point function in terms of Feyman contributions summed over matrices $F \in MF(r,p,2)$. 

\begin{lemma}
For $r > 0$ 
\begin{equation}
G^{[2p],pt,par}_{r,c}(\lambda) = \frac{1}{r!} \sum_{F \in MF(r,p,2)} FC(F,\lambda)
\end{equation}
with the so called Feynman contribution 
\begin{equation}
FC(F,\lambda) := wei_F(m,p) \cdot \Pi_F(\lambda) \cdot  FM_{Pr(G(F))}(\lambda_1)\cdot I_{Pr(G(F))}(d)
\end{equation}
\end{lemma}
\begin{proof}
The lemma follows from equations \eqref{eq_g2psig} and \eqref{eq_igcst}.
\end{proof}
\begin{corollary}
For $r > 0$
\begin{equation}\label{eq_g2d}
G^{[2p],pt,par}_r(\lambda) = \frac{1}{r!}\sum_{G \in GF(r,p,2)} FC(G,\lambda)
\end{equation}
where now the Feynman contribution of the graph $G$ is given by
\begin{equation}\label{fc_graph}
FC(G,\lambda) := gr(G) \cdot wei_G(m,p) \cdot \Pi_G(\lambda)  \cdot FM_{Pr(G)}(\lambda_1) \cdot I_{Pr(G)}(d)
\end{equation}
which is the standard way of writing the contribution which field theorists in physics use. The definitions of so called symmetry factors for graphs and weight polynomials vary in the literature. In our approach we have used a nomenclatura for $gr(G)$ and $wei_G(m,p)$ which is motivated by the two stages of refinement of graphs by vertex and edge numbering. It is the suitable choice to make a point of connection to random walk results.
\end{corollary}
\begin{definition}
For $r > 0$ and $G \in GF(r,p,w)$ we also define the direct Feynman contribution where the integral now is over the graph $G$ itself
\begin{equation}\label{dfc_graph}
FCD(G,\lambda) := gr(G) \cdot wei_G(m,p) \cdot \Pi_G(\lambda)  \cdot FM_G(\lambda_1) \cdot I_{G}(d)
\end{equation}
and the so called truncated Feynman contribution 
\begin{equation}
FCT(G,\lambda) := \lambda_1^{2p} \cdot FC(G,\lambda)
\end{equation}
(assuming, as always in this article, that all so called \lq\lq external momenta\rq\rq\ $P$ are zero).
\end{definition}
For the later subsections we also define
\begin{definition}
For $r > 0$ 
\begin{equation}\label{eq_g2uc}
G^{[2],pt,par}_{r,c}(\lambda) =  G^{[C]}_r(\lambda) +  G^{[U]}_r(\lambda)
\end{equation}
where 
\begin{equation}\label{eq_gf2c}
G^{[C]}_r(\lambda) = \frac{1}{r!}\sum_{G \in GFC(r,1,2)} FC(G,\lambda)
\end{equation}
and 
\begin{equation}
G^{[U]}_r(\lambda) = \frac{1}{r!}\sum_{G \in GFU(r,1,2)} FC(G,\lambda)
\end{equation}
Note: For $r = 0$ 
\begin{equation}
G^{[2],pt,par}_{0,c}(\lambda) =  \frac{1}{\lambda_1} 
\end{equation}
\end{definition}
We also introduce $\Gamma^{[2]}$, the proper two point function and prove a major relationship of it to $G^{[2]}$. 
\begin{definition}\label{de_1pi}
By $GFCP(r,1,w)$ we denote the subset of $GFC(r,1,w)$ such that $G \in GFCP(r,1,w) \Rightarrow EC(Amp(G)) \geq 2$, i.e. the one particle irreducible (1PI) graphs in $GFC(r,1,w)$.
\end{definition}
\begin{remark}
We already know that $G \in GFU(r,1,w) \Rightarrow EC(Amp(G)) \geq 2$ because $Amp(G)$ in this case is Eulerian.
\end{remark}
With definition \ref{de_1pi} and definition \ref{de_psfpef} we can now formulate the two point proper vertex function.
Equivalent to \cite[eq. 4.33]{phi4} we define:
\begin{definition}
\begin{equation}\label{eq_gamma2d}
\Gamma^{[2],pt}_r (P =0,\lambda) = \lambda_1\cdot \delta_{r,0} -  \Sigma_r(P=0,\lambda)
\end{equation}
(in the sequel we will drop the dependency on P). For $r=0$ we define $\Sigma_0(\lambda) := 0$. For $r > 0$ the so called selfenergy
\begin{equation}
\Sigma_r(\lambda) := \Sigma^{[U]}_{r} (\lambda) + \Sigma^{[C]}_r(\lambda)
\end{equation}
with the contribution from the unclosable (and therefore 1PI) graphs
\begin{equation}\label{eq_sigmae}
\Sigma^{[U]}_{r} (\lambda) := \frac{1}{r!} \left( \sum_{G \in GFU(r,1,2)} FCT(G,\lambda) \right) = \lambda_1^2 \cdot G^{[U],pt,par}_r(\lambda)
\end{equation}
and of the closable 1PI graphs
\begin{equation}\label{eq_sigmas}
\Sigma^{[C]}_{r} (\lambda) := \frac{1}{r!} \left(\sum_{G \in GFCP(r,1,2)} FCT(G,\lambda) \right)
\end{equation}
\end{definition}
We will now derive the main relationship between $G^{[2],pt,par}_{r,c}$ and $\Gamma^{[2],pt}_r$. We define the formal series
\begin{equation}
\Gamma^{[2],pt} (\lambda) := \sum_{r \geq 0} \Gamma^{[2],pt}_r (\lambda)
\end{equation}
We now start with
\begin{lemma}\label{le_decopi}
Let $G \in GF(r,1,2)$. 
Then the graph $Pr(G)$ has a uniquely defined decomposition into bridges (isthmusses) $b_0,\ldots,b_{M}$  and subgraphs $G_1,\ldots,G_M$ with $EC(G_i) \geq 2$. The sequence of the $G_i$ and $b_i$ can be chosen such that any Eulerian trail on $Pr(G)$ enters the subgraphs $G_i$ in the order of their indices $i$ going into $G_i$ over the isthmus $b_{i-1}$ and out of $G_i$ over the isthmus $b_i$ for $i=1,\ldots,M$.  In addition let $H_i$ be the subgraph of $Pr(G)$ consisting of $G_i$ and its ingoing edge $b_{i-1}$ and its outgoing edge $b_i$ and the vertex sets $V_{o,1}:= \{st(b_{i-1})\}$ and $V_{o,2} := \{end(b_i)\}$.  Then $Ret(H_i) \in GFU(k_i,1,2)$ or $Ret(H_i) \in GFCP(k_i,1,2)$ for some $k_i$. Additionally
\begin{equation}
r = \sum_{i = 1}^M k_i
\end{equation}
On the other hand: let there be any sequence $(K_1,\ldots,K_M)$ of graphs $K_i \in GFU(k_i,1,2)$ or $K_i \in GFCP(k_i,1,2)$ and $G_i := Amp(K_i)$. Concatenating the $G_i$ over bridges $b_{i-1},b_i$ compatible with the entry/exit structure of $Pr(K_i)$ uniquely results in a graph $H \in PGF(r,1,2)$ up to isomorphism. $H$ additionally has the feature $H \in PGFC(r,1,2)$ unless $M=1$ and $K_1 \in GFU(k_1,1,2)$.\\ 
Let $G = Ret(H)$. Then  
\begin{equation}\label{eq_grpr}
gr(G) = \frac{r!}{k_1!\cdot \ldots \cdot k_M!} \cdot \frac{1}{symg(K_1,\ldots,K_M)}\prod_{i=1}^M gr\left(K_i\right)
\end{equation}
where $symg(K_1,\ldots,K_M)$ is the product of the factorials of the number of isomorphic copies in each group of isomorphic graphs in $(K_1,\ldots,K_M)$.  
\end{lemma}
\begin{proof}
Let $Br(G) \subset E(Pr(G))$ be the set of bridges of $Pr(G)$ including $ank(G)$ and $bnk(G)$. Let $\left(G_i \right)_{i = 1,\ldots,M}$ be the connectivity components of $Pr(G) \setminus Br(G)$ other than the isolated vertices $st(G)$ and $end(G)$. Let $Ec$ be a Eulerian circuit on $G$ starting in $st(G)$. After entering any $G_i$ the circuit $Ec$ must first go through all edges of $G_i$ before entering another bridge, because it cannot return to $G_i$ after crossing a bridge and so the $G_i$ are Eulerian or semi Eulerian and have edge connectivity $\geq 2$. As all inner vertices of $Pr(G)$ have an indegree (and therefore outdegree) greater than one the $G_i$ cannot consist of a single vertex. So we can order the connectivity components and the bridges $b_i \in Br(G)$ such that the circuit $Ec$ after starting in $b_0:=ank(G)$ consecutivly enters $G_i$ and then $b_i$ in the order of the index $i = 1,\ldots, M$ where $b_M:=bnk(G)$. This order of the connectivity components and bridges does not depend on the specific Eulerian circuit $Ec$, as there is only the bridge $b_i$ between $G_i$ and $G_{i+1}$ for $i=1,\ldots,M-1$ and any Euler trail must start in $b_0$ and end in $b_M$. \\
On the other hand: \\
Any concatenation of bridges $(b_0,\ldots,b_{M})$ and graphs $(G_1,\ldots,G_M)$ in the form of the above lemma obviously form a graph $H$ which is Eulerian. $H$ has $r = \sum_{i=1}^M k_i$ inner vertices which all by construction have an indegree greater than one and so $Ret(H) \in GF(r,0,2)$. By construction it is clear that $H \in PGFC(r,1,2)$ unless $M=1$ and $K_1 \in GFU(k_1,1,2)$.\\  
The graph sequence $(G_1,\ldots,G_M)$ can be uniquely extracted from $Pr(G)$ by the construction in the first part of the proof. So they uniquely determine $Pr(G)$ up to graph isomorphism and to exchanging isomorphic copies of the $K_i$.  Because of the free choice of the vertices formula \eqref{eq_grpr} is true.
\end{proof}
\begin{lemma} As is well known in quantum field theory (e.g. \citep[eq. 6.78]{itzykson} or \citep[eq. 5.75]{phi4} with slightly different definitions/notations) for $r \geq 0$
\begin{equation}\label{eq_gr2ga}
G_{r,c}^{[2],pt,par}(\lambda) = \left[\frac{1}{\Gamma^{[2],pt}(\lambda)}\right]_r
\end{equation}
where $[\ldots]_r$ for $r \geq 0$ denotes the sum of the terms with a product of $r$ variables $\lambda_w$ with $w > 1$ in the formal geometric expansion of $1/\Gamma^{[2],pt}(\lambda)$.
\end{lemma}
\begin{proof}
The case $r=0$ is trivial. For $r>0$: If $G \in GF(r,1,2)$ then $Pr(G)$ has a uniquely defined decomposition as shown in lemma \ref{le_decopi} with parts $H_i \in GF(k_i,1,2)$ and $G_i = Amp(H_i)$. 
Now 
\begin{equation}
I_{Pr(G)}(d) = \prod_{i=1}^M I_{Pr(H_i)}(d)
\end{equation}
and 
\begin{equation}
FM_{Pr(G)}(\lambda_1) = \left(\frac{1}{\lambda_1} \right)^{M+1} \prod_{i=1}^M \left( FM_{Pr(H_i)}(\lambda_1) \cdot \lambda_1^2 \right)
\end{equation}
and 
\begin{equation}
wei_G(m,1) = \prod_{i=1}^M wei_{H_i}(m,1)
\end{equation}
because partitions, selections and Euler trails on $G$ are uniquely related to those on $Ret(H_i)$ and vice versa. Using lemma \ref{le_decopi} and especially equation \eqref{eq_grpr} we therefore reach equation \eqref{eq_gr2ga}.
\end{proof}

\subsection{Connection for $\Gamma^{[0]}$}
To connect $G^{[0],pt,par}_{r,c}$ with the results on closed random walks we prove the following
\begin{lemma}\label{le_balphr}
Let $\lambda_k = 0 \Leftarrow k > 2$ and 
let $F \in MF(r,0,2)$. Then either $F \in \tilde{H}_r(2,\ldots,2)$ (defined in \citep[Theorem 1.1]{dhoefc}) or $\Pi_F(\lambda) = 0$. 
\end{lemma}
\begin{proof}
From the definition it is clear that $ \tilde{H}_r(2,\ldots,2) \subset MF(r,0,2)$. So let $F \in MF(r,0,2)$ but $F \notin \tilde{H}_r(2,\ldots,2)$. Then there is a $k$ such that $dias_k(F) > 2$ and therefore $\Pi_F(\lambda) = 0$.
\end{proof}
From field theory we know 
\begin{equation}
\Gamma^{[0],pt}_r (\lambda) = G^{[0],pt,par}_{r,c} (\lambda) = \frac{1}{r!} \left( \sum_{F \in MF(r,0,2)} FC(F,\lambda) \right)
\end{equation}
Now according to lemma \ref{le_balphr} for $\Phi^4$ theory this means 
\begin{equation}\label{eq_gamphi}
\Gamma^{[0],pt}_{\Phi^4,r}(\lambda) = \frac{1}{r!} \left( \sum_{F \in \tilde{H}_r(2,\ldots,2)} FC(F,\lambda)  \right)
\end{equation}
So writing in the terms of \cite{dhoefc} we find
\begin{equation}
\Gamma^{[0],pt}(\lambda) = \lambda_1^{d/2}  \sum_{r=2}^{\infty}   \left(-\lambda_1^{d/2 - 2}\cdot \lambda_2 \right)^r \cdot \Gamma^{[0],pt}_{\phi_4,r}(m)
\end{equation}
with 
\begin{equation}\label{eq_gam0phi4}
\Gamma^{[0],pt}_{\phi_4,r}(m):=\frac{1}{r!}\left( \sum_{F \in H_r(2,\ldots,2)} wei_F(m,0) \cdot \frac{\Gamma_{G(F)}\left(1 - \frac{d}{2}\right)}{(4\pi)^{(r+1)\cdot d/2}} \right) 
\end{equation}
and
\begin{equation}\label{eq_phi4wei}
wei_F(m,0) := \sum_{j \geq 1} a_{F,0,j}\cdot m^j
\end{equation}
and with theorem \ref{th_weifp}
\begin{equation}\label{eq_af10}
a_{F,0,1} = 2^r \cdot cof(A - F) \cdot \frac{1}{\prod_{i,j} F_{i,j}!}
\end{equation}
and therefore $\Gamma_{r,0}^{[0],pt} = 0$ (where we now have dropped the subscript $\phi^4$) and with the notation from \cite[eq. 3.10]{dhoefm} we get for $d = 2$ 
\begin{equation}\label{eq_tgamma0}
\Gamma_{r,1}^{[0],pt} =  \frac{1}{r! \cdot 8^r}  \cdot \left( \sum_{F \in H_r(2,\ldots,2)}  cof(A - F) \cdot \frac{1}{\prod_{i,j} F_{i,j}!} \cdot I(F) \right) 
\end{equation} 

Equation \eqref{eq_tgamma0} is therefore a suitable equation for the point of contact to \cite[eq. 8.57]{dhoefc} for the closed random walk. \\

\subsection{Connection for $G^{[C]}$}
We have already seen that the $Div$ operation leads to a  relationship between $GF(r,0,2)$ and $GFC(r,1,2)$. So we start with a deeper study of the $Div$ operation given in definition \ref{de_divcl} on the level of permutations. 
\begin{definition}
 Let $\sigma \in SF(tu,r,1)$ with $R = Sum(tu)$. Then we define $bnk(\sigma) : = \sigma^{-1}(R + 1)$ and $ank(\sigma) := \sigma(R + 1)$. We say that $\sigma$ is closeable iff $qa_{tu}(bnk(\sigma)) \neq qa_{tu}(ank(\sigma))$ and unclosable otherwise. We denote with $SFC(tu,r,1)$ and with $SFU(tu,r,1)$  the set of closeable and uncloseable permutations in $SF(tu,r,1)$ respectively.
\end{definition}  
\begin{definition}
For $\sigma \in SFC(tu,r,1)$ we can then define 
 \begin{equation}
  Cl(\sigma) = \bar{\sigma} \in S_R
 \end{equation}  
 by 
 \begin{equation}
 \bar{\sigma}(bnk(\sigma)) = ank(\sigma)
 \end{equation}
 and 
 \begin{equation}
 \bar{\sigma}(a) = \sigma(a) \Leftarrow a \neq bnk(\sigma)
 \end{equation}
 Then obviously $\bar{\sigma} \in SF(tu,r,0)$ and $Adj(\bar{\sigma},tu,0) = Cl(Adj(\sigma,tu,1))$. 
\end{definition}
\begin{definition}
On the other hand if we have $\pi \in SF(tu,r,0)$ and $1 \leq a \leq R$ with $R = Sum(tu)$ we can define:
\begin{equation}
Div(a,\pi) = \pi_a \in S_{R + 1}
\end{equation}
by 
\begin{equation}
 \pi_a(a) := R + 1
 \end{equation}
 and 
 \begin{equation}
 \pi_a(R + 1) := \pi(a)
 \end{equation}
 and 
 \begin{equation}
 \pi_a(b) := \pi(b) \Leftarrow b \neq a
 \end{equation}
 Then obviously $Div(a,\pi) \in SF(tu,r,1)$ and there is an $e \in E(G(Adj(\pi,tu,0)))$ with $st(e) = qa_{tu}(a)$ and $end(e) = qa_{tu}(\pi(a))$ such that 
 \begin{equation}\label{eq_divea}
 G(Adj(Div(a,\pi),tu,1)) = Div(e,G(Adj(\pi,tu,0)))
 \end{equation}
 and 
  \begin{equation}\label{eq_diveam}
 Adj(Div(a,\pi),tu,1) = Div((qa_{tu}(a),qa_{tu}(\pi(a))),Adj(\pi,tu,0))
 \end{equation}
 For any $\pi \in SF(tu,r,0)$ and any $1 \leq a \leq Sum(tu)$ we have 
 \begin{equation}
 Cl(Div(a,\pi)) = \pi
 \end{equation}
 independent of $a$ and for any $\sigma \in SFC(tu,r,1)$ we have
 \begin{equation}
 Div(bnk(\sigma),Cl(\sigma)) = \sigma
 \end{equation} 
\end{definition}
\begin{lemma}\label{le_pidiv}
Let $F \in MF(r,0,w)$. 
There is a one to one relationship $Div$ between pairs $(a,\pi)$ with $\sigma \in WS(F)$ and $1 \leq a \leq Sum(F)$ and $\rho \in WS(H)$ with $H \in DM(F)$. 
\end{lemma}
\begin{proof}
For the pair $(a,\pi)$ with $\pi \in WS(F)$ and $1 \leq a \leq Sum(F)$  we have already defined $Div(a,\pi)$, which fulfils $Div(a,\pi) \in WS(H)$ with $H \in DM(F)$. Let on the other hand a $\rho \in WS(H)$ with $H \in DM(F)$ be given. Then the pair $WP(\rho) :=(bnk(\rho),Cl(\rho))$ fulfils $Cl(\rho) \in WS(F)$ and $1 \leq bnk(\rho)\leq Sum(Cl(H))$ and we have $Div(WP(\rho)) = \rho$ and $WP(Div(a,\pi)) = (a,\pi)$.
\end{proof}
\begin{definition}\label{de_deltag}
Let $G \in GF(r,p,w)$. We define 
\begin{equation}
WS(G) := \{\rho \in SF(dia(F),r,p): F = Adj(G) \}
\end{equation}
where $Adj(G)$ can be taken with any mapping $map_G \in BVA(G)$.
\end{definition}
\begin{lemma}\label{le_psfer1}
Let $G \in GF(r,0,w)$. There is a one to one relationship $Div$ between pairs $(a,\pi)$ with $\pi \in WS(G)$ and $1 \leq a \leq Sum(G)$ and permutations $\rho \in WS(H)$ with $H \in DG(G)$. \\
Moreover
\begin{equation}
GFC(r,1,w) = \dot{\cup}_{G \in GF(r,0,w)} DG(G)
\end{equation} 
\end{lemma}
\begin{proof}
Let $G \in GF(r,0,w)$ and $\pi \in WS(G)$ and $1 \leq a \leq Sum(G)$. Let $tu = dia(Adj(G))$. In the graph $G(\sigma)$ there is exactly one edge $e \in E(G(\sigma))$ which corresponds to $a$ which means $G(Div(a,\pi)) = Div(e,G)$ and therefore $Div(a,\pi) \in WS(H)$ with $H \in DG(G)$.\\
 Let on the other hand $\rho \in WS(H)$ with a $H \in DG(G)$ be given. Then $WP(\rho) := (bnk(\rho), Cl(\rho))$ has the properties $G(Cl(\rho)) = Cl(H) = G$ and therefore $Cl(\rho) \in WS(G)$. Moreover $1 \leq bnk(\rho) \leq Sum(G)$ and $WP(Div(a,\pi)) = (a,\pi)$ and $Div(WP(\rho)) = \rho$.\\
Moreover for any $H \in GFC(r,1,w)$ we know $H \in DG(Cl(H))$. Additionally for two graphs 
$G_1, G_2 \in GF(r,0,w)$ which are not isomorphic to each other we know that 
$DG(G_1) \cap DG(G_2) = \emptyset$ because $Cl$ is a well defined function.  
\end{proof}
\begin{lemma}\label{le_divsfc}
Let $G \in GF(r,0,w)$. Then 
\begin{equation}\label{eq_divsfc}
m \cdot \sum_{H \in DG(G)} FCD(H,\lambda) =   \left( - \frac{\partial}{\partial \lambda_1}\right) FC(G,\lambda) 
\end{equation}
\end{lemma}\label{le_divfc}
\begin{proof} 
For pairs $(a,\pi)$ of the form of lemma \ref{le_psfer1} and $H := G(Div(a,\pi))$ and $tu = dia(Adj(G))$ and $R=Sum(G)$ we note 
\begin{equation}
m \cdot FCD(Div(a,\pi),tu,1,\lambda) = m^{cycn(\pi)} \cdot \Pi_{tu}(\lambda) \cdot FM_{H}(\lambda_1) \cdot I_{H}(d)
\end{equation}
Summing this equation over $1 \leq a \leq R$ we get 
\begin{multline}
m \cdot \sum_{a = 1}^R FCD(Div(a,\pi),tu,1,\lambda) = \\ m^{cycn(\pi)} \cdot \Pi_{tu}(\lambda) \cdot FM_{Div(e,G)}(\lambda_1) \cdot \sum_{e \in E(G)} I_{Div(e,G)}(d)
\end{multline}
and therefore with equation \ref{eq_difint}
\begin{multline}
m \cdot \sum_{a = 1}^R FCD(Div(a,\pi),tu,1,\lambda) =\\ m^{cycn(\pi)} \cdot \Pi_{tu}(\lambda) \cdot  \left( - \frac{\partial}{\partial \lambda_1}\right) \left( FM_G(\lambda_1) \cdot I_G(d)  \right)
\end{multline}
Summing over $\pi \in WS(G)$ we finally get
\begin{equation}
m \cdot \sum_{\pi \in WS(G)} \left( \sum_{a = 1}^R FCD(Div(a,\pi),tu,1,\lambda) \right) =  \left( - \frac{\partial}{\partial \lambda_1}\right) \left(FC(G,\lambda) \right)
\end{equation}
On the other hand because of lemma \ref{le_psfer1} we can write 
\begin{multline}
m \cdot \sum_{\pi \in WS(G)} \left( \sum_{a = 1}^R FCD(Div(a,\pi),tu,1,\lambda) \right) = \\ m \cdot
\sum_{H \in DG(G)} \left( \sum_{\rho \in WS(H)} FCD(\rho,tu,1,\lambda) \right)
\end{multline}
But now we see
\begin{equation}
m \cdot \sum_{H \in DG(G)} \left( \sum_{\rho \in WS(H)} FCD(\rho,tu,1,\lambda) \right) = m \cdot \sum_{H \in DG(G)} FCD(H,\lambda)
\end{equation}
which proves the lemma.
\end{proof}
\begin{lemma}\label{le_sumfe01}
Let
\begin{align*}
f:& GFC(r,1,w) & \mapsto &  & \C \\
   & H & \mapsto & & f(H) \\
\end{align*}
be a function. 
Then 
\begin{multline}\label{eq_sumfe01}
 \sum_{G \in GF(r,0,w)} gr(G) \cdot \Pi_G(\lambda) \cdot wei_G(m,0) \left(\sum_{e \in E(G)} f\left(Div(e,G))\right) \right) = \\ m \cdot \left(\sum_{H \in GFC(r,1,w)} gr(H)\cdot \Pi_H(\lambda) \cdot wei_H(m,1) \cdot f(H) \right)
\end{multline}
\end{lemma}
\begin{proof}
We first observe that the sums on both sides are finite because of the factors $\Pi_G(\lambda)$ and $\Pi_H(\lambda)$.\\
For $\pi \in WS(G)$ and $1 \leq a \leq Sum(G)$ and $\rho = Div(a,\pi)$ we can with $tu = dia(Adj(G))$ write:
\begin{equation}\label{eq_rhopi}
m^{cycn(\pi)} \cdot \Pi_{tu}(\lambda) \cdot f(G(Div(a,\pi))) = 
m^{cycn(\rho)} \Pi_{tu}(\lambda) \cdot f(G(\rho))
\end{equation}
and therefore summing equation \eqref{eq_rhopi} over $a$ and $\pi \in WS(G)$ we get
\begin{multline}
\sum_{\pi \in WS(G)} m^{cycn(\pi)} \cdot \Pi_{tu}(\lambda) \cdot \sum_{a=1}^{Sum(G)} f(G(Div(a,\pi))) = \\ gr(G) \cdot \Pi_{G}(\lambda)  \cdot wei_G(m,0) \cdot \sum_{e \in E(G)} f(Div(e,G))
\end{multline} 
Because of the one to one relationship between pairs $(a,\pi)$ with $\pi \in WS(G)$ and $1 \leq a \leq Sum(G)$ and permutations $\rho \in WS(H)$ with $H \in DG(G)$ we can sum the right hand side of equation \eqref{eq_rhopi} over $\rho \in WS(H)$ and then over $H \in DG(G)$ and get 
\begin{equation}
m \cdot \left( \sum_{H \in DG(G)} gr(H) \cdot \Pi_H(\lambda) \cdot wei_H(m,1) \cdot f(H) \right)
\end{equation}
But now equation \eqref{eq_sumfe01} follows from equation \eqref{eq_bigcupw}.
\end{proof}

\begin{lemma}
\begin{multline}\label{eq_delta2}
m \cdot G^{[C]}_r(\lambda) = \\ \frac{1}{r!}  \sum_{G \in GF(r,0,2)} gr(G) \cdot \Pi_G(\lambda) \cdot wei_G(m,0)\cdot \\ \left(\sum_{e \in E(G)} FM_{Pr(Div(e,G))}(\lambda_1)  \cdot I_{Pr(Div(e,G))}(d) \right)
\end{multline}
\end{lemma}
\begin{proof}
By putting equation \eqref{eq_gf2c} into lemma \ref{le_sumfe01} we get equation \eqref{eq_delta2}.
\end{proof}

So for  $\phi^4$-theory we get with lemma \ref{le_balphr}
\begin{equation}
G^{[C]}(\lambda) = \lambda_1^{-1} \cdot \left( \sum_{r=2}^{\infty} \left(-\lambda_1^{d/2 - 2}\cdot \lambda_2 \right)^r G^{[C],ob}_r(m) \right)
\end{equation}
with
\begin{equation}
m \cdot G^{[C],ob}_r(m) = \frac{1}{r!} 
 \left( \sum_{F \in H_r(2,\ldots,2)} wei_F(m,0) \cdot \frac{ \left(\sum_{e \in E(G)}\Gamma_{G(F) \setminus e}\left(1 - \frac{d}{2} \right)\right)}{(4\pi)^{r\cdot d/2}} \right) 
\end{equation}
Writing 
\begin{equation}
G^{[C],ob}_r(m) = \sum_{j \geq 0} G^{[C],ob}_{r,j} \cdot m^j
\end{equation}
using \cite[eq 3.11]{dhoefm} and equations \eqref{eq_phi4wei} and \eqref{eq_af10} we get for $d = 2$
\begin{equation}\label{eq_tdelta2}
G^{[C],ob}_{r,0} = 4 \cdot \frac{1}{r! \cdot 8^r} \left( \sum_{F \in H_r(2,\ldots,2)}  cof(A - F) \cdot \frac{1}{\prod_{i,j} F_{i,j}!} \cdot \mathcal{I}(F) \right) 
\end{equation}
Equation \eqref{eq_tdelta2} is therefore the crucial connection between $G^{[C]}$ and moments of the multiple point range of the unrestricted planar random walk.

\subsection{Connection for $G^{[U]}$}

We now further analyze $GFU(r,1,2)$. 
\begin{definition}
Let $G \in Gr(q)$. We denote the set of vertices $v$ with $ind(v) = outd(v) = k$ with $E_k(G)$
\end{definition}
\begin{definition} Let $G \in Gr(q)$ with a vertex $v \in E_2(G)$. Let $e \in E(G)$ with $st(e) = v$. By definition there is then exactly one edge $\tilde{e} \in E(G)$ with $e \neq \tilde{e}$ and $st(\tilde{e}) = v$. We denote this complementarian edge with $co(e,v)$. For an edge $ed \in E(G)$ with $end(ed) = v$ we also denote the uniquely defined edge $\widetilde{ed} \in E(G) $ with $\widetilde{ed} \neq ed$ and $end(\widetilde{ed}) = v$ with $co(ed,v)$ 
\end{definition}
\begin{lemma}\label{le_ladder}
Let $G \in GFU(r,1,2)$. Then there is an integer $w \geq 1$ and a set of vertices of $G$ called $V_{ladder} = \{v^{[0]}, \ldots, v^{[w]} \}  $ with $v^{[0]} = out(G)$
 and $v^{[i]}) \in E_2(G)$ for $1 \leq i < w$ and a set of edges of $G$ called 
\begin{equation}
E_{ladder} := \{e_a^{[0]},e_b^{[0]}, \ldots,
e_a^{[w-1]}, e_b^{[w-1]} \}
\end{equation}
 such that $e_a^{[0]} =  ank(G)$ and $e_b^{[0]} = bnk(G)$
and
$st\left(e_a^{[j]}\right) = end\left(e_b^{[j]}\right) = v^{[j]}$
for $j=0,\ldots,w-1$ and $end\left(e_a^{[j]}\right) = st\left(e_b^{[j]}\right) = v^{[j+1]}$ for $j=0,\ldots,w-1$ and either \\
case A: \\
 $v^{[w]} \in E_2(G)$
but  $st\left(co\left(e_a^{[w-1])},v^{[w]}\right)\right) \neq end\left(co\left(e_b^{[w-1])},v^{[w]}\right)\right)$ \\or case B: \\ 
$v^{[w]} \in E_q(G)$ with $q >2$ . \\In case B we define $vn(G):= q$.\\ We denote the uniquely defined number $w$ by $lad(G) := w$ and call it the length of the ladder.  
\end{lemma}
\begin{proof}
For $k=0$ $v^{[0]} := out(G)$ and $e_a^{[0]} := ank(G)$ and $e_b^{[0]} := bnk(G)$ and $v^{[1]} := end\left(ank(G)\right) = st\left(bnk(G)\right)$ are well defined. So for a given $k \geq 1$ let there be vertices $v^{[0]},\ldots,v^{[k]}$ all in $E_2(G)$ and edge pairs $e_a^{[0]},e_b^{[0]},\ldots,e_a^{[k-1]},e_b^{[k-1]}$such that $st\left(e_a^{[j]}\right) = end\left(e_b^{[j]}\right) = v^{[j]}$ for $j=0,\ldots,k-1$ and  $end\left(e_a^{[j]}\right) = st\left(e_b^{[j]}\right) = v^{[j+1]}$ for $j = 0,\ldots,k-1$. 
If $st\left(co(e_a^{[k-1]},v^{[k]})\right) \neq end\left(co(e_b^{[k-1]},v^{[k]})\right)$ then $w := k$ and we have a case A. 
Otherwise we define  $v^{[k + 1]} := st\left(co(e_a^{[k-1]},v^{[k]})\right)$. If $v^{[k + 1]} \notin E_2(G)$ we then have $w = k + 1$ and a case B. Otherwise we continue the process for $k + 1$ instead of $k$ setting $e_a^{[k]} := co(e_b^{[k-1]},v^{[k]})$ and $e_b^{[k]} := co(e_a^{[k-1]},v^{[k]})$. As there are only finitely many vertices in G the process will stop eventually. 
 \end{proof}
\begin{definition}
By $GFUA(r,1,2)$ we denote the subset of $GFU(r,1,2)$ which is of case A, by $GFUB(r,1,2)$ the same for case B. Then lemma \ref{le_ladder} means that 
\begin{equation}
GFU(r,1,2) = GFUA(r,1,2) \dot\cup GFUB(r,1,2)
\end{equation}
\end{definition}
\begin{definition}
Let $K \in GFC(k,1,2)$ and $w > 0$ a natural number. We define the graph $L:=GFA(K,w)$ to be such that 
\begin{equation}
V(L) := V(K) \dot{\cup} \{v^{[0]},\ldots,v^{[w-1]}\}
\end{equation} 
and define additionally $v^{[w]} := out(K)$.
\begin{equation}
E(L) := E(K) \dot{\cup} \{e_a^{[0]},e_b^{[0]},\ldots,e_a^{[w-1]},e_b^{[w-1]} \}
\end{equation}
with $st(e_a^{[k]}) = v^{[k]} = end(e_b^{[k]})$ and $end(e_a^{[k]}) = v^{[k+1]} = st(e_b^{[k]})$ for $0 \leq k \leq w-1$. 
All other relationships for $e \in E(K)$ are supposed to be inherited into $E(L)$.
Additionally we set $V_o(L) := \{v^{[0]}\}$. With this definition $GFA(K) \in GF(k+w,1,2)$.
\end{definition}
\begin{definition}
Let $K \in GF(k,0,2)$ and $w > 0$ a natural number and $v \in V(K)$ a vertex. We define the graph $L:=GFB(K,w,v)$ to be such that 
\begin{equation}
V(L) := V(K) \dot{\cup} \{v^{[0]},\ldots,v^{[w-1]}\}
\end{equation} 
and define additionally $v^{[w]} := v$.
\begin{equation}
E(L) := E(K) \dot{\cup} \{e_a^{[0]},e_b^{[0]},\ldots,e_a^{[w-1]},e_b^{[w-1]} \}
\end{equation}
with $st(e_a^{[k]}) = v^{[k]} = end(e_b^{[k]})$ and $end(e_a^{[k]}) = v^{[k+1]} = st(e_b^{[k]})$ for $0 \leq k \leq w$. 
All other relationships for $e \in E(K)$ are supposed to be inherited into $E(L)$.
Additionally we set $V_o(L) := \{v^{[0]}\}$. With this definition $GFB(K,w,v) \in GF(k+w-1,1,2)$.
\end{definition}
\begin{definition}
Let $G \in GFU(r,1,2)$. We define the graph $sld(G)$ as the subgraph of $G$ which does not contain
 $out(G)$ nor $V_{ladder}\setminus\{v^{[w]}\}$ in its
  vertex set and does not contain
  $E_{ladder}(G)$ in its edgeset. In case A we also define the vertex $v^{[lad(G)]}$ which in $sld(G)$ has indegree $1$ to be the only outer vertex of $sld(G)$. With these definitions in case A $sld(G) \in GFC(r-w,1,2)$. In case B $sld(G) \in GF(r-w+1,0,2)$. 
\end{definition}  
\begin{lemma}\label{le_pea}
For $G \in GFUA(r,1,2)$ we have 
\begin{equation}
G = GFA(sld(G),lad(G))
\end{equation}
up to isomorphy of graphs. For any graph $H \in GFC(k,1,2)$ and any $w > 0$ we have $GFA(H,w) \in GFUA(k+w,1,2)$ and 
\begin{equation}
sld(GFA(H,w)) = H
\end{equation}
and 
\begin{equation}
lad(GFA(H,w)) = w
\end{equation}
So there is a one to one relationship between graphs $G \in GFUA(r,1,2)$ and pairs $(H,w)$ with $w >0$ and $H \in GFC(r-w,1,2)$
\end{lemma}

\begin{proof}
Let $G \in GFUA(r,1,2)$. As all vertices in $V_{ladder}\setminus\{v^{[lad(G)]}\}$ are cutvertices the connectivity of $sld(G)$ follows from the one of $G$. Moreover by construction $sld(G)$ is a balanced graph and therefore Eulerian. All inner vertices of $sld(G)$ have degree $\geq 2$ and $sld(G)$ is by construction closeable. Therefore $sld(G) \in GFC(r-lad(G),1,2)$.  The procedure $GFA(sld(G),lad(G))$ by construction now results in a graph isomorphic to $G$. \\ Let on the other hand a graph $H \in GFC(k,1,2)$ and a number $w \geq 1$ be given. Now by construction $GFA(H,w)$ is not closeable, i.e. $GFA(H,w) \in GFU(r+w,1,2)$. The procedure in the proof of lemma \ref{le_ladder} obviously stops with $H =  sld(GFA(H,w))$ and so the lemma is proven. 
\end{proof}
\begin{lemma}\label{le_peb}
For $G \in GFUB(r,1,2)$ we have
\begin{equation}
G = GFB(sld(G),lad(G),v^{[lad(G)]})
\end{equation}
up to isomorphy of graphs. For any graph $H \in GFC(k,0,2)$ and any $w > 0$ and any vertex $v \in V(H)$ with $ind(v)>2$ we have $GFB(H,w,v) \in GFUB(k+w-1,1,2)$ and 
\begin{equation}
sld(GFB(H,w,v)) = H
\end{equation}
and 
\begin{equation}
lad(GFB(H,w,v)) = w
\end{equation}
and 
\begin{equation}
v^{[lad(GFB(H,w,v)]} = v
\end{equation}
So there is a one to one relationship between graphs $G \in GFUB(r,1,2)$ and triples $(H,w,v)$ with $w >0$ and $H \in GF(r-w+1,0,2)$ and $v \in V(H)$. 
\end{lemma}
\begin{proof}
Let $G \in GFUB(r,1,2)$. As all vertices in $V_{ladder}\setminus\{v^{[lad(G)]}\}$ are cutvertices the connectivity of $sld(G)$ follows from the one of $G$. Moreover by construction $sld(G)$ is a balanced graph and therefore Eulerian. All inner vertices have degree $\geq 2$.  Therefore $sld(G) \in GF(r-lad(G)-1,0,2)$.  The procedure $GFB(sld(G),lad(G),v^{[lad(G)]})$ by construction now results in a graph isomorphic to $G$. \\ Let on the other hand a graph $H \in GF(k,0,2)$ and a number $w \geq 1$ and a vertex $v \in V(H)$ with $ind(v) > 2$ be given. Now by construction $GFB(H,w,v)$ is not closeable, i.e. $GFB(H,w,v) \in GFU(k+w-1,1,2)$. The procedure in the proof of lemma \ref{le_ladder} obviously stops with $H =  sld(GFB(H,w,v))$ and so the lemma is proven. 
\end{proof}
\begin{lemma}\label{le_grab}
For $G \in GFUA(r,1,2)$ and $H = sld(G)$ we have 
\begin{equation}\label{eq_graba}
gr(G) = \frac{r!}{(r - lad(G))!}\cdot gr(H)
\end{equation}
For $G \in GFUB(r,1,2)$ and $H = sld(G)$ we have 
\begin{equation}\label{eq_grabb}
gr(G) = \frac{r!}{(r - lad(G) + 1)!}\cdot gr(H)
\end{equation}
\end{lemma}
\begin{proof}
For a given $G \in GFUA(r,1,2)$ there are 
\begin{equation}
{{r}\choose{lad(G)}} \cdot lad(G)!
\end{equation}
ways to choose the numbering of the $lad(G)$ inner vertices $V_{ladder}\setminus\{v^{[0]}\}$. The factor $lad(G)!$ comes from the fact that the vertices in  $V_{ladder}\setminus\{v^{[0]}\}$ are distiguishable as they form a sequence in the ladder. This already fixes the outer vertex of $sld(G)$ which is the vertex $v^{[w]}$. We are left with the $(r - lad(G))$ inner vertices of $sld(G)$ which can still be chosen freely to give an additional factor $gr(sld(G))$. \\
For a given $G \in GFUB(r,1,2)$ there are 
\begin{equation}
{{r}\choose{lad(G)-1}} \cdot (lad(G) - 1)!
\end{equation}
ways to choose the numbering of the $lad(G)-1$ inner vertices $V_{ladder}\setminus\{v^{[0]},v^{[w]}\}$. The factor $(lad(G) - 1)!$ comes from the fact that the vertices in  $V_{ladder}\setminus\{v^{[0]},v^{[w]}\}$ are distiguishable as they form a sequence in the ladder. We are left with the $(r - lad(G) + 1)$ vertices of $sld(G)$ which can still be chosen freely to give an additional factor $gr(sld(G))$. \\
\end{proof}

\begin{lemma}\label{le_decomp}
Let $G \in GF(r,1,1)$ and $v$ be a cutvertex of $G$. Then there is a unique decomposition of graphs $G_1$ and $G_2$ such that $G_1 \cup G_2 = G$ and $G_1 \cap G_2 = \{v\}$ such that $G_1$ contains $out(G)$ and $G_1 \in GF(r_1,1,1)$ (with $out(G)$ its outer vertex) and $G_2 \in GF(r_2,0,1)$ for integer $r_1,r_2 > 0$ such that $r_1 + r_2 = r + 1$. 
\end{lemma}
\begin{proof}
For $G \in GF(r,1,1)$ we know that there  is an Euler path $Ep$ on $G$ starting and ending in $out(G)$. Then we can define $G_1$ as the union of the vertex $v$ together with the connectivity component of $G \setminus\{v\}$ which contains the outer vertex of $G$. We define $G_2$ as the union of the other connectivity components together with $v$. Then by construction $G = G_1 \cup G_2$ and $G_1 \cap G_2 = \{v\}$. $Ep$ now starts in $out(G)$ and so the first edge of $Ep$ must be in $E(G_1)$. $Ep$ can later on only leave $G_1$ or enter $G_1$ over $v$ because $v$ is a cutvertex. So we can define $Ep_1$ as the sequence of those edges in $Ep$ which are in $G_1$ in the order of $Ep$. $Ep_1$ then is an Euler trail from $out(G)$ to $out(G)$, $G_1$ Eulerian and therefore $G_1 \in GF(r_1,1,1)$ with $r_1 > 0$. As $v$ is a cutvertex $V(G_2)$ must have more than one element. We define $Ep_2$ as the sequence of $Ep$ which is in $G_2$ in the same order as in $Ep$. It then is a sequence of edges going from $v$ to $v$ because $v$ is a cutvertex and it is Eulerian and therefore $G_2$ is Eulerian. So $G_2 \in GF(r_2,0,1)$ with $r_2 > 0$ and $r_1 + r_2 = r + 1$. 
\end{proof}
\begin{corollary}
By the same reasoning: \\
Let $G \in GF(r,0,1)$ and $v$ be a cutvertex of $G$. Then there is a decomposition of graphs $G_1$ and $G_2$ such that $G_1 \cup G_2 = G$ and $G_1 \cap G_2 = \{v\}$  and $G_1 \in GF(r_1,0,1)$ and $G_2 \in GF(r_2,0,1)$ for integer $r_1,r_2 > 0$ such that $r_1 + r_2 = r + 1$. The decomposition is unique up to a permutation of $G_1$ with $G_2$.
\end{corollary}
\begin{lemma}
Let $G \in GF(r,1,2)$ and $v$ be a cutvertex of $G$ and $G_1$, $G_2$ the unique decomposition of $G$ as proven in lemma \ref{le_decomp}. Let the indegree (and therefore the outdegree) of $v$ in $G_1$ be $1$ and the indegree (and therefore the outdegree) of $v$ in $G_2$ be $j$. Then 
\begin{equation}\label{eq_weipr}
m \cdot wei_G(m,1) = (j + 1)\cdot (j + m)\cdot wei_{G_1}(m,1) \cdot wei_{G_2}(m,0)
\end{equation}
\end{lemma}
\begin{proof}
Because $G_1 \cap G_2 = \{v\}$ we can choose a mapping $map_G \in BVA(G)$ of the inner vertices of $G$ such that $F := Adj(G)$ can be written such that $F = F_1 + F_2$ with $G_1 = GK(F_1)$ and $G_2 = GK(F_2)$ and $(F_1)_{i,j} = 0$ if $i > d(v)$ or $j>d(v)$ and $(F_2)_{i,j} = 0$ if $i < d(v)$ or $j < d(v)$ where $d(v) := map_G(v)$. We will in the sequel work with this mapping which is block diagonal up to the $d(v)^{th}$ row and column of $F$.\\
We will now in a first step look at partitions $\tau \in Partc(F,1,s)$ with $\tau = (\tau(1),\ldots,\tau(s))$. 
Because of lemma \ref{le_coneul} we assume without restriction for this proof that every $GK(\tau(\alpha)))$ corresponds to a connected subgraph of $G$ for $1 \leq \alpha \leq s$. Otherwise the $\tau$ to which $\tau(\alpha)$ belongs does not contribute to $wei_G(m,1)$. If $dias(\tau(\alpha))_{d(v)} = 0$ 
  then its matrix elements must either be nonzero only where the ones of $F_1$ are nonzero or nonzero only where the ones of $F_2$ are nonzero as $v$ is a cutvertex. 
  If $dias(\tau(\alpha))_{d(v)} \neq 0$ the graph $GK(\tau(\alpha))$ on the other hand can either be completely in $G_1$ or completely in $G_2$ or have a decomposition according to lemma \ref{le_decomp} and/or its corollary. If a decomposition exists, we can denote it by $GK(\tau(\alpha)_1), GK(\tau(\alpha)_2)$ with $GK(\tau(\alpha)_1)$ being a connected subgraph of $G_1$ and $GK(\tau(\alpha)_2)$ being a connected subgraph of $G_2$ and therefore corresponding to matrices $\tau(\alpha)_1$ and $\tau(\alpha)_2$. We call such a $\tau(\alpha)$ overlapping. Because the indegree of $v$ in $G_1$ is $1$ there can at most be one such overlapping $\tau(\alpha)$ in $\tau$ and therefore we can define that $\tau$ is called overlapping iff it contains an overlapping $\tau(\alpha)$. We denote by $Partco(F,s,1)$ the overlapping partitions of $F$ and with $Partcn(F,s,1)$ the nonoverlapping partitions. 
\begin{equation}\label{eq_partsum}
Partc(F,s,1) = Partco(F,s,1) \dot{\cup} Partcn(F,s,1)
\end{equation}  
   Let us define $T_1(\tau)$ to be the set of $\alpha$ such that $GK(\tau(\alpha))$ is completely in $G_1$ and $T_2(\tau)$ respectively for $G_2$ and the set $T_o(\tau)$ to be the set of $\alpha$ where $\tau(\alpha)$ is overlapping.\\
For a given $\tau$ we now define $\tau^{[1]} = (\tau^{[1]}(1),\ldots,\tau^{[1]}(s_1))$ to be the subsequence of $\tau$ of those $\tau(\alpha)$ which have either $\alpha \in T_1(\tau)$ or $\tau(\alpha)_1$ if $\alpha \in T_o(\tau)$ in the order of $\tau$ and $\tau^{[2]}$ accordingly for $\alpha \in T_2(\tau)$ or $\tau(\alpha)_2$. These subsequences can be restricted to the dimensions $\leq d(v)$ outside of which the matrix elements of $F_1$ vanish for $\tau^{[1]}$ and to those dimensions $\geq d(v)$ outside of which the matrix elements of$F_2$ vanish for $\tau^{[2]}$ respectively as their matrix elements are also $0$ outside of those dimensions. We call these restrictions $\widetilde{\tau}^{[1]}$ and $\widetilde{\tau}^{[2]}$  From the construction it is then clear that $\widetilde{\tau}^{[1]} \in Partc(komp(F_1),s_1,1)$ and $\widetilde{\tau}^{[2]} \in Partc(komp(F_2),s_2,0)$. \\
On the other hand let us assume we have any $\widetilde{\rho}^{[1]} \in Partc(komp(F_1),s_1,1)$ and $\widetilde{\rho}^{[2]} \in Partc(komp(F_2),s_2,0)$ and their extensions $\rho^{[1]}$ and $\rho^{[2]}$ on the dimensions of $F$. We denote by $T_{1,o}$ the set of $\beta$ such that $GK(\rho^{[1]}(\beta))$ contains $v$ and $T_{2,o}$ respectively. We denote by $\beta_o$ the one element in $T_{1,o}$ and choose any $\gamma \in T_{2,o}$. We can then define
\begin{equation}
\rho_{\gamma,o} := \rho^{[1]}(\beta_o) + \rho^{[2]}(\gamma)
\end{equation}
Then obviously $GK(\rho_{\gamma,o})$ is a connected subgraph of $G$. We can therefore define a $\rho_{\gamma}$ by putting the sequence $\rho^{[1]}$ without $\rho^{[1]}(\beta_o)$ and $\rho^{[2]}$ without $\rho^{[2]}(\gamma)$ next to each other and adding $\rho_{\gamma,o}$ at the end of the sequence. Obviously any such $\rho_{\gamma} \in Partco(F,s_1+s_2-1,1)$. 
In addition we define $\rho_{sep}$ to be the concatenation of $\rho^{[1]}$ and $\rho^{[2]}$ as sequences. We then know $\rho_{sep} \in Partcn(F,s_1+s_2,1)$.
Different $\rho^{[2]}(\gamma)$ obviously lead to different $\rho_{\gamma}$ up to identical copies of $\rho^{[2]}(\gamma)$ in $\rho^{[2]}$. 
$\rho_{sep}$ obviously is different from any $\rho_{\gamma}$ \\
So because of the decomposition of a $\tau \in Partc(F,s,1)$ proven before up to a reordering of the elements in the sequences any $\rho \in Partc(F,s,1)$ can be uniquely composed from $\rho^{[1]}, \rho^{[2]}$ in the form of $\rho_{sep}$ for nonoverlapping partitions. For overlapping partitions $\rho_{\gamma}$ for those $\gamma \in T_{2,o}$ which belong to identical matrices in the sequence $\rho^{[2]}$ we know that they lead to the same $\rho_{\gamma}$. If we denote the respective number of identical matrices with $m_{\gamma}$ then there is a $1:m_{\gamma}$ relationship in this case \\
We note 
\begin{equation}
Mud(F) = Mud(F_1) \cdot Mud(F_2) \cdot (j+1)
\end{equation}
because the vertex $v$ had an indegree of $j$ in $G_2$ and of $j+1$ in $G$. All other factors in equation \eqref{eq_weifp} are the same for $\rho_{sep}$ on the one hand and the product of the factors for $\rho^{[1]}$ and $\rho^{[2]}$ on the other hand. So 
\begin{equation}\label{eq_psisep}
\psi(\rho_{sep}) = (j+1) \cdot \psi(\rho^{[1]}) \cdot \psi(\rho^{[2]})
\end{equation} 
(where $\psi$ is defined in equation \eqref{eq_weifptau}).
So summing the right hand side of equation \eqref{eq_psisep} over $\rho^{[1]} \in Part(F_1,s_1,1)$ and $\rho^{[2]} \in Part(F_2,s_2,1)$ leads to 
\begin{equation}\label{eq_sumsep}
\sum_{\rho \in Partn(F,1,s_1 + s_2)} \psi(\rho) = (j+1) \cdot a_{F_1,1,s_1} \cdot a_{F_2,1,s_2}
\end{equation}
where $Partn$ denote the equivalence classes of nonoverlapping admissible partitions. 
For $\rho_{\gamma}$ we get 
\begin{equation}\label{eq_psinov}
Sym(\rho_{\gamma}) \cdot m_{\gamma} = Sym(\rho^{[1]}) \cdot Sym(\rho^{[2]})
\end{equation}
where $m_{\gamma}$ as before denotes the multiplicity of the identical copies of $\rho^{[2]}(\gamma)$ in $\rho^{[2]}$. This immediately follows from the fact that $\rho_{\gamma}$ has only $m_{\gamma} - 1$ copies of $\rho^{[2]}(\gamma)$ in it.   
We also get 
\begin{equation}
Co(\rho_{\gamma,o}) = Co(\rho^{[1]}(\beta_o)) \cdot Co(\rho^{[2]}(\gamma))
\end{equation}
which is easily seen if one blots out the matrix row and column which corresponds to $v$ in the matrices on both sides. \\
Therefore 
\begin{equation}
Eul(\rho_{\gamma}) = Eul(\rho^{[1]}(\beta_o)) \cdot Eul(\rho^{[2]}(\gamma)) \cdot j_{\gamma}
\end{equation}
where $j_{\gamma}$ is the indegree of $GK(\rho^{[2]}_{\gamma})$ in $v$. 
Again all other factors in equation \eqref{eq_weifp} are the same for $\rho_{\gamma}$ on the one hand and the product of the factors for $\rho^{[1]}$ and $\rho^{[2]}$ on the other hand.\\
Therefore 
\begin{equation}
\frac{\chi(\rho_{\gamma})}{Sym(\rho_{\gamma})} = j_{\gamma} \cdot m_{\gamma} \cdot \frac{\chi(\rho^{[1]})}{Sym(\rho^{[1]})} \cdot  \frac{\chi(\rho^{[2]})}{Sym(\rho^{[2]})}
\end{equation}
For a given pair $\rho^{[1]},\rho^{[2]}$ to get the relative factor for all possible overlapping concatenations we have to sum over the concatenations in the form of $\rho_{\gamma}$ over all $\gamma \in T_{2,o}$ dividing by a factor $m_{\gamma}$ because of the $1:m_{\gamma}$ relationship between concatenations of $\rho^{[1]},\rho{[2]}$ to $\rho_{\gamma}$. We therefore get  
\begin{equation}
\sum_{\gamma \in T_{2,o}} \frac{\psi(\rho_{\gamma})}{m_{\gamma}} = (j + 1) \cdot  \left( \sum_{\gamma \in T_{2,o}} j_{\gamma} \right) \cdot \psi(\rho^{[1]}) \cdot \psi(\rho^{[2]})
\end{equation}
Now for any $\rho^{[2]}$ the sum of the number of the indegrees $j_{\gamma}$ over alle $\gamma \in T_{2,o}$ must yield $j$ by construction so we get 
\begin{equation}\label{eq_psiov}
 \sum_{\gamma \in T_{2,o}} \frac{\psi(\rho_{\gamma})}{m_{\gamma}} = (j + 1) \cdot j \cdot \psi(\rho^{[1]}) \cdot \psi(\rho^{[2]})
\end{equation}
So fortunately the relative factor $j\cdot (j+1)$ is independant of our decomposition. Summing \eqref{eq_psiov}  over equivalence classes of all overlapping admissible partitions $Parto$ on the left hand side is equivalent to summing over all equivalence classes of admissible partitions of $F_1$ and $F_2$ respectively on the right hand side. We also have to take into account the $1:m_{\gamma}$ relationship and therefore get. 
\begin{equation}\label{eq_sumov}
\sum_{\rho \in Parto(F,s_1 + s_2 - 1,1)} \psi(\rho) = j\cdot (j+1) \cdot a_{F_1,1,s_1} \cdot a_{F_2,1,s_2}
\end{equation}
An additional factor $m$ has to be taken into account to correctly count the contributions to $wei$ in the nonoverlapping case. Putting equations \eqref{eq_sumov} and \eqref{eq_sumsep} together with \eqref{eq_partsum} we then reach equation \eqref{eq_weipr}.
\end{proof}

\begin{definition}
Let $k> 0$ and $G \in GFC(k,0,2)$. For a given number $w \geq 1 $ we define the set $GFU_A(w,G)$ to be the set of all graphs $H \in GFUA(k + w,1,2)$ such that $G=Cl(sld(H))$. By the same token we define the set $GFU_B(w,G)$ to be the set of all graphs $\tilde{H} \in GFUB(k + w-1,1,2)$ such that $sld(\tilde{H}) = G$.
\end{definition}
\begin{lemma}\label{le_msum}
For $k>0$ and $G \in GF(k,0,2)$ and a given natural number $w \geq 1$ the following formula are true:
\begin{multline}\label{eq_masum}
m \cdot \frac{1}{(k+w)!}\sum_{H \in GFU_A(w,G)} FCT(H,\lambda) = \\ \frac{1}{k!} \cdot  \left(2\cdot (1 + m)(\lambda_1^{\frac{d}{2} - 2}\cdot (-\lambda_2))\cdot \frac{\Gamma\left(2 - \frac{d}{2}\right)}{(4\pi)^{d/2}}\right)^{w-1} \cdot \\
 2 \cdot (1 + m) \cdot (-\lambda_2) \cdot \left(-\frac{\partial}{\partial \lambda_1} \right)\left(FC(G,\lambda) \right)
\end{multline}
and 
\begin{multline}\label{eq_mbsum}
m \cdot \frac{1}{(k+w-1)!} \sum_{\tilde{H} \in GFU_B(w,G)} FCT(\tilde{H},\lambda) = \\ \frac{1}{k!}   \left(2\cdot (1 + m)(\lambda_1^{\frac{d}{2} -2}\cdot (-\lambda_2)) \cdot \frac{\Gamma\left(2 - \frac{d}{2}\right)}{(4\pi)^{d/2}}\right)^{w-1} \cdot \\
 \sum_{j=2}^{\infty}(j + 1) \cdot (j + m) \cdot \lambda_{j+1} \cdot\frac{\partial} {\partial \lambda_j}\left( FC(G,\lambda)\right)
\end{multline}
\end{lemma}
\begin{proof}
We will look at the contribution $FCT(H,\lambda) $ and $FCT(\tilde{H},\lambda) $ according to its different factors in \eqref{fc_graph}.
Let us start with $H \in GFU_A(w,G)$.
We have 
\begin{equation}
\Pi_{H}(\lambda) = \left(-\lambda_2\right)^{w} \cdot \Pi_{sld(H)}(\lambda)
\end{equation}
and 
\begin{equation}
\#V(Pr(H)) = \#V(sld(H)) + w + 1
\end{equation}
and
\begin{equation}
\#E(Pr(H)) = \#E(sld(H)) + 2\cdot w
\end{equation}
and 
\begin{equation}
L(Pr(H)) = L(sld(H)) + w - 1 
\end{equation}
and so get
\begin{equation}
FM_{Pr(H)}(\lambda_1) = \lambda_1^{-2} \cdot  \lambda_1^{\left(\frac{d}{2} -2 \right) \cdot (w - 1) } FM_{sld(H)}(\lambda_1)
\end{equation}
between $H$ and $sld(H)$. 
Using equation \eqref{eq_weipr} for each step of the ladder we get:
\begin{equation}
wei_H(m,1) := \left(2\cdot (1 + m)\right)^{w-1}  \cdot wei_{sld(H)}(m,1)
\end{equation}
(because each step of the ladder corresponds to the $2 \times 2$ matrix 
\begin{equation}
\omega := \begin{pmatrix} 0 & 1 \\ 1 & 0 \end{pmatrix}
\end{equation}
and $wei_{\omega}(m,0) = m$).\\ 
For the integral
$I_{Pr(H)}(d)$ we note that any spanning tree of $H$ consists of a spanning tree of $sld(H)$ and exactly one edge at each step of the ladder and all those combination exist exactly once. So the Kirchhoff-Symanzik polynomial \eqref{eq_poly} inside the integral $I_{Pr(H)}(d)$ as given in equation \eqref{eq_intig} admits separation into integration over the variables of the edges of each step of the ladder which leads to 
\begin{equation}
I_{Pr(H)}(d) = I_{step}(d)^{w-1} \cdot I_{sld(H)}(d)  
\end{equation}
where 
\begin{equation}
I_{step} (d) :=  \frac{1}{(4\pi)^{\frac{d}{2}}}\int_0^{\infty} d\alpha_1 \int_0^{\infty} d\alpha_2 \cdot \frac{ exp(-\alpha_1 - \alpha_2)} {\left(\alpha_1 + \alpha_2\right)^{\frac{d}{2}} } \\ = 
\frac{\Gamma(2-\frac{d}{2})}{(4\pi)^{\frac{d}{2}}}
\end{equation}
From lemma \ref{le_grab} we know:
\begin{equation}
gr(H) = \frac{(k+w)!}{k!} \cdot gr(sld(H))
\end{equation} 
Putting all factors together and summing over all possibilities $sld(H) \in DG(G)$ (see definition \ref{de_deltag})  we use lemma \ref{le_divsfc} to reach equation \eqref{eq_masum} (compare also e.g.  \citep[eq. 6.94]{itzykson}) \\
Let us continue with $\tilde{H} \in GFU_B(w,G)$.
Using the same arguments as for the integral for $H$ we can immediately write: 
\begin{equation}
I_{Pr(\tilde{H})}(d) = I_{step}(d)^{w-1} \cdot I_{G}(d)  
\end{equation}
We also find 
\begin{equation}
\#E(Pr(\tilde{H})) = \#E(G) + 2\cdot w
\end{equation}
and 
\begin{equation}
L(Pr(\tilde{H})) = L(G) + w -1
\end{equation}
and therefore 
\begin{equation}
FM_{Pr(\tilde{H})}(\lambda_1) = \lambda_1^{-2} \cdot  \lambda_1^{\left(\frac{d}{2} -2 \right) \cdot (w - 1) } FM_{G}(\lambda_1)
\end{equation}
between $\tilde{H}$ and $G = sld(\tilde{H})$. 
From lemma \ref{le_grab} we know 
\begin{equation}
gr(\tilde{H}) = \frac{(k+w-1)!}{k!} \cdot gr(sld(\tilde{H}))
\end{equation}
For the next steps we have to choose a specific vertex $v \in V(G)$ to be $v^{[w]} \in V(\tilde{H})$ with $vn(\tilde{H}) = j$.  
If such a vertex exists we have 
\begin{equation}
m \cdot wei_{\tilde{H}}(m,1) := \left(2\cdot (1 + m)\right)^{w-1}  \cdot (j+1)\cdot(j+m) \cdot wei_G(m,0)
\end{equation}
and 
\begin{equation}
\lambda_j \cdot \Pi_{\tilde{H}}(\lambda) = \left(-\lambda_2\right)^{w -1} \cdot\lambda_{j+1} \cdot \Pi_G(\lambda)
\end{equation}
We notice then that for a given graph $G = sld(\tilde{H})$ 
\begin{equation}
\sum_{\tilde{H} \in GFU_B(w,G) \wedge vn(\tilde{H}) = j} \Pi_{\tilde{H}}(\lambda) =
(-\lambda_2)^{w-1} \cdot \lambda_{j+1} \cdot \frac{\partial}{ \partial \lambda_j} \left(\Pi_G(\lambda)\right) 
\end{equation}
which is trivially also true if $G$ does not have a vertex with indegree equal to $j$. 
Putting all this together we reach equation \eqref{eq_mbsum} by summing over $j \geq 2$.
Now putting everything together using equation \eqref{eq_syfg} and \eqref{fc_graph} we reach equations 
\eqref{eq_masum} and \eqref{eq_mbsum}.
\end{proof}
\begin{lemma}
For $\Sigma_r^{[U]}(\lambda)$ let us define the formal series 
\begin{equation}
\Sigma^{[U]}(\lambda) := \sum_{r \geq 0} \Sigma_r^{[U]}(\lambda)  
\end{equation}
Then
\begin{equation}\label{eq_sigmau}
m \cdot \Sigma^{[U]}(\lambda) = \frac{ 1}{1 + \left( \lambda_1^{d/2 -2} \cdot \lambda_2 \right)\cdot \left( \frac{2 \cdot (1 + m) \cdot \Gamma(2 - \frac{d}{2})}{(4\cdot\pi)^{\frac{d}{2}}}\right)} \cdot D \left( \Gamma^{[0],pt}(\lambda)
\right)
\end{equation}
is to be understood as a formal sum of monomials in $\lambda_j$ with $j > 1$ and 
\begin{equation}  
D := \sum_{j = 1}^{\infty} (j + 1)\cdot (j + m)\cdot  \lambda_{j + 1} \cdot \frac{\partial}{\partial \lambda_j}
\end{equation}
\end{lemma}
\begin{proof}
In lemma \ref{le_ladder} we had proven a decomposition of 
\begin{equation}
GFU(r,1,2) = GFUA(r,1,2) \dot{\cup} GFUB(r,1,2)
\end{equation}
In lemma \ref{le_pea}  by the one to one relationship we had shown that 
\begin{equation}
GFUA(r,1,2) = =\dot{\bigcup}_{1 \leq w \leq r} \left( \dot{\bigcup}_{G \in GF(r-w,0,2)} GFU_A(w,G) \right)
\end{equation}
In lemma \ref{le_peb} we had shown that 
\begin{equation}
GFUB(r,1,2) =\dot{\bigcup}_{  1 \leq w \leq r} \left(\dot{\bigcup}_{G \in GF(r-w+1,0,2)} GFU_B(w,G) \right)
\end{equation}
So summing equations \eqref{eq_masum} and \eqref{eq_mbsum} over $1 \leq w \leq r$ and $r \geq 2$ we get equation \eqref{eq_sigmau}.
\end{proof}

Specifically according to equation  \eqref{eq_gam0phi4}  for $\phi^4$-theory for $d=2$ we get 
\begin{multline}\label{eq_sigu} 
m \cdot \Sigma^{[U]}(\lambda) = \frac{2 \cdot \lambda_1}{\left(1 + \frac{ ( 1 + m) \cdot \lambda_2 }{2\pi \cdot \lambda_1} \right) } \cdot \\ 
\left(\sum_{r \geq 2} \Gamma^{[0],pt}_{\phi^4,r} (m) \cdot (1 + m)\cdot (r - 1) \cdot \left(-\frac{\lambda_2}{\lambda_1} \right)^{r+1} \right) 
\end{multline} 

Equation \eqref{eq_sigu} together with equation \eqref{eq_tgamma0} now is the crucial connection between $\lambda_1^2 \cdot G^{[U]} (\lambda) = \Sigma^{[U]}(\lambda)$ (see equation \eqref{eq_sigmae}) and the distribution of the multiple point range of the closed planar random walk. 

\section{Uniform asymptotic expansion}
To continue on our way to proove theorem \ref{th_main} we will now work on the integral transforms which appear in it. To make equations easier we use suitable variables.  
\begin{definition}
For $s \in \C\setminus\{0\}$ and real $t \in (0,\infty)$ let us define the integral kernel
\begin{equation}
h(t,s) := \exp\left( -ts - \frac{it}{2\pi}Log(ts) \right)
\end{equation} 
Then for real $\mu \in [0,\infty)$ and 
\begin{equation}
\Re(s) - \frac{\arg(s)}{2\pi} > d_0 > 0 
\end{equation}
  we can define the absolutely converging integral
\begin{equation}
P_{\mu}(s) := \int_0^{\infty} dt \cdot t^{\mu}\cdot h(t,s)
\end{equation}
\end{definition}
\begin{remark}\label{re_fus}
In the sequel we will work with asymptotic expansions around $s = \infty$ in the sector $\left|\arg(s)\right| < \frac{\pi}{2} - \xi$ for a fixed $\xi \in \left(0,\frac{\pi}{2}\right)$. We will therefore use the notation $s := \left|s\right| \cdot e^{i\phi}$ with $\phi \in (-\pi/2 + \xi,\pi/2 - \xi)$. We remark that if $\left|s\right|$ is sufficiently big then 
\begin{equation}
\cos\left(\frac{\pi}{2} - \xi\right) - \frac{1}{4\left|s\right|} \geq \cos\left(\frac{\pi}{2} - \frac{\xi}{2}\right)
\end{equation}
and therefore 
\begin{equation}\label{eq_fuun}
\left|h(t,s)\right| \leq \exp\left(-t\left|s\right|\cdot \cos(\frac{\pi}{2} - \frac{\xi}{2})\right) 
\end{equation}
In the sequel we will therefore assume \eqref{eq_fuun}.
\end{remark}
\begin{lemma}
For a given real $\mu $  the functions 
\begin{equation}
\phi_{p,q}^{[\mu]}(s) := \frac{1}{s^{p + \mu + 1}}  \cdot Log(s)^q 
\end{equation}
for integer $p \geq 0$ and  integer $q$ with $0 \leq p \leq q$
obviously form an asymptotic family of functions on $\C \setminus \{0\}$ around the point $s = \infty$. 
Then for any sector $\left| arg(s) \right| < \frac{\pi}{2} - \xi $ with a real $\xi \in (0,\frac{\pi}{2})$ we have a uniform asymptotic expansion around the point $s = \infty$ 
\begin{equation}
P_{\mu}(s) := \frac{1}{s^{\mu + 1}} \sum_{p = 0}^{\infty} \left( \frac{1}{2\pi\cdot i \cdot s}\right)^p \gamma_p\left(Log(s) \right) 
\end{equation}
with 
\begin{equation}
\gamma_p(y) := \sum_{r=0}^p \left(\sum_{m=0}^r \frac{(-1)^m y^{p-r+m}}{(p-r)!\cdot m! \cdot (r-m)!} \Gamma^{(r-m)}(\mu + r + 1)\right)
\end{equation}
where
\begin{equation}
\Gamma^{(q)} (z) := \left(\frac{\partial}{\partial z} \right)^q \Gamma(z) 
\end{equation}
denotes the $q^{th}$ derivative of the Eulerian $Gamma$ function. 
\end{lemma}
\begin{proof}
We first note that for $\Re(s) > 0$ the integral
\begin{equation}
S_{\mu,q}(s) := \int_0^{\infty} dt \cdot t^{\mu} e^{-ts} \cdot \sum_{p=0}^q \frac{1}{p!} \left( \frac{-it}{2\pi} \cdot Log(ts)\right)^p
\end{equation}
exists and can (with simple algebra) be calculated to be  
\begin{equation}
S_{\mu,q}(s) =  \frac{1}{s^{\mu + 1}} \sum_{p = 0}^{q} \left( \frac{1}{2\pi\cdot i \cdot s}\right)^p \gamma_p\left(Log(s) \right) 
\end{equation}
In the sequel we consider an upper bound for 
\begin{equation}
\left|P_{\mu}(s) - S_{\mu,q}(s)\right|
\end{equation}
realizing that the integrand of $S_{\mu,q}(s)$ is a Taylor expansion of the integrand of $P_{\mu}$. We divide the integral over $(0,\infty)$ into a \lq\lq near\rq\rq\  and a \lq\lq far\rq\rq\  part. 
 For $t \in \left(0,\left|s\right|^{(-\beta)}\right)$ the following inequality follows from looking at a possible maximum in $t$ and the value at the border of the interval
\begin{equation}
\left| t \cdot Log(ts) \right| \leq \left|s\right|^{-\beta} \cdot \left( \frac{\pi}{2} + (1 - \beta) \left|\ln(\left|s\right|)\right| \right) + \frac{1}{ \left|s\right|^2} 
\end{equation}
So there is a real constant $K(\beta)$ such that for $t \in \left(0,\left|s\right|^{(-\beta)}\right)$ and $\left|s\right| \geq K(\beta)$ 
\begin{equation}
\left| t \cdot Log(ts) \right| \leq \frac{1}{2}
\end{equation}
and we can use the fact, that the exponential function has a convergent Taylor expansion to write
\begin{equation}
\left| e^{-\frac{it}{2\pi}Log(ts)} - \sum_{p=0}^q \frac{1}{p!}\left( \frac{-it}{2\pi} \cdot Log(ts)\right)^p \right| \leq M(\beta,q) \left|t \cdot Log(ts) \right|^{q + 1}
\end{equation}
with some real constant $M(\beta,q)$ and $t \in \left(0,\left|s\right|^{(-\beta)}\right)$ 

We can therefore write
\begin{multline}
I_1 := \left|\int_0^{\left|s\right|^{(-\beta)}} dt \cdot t^{\mu} e^{-ts}\left( e^{-\frac{it}{2\pi}Log(ts)} - \sum_{p=0}^q \frac{1}{p!}\left( \frac{-it}{2\pi} \cdot Log(ts)\right)^p \right) \right| \\
\leq \frac{M(\beta,q)}{\left|s\right|^{\mu + q + 2}} \int_0^{\left|s\right|^{1 - \beta}}dy\cdot y^{\mu + q + 1}\cdot e^{-y\cdot cos(\phi)}\cdot\left(\left|\ln(y)\right| + \frac{\pi}{2}\right)^{q + 1}
\\
 \leq M(\beta,q) \frac{L(\mu, q,\xi)}{ \left| s\right|^{\mu+q+2}}
\end{multline}
with the real constant
\begin{equation}
L(\mu, q,\xi) := \int_0^{\infty}dy\cdot y^{\mu + q + 1}\cdot e^{-y\cdot cos(\frac{\pi}{2} - \xi)}\cdot\left(\left|\ln(y)\right| + \frac{\pi}{2}\right)^{q + 1}
\end{equation}
Let us in the sequel assume that $\left|s\right|$ is so big that for $y \geq \left|s\right|^{1 - \beta}$ the equation
\begin{equation}
\left|\ln(y)\right| + \frac{\pi}{2} < y 
\end{equation}
is true.
For integer $k \geq 0$ we consider the integral
\begin{multline}
I_{2,k} := \left| \int_{\left|s\right|^{(-\beta)}}^{\infty} dt \cdot t^{\mu} \cdot e^{-ts} \left(t \cdot Log(ts) \right)^k \right| \\
\leq \frac{1}{\left|s \right|^{\mu + k + 1}} \int_{\left|s\right|^{(1-\beta)}}^{\infty} dy \cdot y^{\mu + k} \cdot e^{-y\cos(\phi)} \cdot \left(|\ln(y)| + \frac{\pi}{2} \right)^k \\
\leq \frac{ \Gamma\left(\mu+2k + 1,\left|s\right|^{(1-\beta)}\right)}{cos(\pi/2 - \xi)^{\mu + 2k + 1} \cdot \left|s \right|^{\mu + k + 1}}  
\end{multline} 
We can now use \citep[p.336 eq.8.357]{rg} for the asymptotic expansion of the incomplete Gamma function and get
\begin{multline}\frac{ \Gamma\left(\mu+2k + 1,\left|s\right|^{(1-\beta)}\right)}{\left|s \right|^{\mu + k + 1}} \\ \sim \frac{\cdot \left|s \right|^{(1-\beta)\cdot (\mu + 2k)}}{\left|s \right|^{\mu + k + 1}} \exp\left(-\left|s \right|^{(1-\beta)} \right) \left( 1 + O\left(\frac{1}{\left|s \right|^{(1-\beta)}} \right)\right)
\end{multline}
and therefore $I_{2,k}$ is exponentially small in $\left|s \right|$.
We now use equation \eqref{eq_fuun} to consider the integral
\begin{multline}
I_3 := \left| \int_{\left|s\right|^{(-\beta)}}^{\infty} dt \cdot t^{\mu} \cdot h(t,s) \right| \\
\leq \frac{1}{\left| s \right|^{\mu+1}} \int_{\left|s\right|^{(1-\beta)}}^{\infty} dy \cdot  y^{\mu} \cdot \exp\left(-y\cdot \cos\left(\frac{\pi}{2} - \frac{\xi}{2}\right)\right) 
\\
\leq
\frac{\Gamma\left(\mu+1,\left| s\right|^{1-\beta}\right)}{\left|s \right|^{\mu + 1}\cdot \cos\left(\frac{\pi}{2} -\frac{\xi}{2} \right)^{\mu + 1} }
\end{multline}
 which therefore again is exponentially small in $\left|s \right|$.
Putting all this together we get 
\begin{equation}
\left| P_{\mu}(s) - S_{\mu,q}(s) \right| \leq I_1 + I_3 + \sum_{k=0}^q I_{2,k} \sim O\left(\frac{1}{s^{\mu + q + 2}} \right)
\end{equation}
uniformally for $\left| \arg(s) \right| < \frac{\pi}{2} - \xi$ which is equivalent to the Lemma.
\end{proof}
\begin{lemma}\label{le_momen}
Let the function 
\begin{eqnarray}
f: (0,\infty) & \mapsto &\C \\
t & \mapsto & f(t)
\end{eqnarray}
be such that there exists real constants $C, a, \tau$ with
\begin{equation}
\left| f(t) \right| < C\cdot e^{a\cdot t}
\end{equation}
for $t \in [\tau,\infty)$ 
and let there be a real constant $\alpha$ and complex constants $\beta_k$ for $k=0,\ldots,n$ such that 
\begin{equation}
f(t) = t^{\alpha} \cdot \left( \sum_{k=0}^n \beta_k \cdot t^k  + R_{n+1}(t) \right)
\end{equation}
and for a real constant $A$
\begin{equation}
\left| R_{n+1}(t) \right| < A \cdot t^{n+1}
\end{equation}
for $t \in (0,\tau)$

Then the integral
\begin{equation}
F(s) := \int_0^{\infty} dt \cdot f(t) \cdot h(t,s)
\end{equation}
obviously converges absolutely for 
\begin{equation}
a - Re(s) + \frac{1}{2\pi} arg(s) < 0
\end{equation}
On every sector $\left| \arg(s) \right| < \pi/2 - \xi$ 
\begin{equation}\label{eq_asyexp}
F(s) - \sum_{k=0}^{n} \beta_k \cdot P_{\alpha + k}(s) \sim O\left(\frac{1}{s^{\alpha + n+2}} \right)
\end{equation}
holds uniformly. 
\end{lemma}
\begin{proof}
We assume equation \eqref{eq_fuun} and $|s|$ to be so big that 
\begin{equation}
\left| s \right| \cdot \cos\left(\frac{\pi}{2} -\frac{\xi}{2} \right) - a > 0
\end{equation}
We write 
\begin{multline}
F(s) - \sum_{k=0}^{n} \beta_k \cdot P_{\alpha + k}(s) = 
\int_0^{\tau} t^{\alpha} R_{n + 1}(t) \cdot h(t,s)  \\ +
\int_{\tau}^{\infty} dt\cdot f(t) \cdot h(t,s) \\ - \sum_{k=0}^n \beta_k \cdot \int_{\tau}^{\infty} dt \cdot t^{k+\alpha} \cdot h(t,s) = J_1 + J_2 + \sum_{k=0}^n \beta_k \cdot J_{3,k} 
\end{multline}
We notice 
\begin{multline}
|J_1| = \left|\int_0^{\tau} t^{\alpha} R_{n + 1}(t) \cdot h(t,s)  \right| \leq A \cdot \int_0^{\infty} dt \cdot t^{n + 1 + \alpha} \cdot \left| h(t,s) \right| \\ \leq 
A \cdot \frac{\Gamma(n + \alpha + 2)}{\left|s \right|^{n+\alpha + 2} \cdot \left(\cos\left(\frac{\pi}{2} -  \frac{\xi}{2}\right)\right)^{n + \alpha + 2}} 
\end{multline}
and 
\begin{equation}
|J_2| = \left|\int_{\tau}^{\infty} dt\cdot f(t) \cdot h(t,s)  \right| \leq 
C \cdot \frac{\exp\left( - \left(\left|s \right| \cdot \cos\left(\frac{\pi}{2} -  \frac{\xi}{2}\right) - a\right) \cdot \tau \right)}{ \left(\left|s \right| \cdot \cos\left(\frac{\pi}{2} -  \frac{\xi}{2}\right) - a\right)}
\end{equation} 
which is exponentially small in $\left|s\right|$ 
and finally
\begin{equation}
|J_{3,k}| = \left| \int_{\tau}^{\infty} dt \cdot t^{k+\alpha} \cdot h(t,s) \right| \leq \frac{\Gamma\left(\alpha + k + 1,\tau \cdot \left|s \right| \cdot \cos\left(\frac{\pi}{2} -  \frac{\xi}{2}\right)\right) }{\left|s \right|^{k+\alpha + 1} \cdot \left(\cos\left(\frac{\pi}{2} -  \frac{\xi}{2}\right)\right)^{k + \alpha + 1}} 
\end{equation}

which again is exponentially small. 
Putting all this together we reach the lemma. 
\end{proof}
\begin{lemma}\label{le_gam}
Let $r > 1$ be an integer. We write 
\begin{equation}
\frac{t^{r -1}}{\Gamma\left(r - \frac{i}{2\pi}t\right)} = \sum_{p=0}^{\infty} ga_{p}(r)\cdot t^p
\end{equation} 
Then for integer $M > r$ let us define
\begin{equation}
Pa_M(r,s) := \sum_{p=0}^{M-1} ga_{p}(r)\cdot P_{p}(s)  
\end{equation}
Then for any sector $\left|arg(s)\right| < \frac{\pi}{2} - \xi$ we have the following uniform asymptotic expansion around $s = \infty$ 
\begin{equation}\label{eq_algebre}
Pa_M(r,s) = \frac{1}{s^r} \sum_{p=0}^{M-r} \left(\frac{1}{2\pi\cdot i\cdot s}\right)^p + O\left(\frac{1}{s^{M+1}} \right)
\end{equation}
\end{lemma}
\begin{proof}
Using \cite[p.330, eq.8.315.2]{rg} for $z=R - i/(2\pi)\cdot b$ we get
\begin{equation}\label{eq_inttr}
\frac{1}{2\pi\cdot i} \int_{a - i\cdot \infty}^{a + i\cdot \infty}  ds \cdot  \frac{e^{b \cdot \left(s + \frac{i}{2\pi}\cdot Log(s) \right)}}{s^R}  = \theta(b) \cdot \frac{b^{R - \frac{i}{2\pi}b -1}}{\Gamma\left(R - \frac{i}{2\pi} b \right)}
\end{equation}
for real $b$, $a \in (0,\infty)$, complex $s$ with $\left|\arg(s)\right| < \pi/2$ and complex $R$ with $\Re(R) > 0$. $\theta(b)$ here denotes the Heaviside step function at $b$. The path of integration is parallel to the imaginary axis.
We note that for any $s \in \C\setminus\{0\}$ the equation
\begin{equation}
w :=  s + \frac{i}{2\pi}\cdot Log(s)
\end{equation}
can be solved by
\begin{equation}
s = \frac{i}{2\pi}\cdot W_{L}\left( 2\pi \cdot e^{- 2\pi \cdot i \cdot \left(w + \frac{1}{4}\right)} \right)
\end{equation}
where $W_L$ is any branch of the complex Lambert's function. In the sequel we will follow the definition of those branches in \cite{Wf}.
Then from \cite[Figure 4]{Wf} we see that depending on the value of $a$ the path of integration is either completely in the range of one branch of $W_L$ or in the range of two branches which at the point of transition are connected by counterclockwise continuity (which means that by a different definition of the branches the path could have been fitted into the range of one such branch always).
Without restriction therefore we assume the first case and will write $W_{L,br(a)}$ for this branch. Let us now define
\begin{eqnarray}
\psi: \C & \mapsto &\C \\
w & \mapsto & \frac{i}{2\pi}\cdot W_{L,br(a)}\left( 2\pi \cdot e^{- 2\pi \cdot i \cdot \left(w + \frac{1}{4}\right)} \right)
\end{eqnarray}
We then have by substitution
\begin{equation}\label{eq_lapl1}
\frac{1}{2\pi\cdot i} \int_{\gamma_W}  dw \cdot  \frac{e^{b w}}{\psi(w)^R \left(1 - \frac{1}{2\pi\cdot i \cdot \psi(w)} \right)} = \theta(b) \cdot \frac{b^{R - \frac{i}{2\pi}b -1}}{\Gamma\left(R - \frac{i}{2\pi} b \right)}
\end{equation}
where $\gamma_W$ is the transformed path corresponding to the line paralell to the imaginary axis. It goes from $a + \frac{1}{4} - i \cdot \infty$ to $a - \frac{1}{4} + i \cdot \infty$. For $Re(R) > 1$ because of the vanishing behaviour (using \citep[eq. 4.18]{Wf}) of the integrand in the vicinity of $w = \pm i\cdot \infty$ the path of integration can be deformed into a line paralell to the imaginary axis again without a change of the value of the integral. But then the left hand side of equation \eqref{eq_lapl1} is a Bromwich integral and therefore can be inverted to a Laplace integral. So we get
\begin{equation}
\int_0^{\infty} db \cdot e^{-bw} \frac{b^{R - \frac{i}{2\pi}b -1}}{\Gamma\left(R - \frac{i}{2\pi} b \right)} = \frac{1}{\psi(w)^R \left(1 - \frac{1}{2\pi\cdot i \cdot \psi(w)} \right)}
\end{equation}
The inversion from Bromwich to Laplace is possible because the function 
\begin{eqnarray}
f_{fou}: \R & \mapsto &\C \\
b & \mapsto & e^{-b\cdot w_0} \cdot \theta(b)\cdot \frac{b^{R - \frac{i}{2\pi}b -1}}{\Gamma\left(R - \frac{i}{2\pi} b \right)}
\end{eqnarray}
for $w_0 > 1/2$ declines exponentially for $b \rightarrow \infty$ and therefore $f_{fou} \in L_2(\R) \cap L_1(\R) $ and therefore is invertible as a Fourierintegral.\\
We can now resubstitute into $s$ if $\Re(s) - \arg(s)/(2\pi) > 1/2$ and find
\begin{equation}\label{eq_intg}
\int_0^{\infty} db \cdot e^{-b\left(s + \frac{i}{2\pi}Log(s\cdot b) \right)} \frac{b^{R -1}}{\Gamma\left(R - \frac{i}{2\pi} b \right)} = \frac{1}{s^R \left(1 - \frac{1}{2\pi\cdot i \cdot s} \right)}
\end{equation}
But now we can apply lemma \ref{le_momen} to the integral on the left hand side of \eqref{eq_intg} and get a uniform asymptotic expansion around $s = \infty$ . On the right hand side of \eqref{eq_intg} the 
geometric series can be expanded to yield another uniform asymptotic expansion. Both have to be equal and therefore \eqref{eq_algebre} is true.
\end{proof}
\begin{corollary}
Equation \eqref{eq_intg} can be differentiated (multiply) in $s$ as both sides are holomorphic in $s$ for $\Re(s) - \arg(s)/(2\pi) > 1/2$ and corresponding uniform asymptotic expansions derived by the same principle from the corresponding equations.
For example
\begin{multline}
\left(- \frac{1}{1 - \frac{1}{2\pi i s}} \frac{\partial}{\partial s} \right)^p 
\int_0^{\infty} db \cdot e^{-b\left(s + \frac{i}{2\pi}Log(s\cdot b) \right)} \frac{b^{R -1}}{\Gamma\left(R - \frac{i}{2\pi} b \right)} = \\
\int_0^{\infty} db \cdot e^{-b\left(s + \frac{i}{2\pi}Log(s\cdot b) \right)} \frac{b^{R + p -1}}{\Gamma\left(R - \frac{i}{2\pi} b \right)} = 
\\
\left(- \frac{1}{1 - \frac{i}{2\pi i s}} \frac{\partial}{\partial s} \right)^p 
 \frac{1}{s^R \left(1 - \frac{1}{2\pi\cdot i \cdot s} \right)}
\end{multline}
or 
\begin{multline}
\left( - \frac{s}{1 - \frac{1}{2\pi i s}} \frac{\partial}{\partial s} \right) \cdot \left( 1 - \frac{1}{2\pi i s} \right) \frac{1}{s^R \left(1 - \frac{1}{2\pi\cdot i \cdot s} \right)} = \\
\frac{R}{s^R \left(1 - \frac{1}{2\pi\cdot i \cdot s} \right)} = \\
\int_0^{\infty} db \cdot e^{-b\left(s + \frac{i}{2\pi}Log(s\cdot b) \right)} \frac{R\cdot b^{R -1}}{\Gamma\left(R - \frac{i}{2\pi} b \right)} 
\end{multline}
\end{corollary}
\begin{corollary}
Using 
\begin{equation}
\frac{1}{\Gamma\left(1 - \frac{ib}{2\pi}\right)} = \frac{1 - \frac{ib}{2\pi} }{\Gamma\left(2 - \frac{ib}{2\pi}\right)} 
\end{equation}
and the previous corollary we see that lemma \eqref{le_gam} is also true for $r = 1$. 
\end{corollary}
\section{Proof of Theorem 1.1}
\begin{proof} 

Let $w_{cl}$ and $w_{nr}$ denote closed and non restricted planar random walks respectively. Let $N_{2k}(w_{cl})$ and $N_{2k}(w_{nr})$ be their $2k$ multiple point range (the number of points of multiplicity $2k$).
Let us denote the following two random distributions with Borel measure on $\R$:
\begin{equation}
\beta_0(n,k):=\frac{\ln(2n)^3}{4\pi^3 \cdot (2n)}\cdot \left( N_{2k}(w_{cl}) - E_n(N_{2k}(w_{cl}) \right)
\end{equation}
for closed random walks $w_{cl}$ of length  $2n$ and
\begin{equation}
\beta_2(n,k):=\frac{\ln(n)^3}{4\pi^3 \cdot n}\cdot \left( N_{2k}(w_{nr}) - E_n(N_{2k}(w_{nr}) \right)
\end{equation}
for non restricted random walks $w_{nr}$ of length $n$.
Then the Hankel matrix of $\beta_i(n,k)$ is positive semidefinite for $i=0,2$ and all $n,k \in \N$. Furthermore 
\begin{equation}
\mu_p(i,k) := \lim_{n \to \infty} E\left(\beta_i(n,k)^{p}\right)
\end{equation}
(E(.) the expectation value) exist and are independant of $k$, i.e $\mu_p(i,k) = \mu_p(i,1)$ as shown in \citep[Theorem 8.4, Theorem 9.4]{dhoefc} and \citep[Theorem 3.5]{hamana_ann}. The Hankel matrix of these $\mu_p(i,1)$ then is positive semidefinite and therefore there exist probability distributions $\nu_i$ with Borel measure on $\R$  such that their moments are $\mu_p(i,1)$ as the corresponding Hamburger problem is solvable \cite[ch. 4]{Hamburg}.
(The possibility that the corresponding measures are not unique does not play a role in the proof. For $\nu_2$ which is proportional to the intersection local time of a Brownian motion however the uniqueness is well known and proven.) We denote the characteristic functions of the random variables $\nu_i$ 
\begin{equation}
\Phi_i(t) := E(e^{it\nu_i})
\end{equation}
for $i=0,2$ respectively. By construction $\nu_0, \nu_2$ have finite moments of any order.  Because these moments are real and finite we know that $E(\left| \nu_i \right|^j)$ is finite for any order $j$ also \cite[ch. 9.3]{loeve}. Therefore there is an asymptotic (Taylor type) expansion around $t=0$. 
\begin{equation}
\Phi_i(t) = \left(\sum_{j=0}^{k} \frac{(it)^{j}}{j!} E\left((\nu_i)^j \right) \right) + O\left( t^{(k+1)} \right)
\end{equation}
But therefore $\Phi_0(t)\cdot e^{-\frac{i \cdot \gamma \cdot t}{2\pi}}$ and $\Phi_2(t) \cdot t \cdot e^{\frac{i \cdot (1 - \gamma) \cdot t}{2\pi}}$ fulfil the conditions for a function $f(t)$ of lemma \ref{le_momen} with $a=0$. Therefore setting 
\begin{equation}
s = \frac{i}{g}
\end{equation}
we realize that the integral transforms
\begin{equation}
L_0\left(k,\frac{i}{s}\right) := \int_0^{\infty} dt \cdot \Phi_0(t) \cdot e^{-\frac{i \cdot \gamma \cdot t}{2\pi}} \cdot \vartheta_k\left(\frac{i}{s},t\right)
\end{equation}
and 
\begin{equation}
L_2\left(k,\frac{i}{s}\right) := \int_0^{\infty} dt \cdot \Phi_2(t) \cdot t \cdot e^{\frac{i \cdot (1 - \gamma) \cdot t}{2\pi}} \cdot \vartheta_k\left(\frac{i}{s},t\right)
\end{equation}
for $k=0$ according to lemma \ref{le_momen} have the uniform asymptotic expansion given by equation \eqref{eq_asyexp} in $\frac{1}{s}$.  We will now analyze this expansion. It can of course be rewritten into a uniform asymptotic expansion in $g$ around $g = 0$ for any sector
\begin{equation}
\left| \arg(g) - \frac{\pi}{2} \right| < \frac{\pi}{2} - \xi
\end{equation}
Now from \eqref{eq_tgamma0} together with \cite[eq. 8.57]{dhoefc} we find the asymptotic expansion
\begin{equation}\label{eq_phi0g}
\Phi_0(t) = e^{\frac{i \cdot \gamma \cdot t}{2\pi}} \cdot \left( 
\frac{1}{\Gamma\left(1 - \frac{it}{2\pi}\right)} + 4\pi \cdot 
\left( \sum_{r =2}^{\infty} \Gamma^{[0],pt}_{r,1} \cdot (r-1)  \cdot \frac{(-it)^r }{\Gamma\left(r - \frac{it}{2\pi}\right)}
\right)
\right)
\end{equation}
From \eqref{eq_tdelta2} and \eqref{eq_tgamma0} together with \cite[eq. 1.1]{dhoefc} we find the asymptotic expansion
\begin{multline}\label{eq_phi2g}
\Phi_2(t) = e^{\frac{i \cdot (\gamma - 1) \cdot t}{2\pi}} \cdot \Biggl[ 
\frac{1}{\Gamma\left(2 - \frac{it}{2\pi}\right)} + 
\left( 
\sum_{r =2}^{\infty} 
G^{[C],ob}_{r,0} \cdot (r + 1) \cdot 
\frac{(-it)^r }{\Gamma\left(r + 2 - \frac{it}{2\pi}\right) }
\right)  \\
 + 2 \cdot \left(
\sum_{r=2}^{\infty} 
\Gamma^{[0],pt}_{r,1} \cdot (r-1) \cdot 
\frac{(-it)^{r + 1} }{\Gamma\left(r + 2 - \frac{it}{2\pi}\right) }
\right)
\Biggr]
\end{multline}
For real $0< \xi < \frac{\pi}{2}$ denote by 
\begin{equation}
Re_{\xi} := \{s \in \C: \left|arg(s)\right| < \frac{\pi}{2} - \xi \}
\end{equation}
and 
\begin{equation}
Rh_{k} := \{s \in \C: \Re(s) - \frac{\arg(s)}{2\pi} > -k\}
\end{equation}
Then $L_0\left(k,\frac{i}{s}\right)$ and $L_2\left(k,\frac{i}{s}\right)$ are analytic on the open sets $Rh_{k}$ and for $k=0$ have uniform asymptotic expansions on $Re_{\xi}$. Using \eqref{eq_asyexp} we find
\begin{equation}
L_0\left(0,\frac{i}{s}\right) = \frac{1}{s} - \frac{i}{2\pi\cdot s^2} + O\left(\frac{1}{s^3}\right)
\end{equation}
and 
\begin{equation}
L_2\left(0,\frac{i}{s}\right) = \left(\frac{1}{s}\right)^2 + O\left(\frac{1}{s^3}\right)
\end{equation}
Therefore let us define
\begin{equation}
\eta_0\left( \frac{1}{s} \right) := \frac{1}{4\pi\cdot s } \left( LS_{s} \left(\mu(s) \cdot LS_{s} \left( \mu(s) \cdot L_0\left(0,\frac{i}{s}\right) - \frac{1}{s} \right) \right) \right)
\end{equation}
where 
\begin{equation}
\mu(s) := 1 + \frac{i}{2\pi\cdot s}
\end{equation}
and $LS_{s}$ is the local antiderivative in $s$. Then $\eta_0$ is a well defined holomorphic function on $Re_{\xi}$ and has a uniform asymptotic expansion in $\frac{1}{s}$ on $Re_{\xi}$ and taking antiderivatives can be done order by order in the asymptotic expansion in the obvious way. $\eta_0$ can be made unique by choosing integration constants such that $\eta_0\left( \frac{1}{s} \right) \sim O\left(\left( \frac{1}{s}\right)^2\right)$.
Let us for an arbitrary natural number $M > 2$ now define 
\begin{equation}
\Gamma^{[0]}_M\left( \frac{i}{s} \right) := \left(\sum_{r=2}^M \Gamma^{[0],pt}_{r,1} \cdot \left( -\frac{i}{s} \right)^r \right) 
\end{equation}
Putting equation \eqref{eq_phi0g} into equation \eqref{eq_asyexp} we then get for the uniform asymptotic expansion 
\begin{equation}
\eta_0\left( \frac{1}{s} \right) = \Gamma^{[0]}_M\left( \frac{i}{s} \right) + O\left( \left( \frac{1}{s}\right)^{M+1} \right) 
\end{equation}
Starting in equation \eqref{eq_zeta0d} we see by simple algebra
\begin{equation}
\eta_0\left(\frac{1}{s}\right) = \zeta_{0}\left(0,\frac{i}{s}\right)
\end{equation}
and so the perturbation expansion $\Gamma^{[0]}_M$ is a uniform asymptotic expansion and \eqref{eq_gamma0} is true.
Let us also define 
\begin{equation}
\eta_2\left( \frac{1}{s} \right) := (-s)\cdot LS_{s} \left( \mu(s) \cdot L_2\left(0,\frac{i}{s}\right) \right)
\end{equation}
Then $\eta_2$ is a well defined holomorphic function on $Re_{\xi}$, has a uniform asymptotic expansion in $\frac{1}{s}$ on $Re_{\xi}$ and taking the antiderivative can be done order by order in the asymptotic expansion in the obvious way. $\eta_2$ can be made unique by choosing the integration constant such that 
\begin{equation}\label{eq_eta2a}
\eta_2\left( \frac{1}{s} \right) = 1 +  O\left(\left( \frac{1}{s}\right)\right)
\end{equation}
If we denote the series 
\begin{equation}
G^{[C]}_M\left(\frac{i}{s} \right) := \sum_{r=2}^{M} G^{[C],ob}_{r,0} \cdot\left(- \frac{i}{s} \right)^r
\end{equation}
we see that putting \eqref{eq_phi2g} into equation \eqref{eq_asyexp} leads to
\begin{equation}
\eta_2\left( \frac{1}{s} \right) = G^{[C]}_M\left(\frac{i}{s} \right)  + 1 + \frac{2i}{s\cdot \mu(s)} \frac{d}{ds} \left(s \cdot \Gamma^{[0]}_M\left( \frac{i}{s} \right) \right) +  O\left( \left( \frac{1}{s}\right)^{M+1} \right) 
\end{equation}
We have used that differentiation can be exchanged with expansion into the uniform asymptotic series $\Gamma^{[0]}_M$ because $\eta_0$ is holomorphic on the open and convex set $Re_\xi$.
We define 
\begin{equation}
G^{[2]}_M\left(\frac{i}{s} \right) := \sum_{r=2}^{M} G^{[2],pt,par}_{r,c,0} \cdot \left(- \frac{i}{s} \right)^r
\end{equation}
where $G^{[2],pt,par}_{r,c,0}$ is the $0^{th}$ coefficient of the polynomial $G^{[2],pt,par}_{r,c}(\lambda)$ in $m$ for $\lambda = (1,-1,0,\ldots)$. We also define
\begin{equation}
\Gamma^{[2]}_M\left(\frac{i}{s} \right) := \sum_{r=2}^{M} \Gamma^{[2],pt}_{r,0} \cdot \left(- \frac{i}{s} \right)^r
\end{equation}
 By equations \eqref{eq_sigu} and \eqref{eq_sigmae} and \eqref{eq_g2uc} we realize
\begin{equation}
\eta_2\left( \frac{1}{s} \right) = G^{[2]}_M\left(\frac{i}{s} \right) +  O\left( \left( \frac{1}{s}\right)^{M+1} \right) 
\end{equation}
and so the perturbation expansion $G^{[2]}_M$ is a uniform asymptotic expansion.
Starting in equation \eqref{eq_zeta2d} by simple algebra
\begin{equation}
\zeta_{2}\left(0,\frac{i}{s}\right) = \frac{1}{\eta_2\left(\frac{1}{s}\right)}
\end{equation} on $Re_{\xi}$ with exception of the points discrete in $Re_{\xi}$ where $\eta_2$ is $0$. Because of equation \eqref{eq_eta2a} the uniform asymptotic expansion of $\eta_2$ can be inverted too on $Re_{\xi}$. Because of \eqref{eq_gr2ga}  therefore the standard perturbation series $\Gamma^{[2]}_M$ is also a uniform asymptotic expansion and equation \eqref{eq_gamma2} is true.  
\end{proof}

\section{The rising edge behaviour}
\begin{theorem}
If Lipatov's \citep{lipatov} asymptotic formulas for a real planar $N$ component $\phi^4$ theory  are true (see e.g. \citep[eq. 79]{Suslov_1997}), then the Borel measures of the random variables $\nu_0$ and $\nu_2$ are unique (and therefore the functions $\zeta_0$ and $\zeta_2$) and have well defined distribution functions $f_0(x)$ and $f_2(x)$ with asymptotics for $x \rightarrow -\infty$ given by
\begin{equation}\label{eq_asdist0}
f_0(x) \sim e^{-\frac{\gamma}{2 \pi A}}\cdot \frac{8\pi\cdot \xi_0}{A} \cdot 
\left( -\frac{x}{A}\right)^{\mu_0 -1} e^{\frac{x}{A}} \cdot \left( 1 + O\left( \frac{1}{x} \right) \right)
\end{equation}
\begin{equation}\label{eq_asdist2}
f_2(x) \sim e^{\frac{1 -\gamma}{2 \pi A}}\cdot \frac{I_1^2\cdot \xi_2}{A} \cdot 
\left( -\frac{x}{A}\right)^{\mu_2 -1} e^{\frac{x}{A}} \cdot \left( 1 + O\left( \frac{1}{x} \right) \right)
\end{equation}
where $\gamma$ is Eulers constant 
\begin{equation}
\mu_0 = \frac{5}{2} - \frac{1}{2\pi A}
\end{equation}
\begin{equation}
\mu_2 = \frac{3}{2} - \frac{1}{2\pi A}
\end{equation} 
and 
\begin{equation}
\xi_0 = \xi_0(N=0,d=2)
\end{equation}
and 
\begin{equation}
\xi_2 = \xi_2(N=0,d=2)
\end{equation}
using
\begin{equation}
\xi_{M}(N,d) = \frac{2^{N-1}}{(2\pi)^{\frac{N+d+1}{2}}}\cdot \left( \frac{I_6 - I_4}{d} \right)^{\frac{d}{2}}\cdot \left(\frac{4}{I_4} \right)^{\frac{M+d}{2}} \cdot D_L^{-\frac{1}{2}} \cdot D_T^{-\frac{N-1}{2}}
\end{equation}
and $A = \frac{4}{I_4}$ where the constants $I_1,I_4,I_6, D_T, D_L$ (which  depend on $d$) are all related to the well known Gagliardo-Nirenberg type inequality for functions 
$g: \R^d \rightarrow \R$
\begin{equation}
\| g \|_4 \leq C \sqrt{\|\nabla g\|_2} \cdot \sqrt{\|g\|_2}
\end{equation}
Numerical results for these constants have been given for a couple of values of $d$ including the ones we need here ($d=2$) in \citep[Table 20.3, p.389]{phi4}.
\end{theorem}
\begin{proof}
We first note that an $N$-component real $\phi^4$ theory with $N$ components and a Lagrangian 
\begin{equation}\label{eq_rlagra}
\mathcal{L}_{\R}(x) = \frac{1}{2}
\partial_{\mu}\Phi(x)^{t}\partial^{\mu}\Phi(x) + \frac{1}{2}\lambda_1 \cdot \Phi(x){t}\Phi(x) + \frac{1}{4}\lambda_2 \cdot \left(\Phi^{t}(x)\Phi(x)\right)^{2} 
\end{equation}
(as used in \cite[eq.1]{Suslov_1997}) can for an even $N=2m$ number of real components be transformed into the Lagrangian \eqref{eq_lagra1} with $m$ complex components by defining
\begin{equation}
\Phi_{\C,j} := \frac{1}{\sqrt{2}}\left(\Phi_{\R,2j - 1} + i \cdot \Phi_{\R,2j}\right) 
\end{equation}
Therefore, taking into account the minus sign before the the coupling constant in \eqref{eq_g0} and \eqref{eq_g1} and using \cite[eq. 79]{Suslov_1997}
\begin{equation}\label{eq_ga0ras}
\Gamma_{r}^{[0],pt}(1,1) \sim \xi_0(2m,2) \cdot \Gamma\left(r + m + \frac{1}{2} \right)   \cdot \left( \frac{4}{I_4} \right)^r \cdot  \frac{2m\cdot \pi^{\frac{2m}{2}}}{\Gamma\left( 1 + \frac{2m}{2}\right)} \left(1 + O\left( \frac{1}{r} \right)\right)
\end{equation}
and 
\begin{multline}\label{eq_ga2ras}
\Gamma_{r}^{[2],pt} (1,1) \sim \\ 
\xi_2(2m,2)\cdot I_1^2 \cdot \Gamma\left(r + m + \frac{3}{2} \right)   \cdot \left( \frac{4}{I_4} \right)^r \cdot  \frac{2m\cdot \pi^{\frac{2m}{2}}}{\Gamma\left( 1 + \frac{2m}{2}\right)} \frac{1}{2m}\left(1 + O\left( \frac{1}{r} \right)\right)
\end{multline}
Using \eqref{eq_ga0ras} and \eqref{eq_ga2ras} in equations \eqref{eq_phi0g} and \eqref{eq_phi2g} we see that the right hand side of these eqations is an absolutely converging series for complex $\left| t \right| < \frac{I_4}{4}$  around the point $t=0$. It therefore represents a holomorphic function in $t$ and therefore also has a power series in $t$ with radius of convergence equal to $\frac{I_4}{4}$. Therefore the corresponding distributions $\nu_0$ and $\nu_2$ are unique and the distribution functions $f_0(x)$ and $f_2(x)$ exist and are Fourier transforms of the functions $\Phi_i$. Summing the leading contribution we find 
\begin{equation}\label{eq_phi0as}
\Phi^{as}_0(t) = 8\pi\cdot \xi_0\frac{\Gamma\left(\frac{3}{2} \right)}{ \Gamma\left(-\frac{it}{2\pi} \right)} e^{\frac{\gamma i t}{2\pi}} \cdot _2F_1\left(1,\frac{3}{2};-\frac{it}{2\pi}; -\frac{4it}{I_4}\right)
\end{equation}  
and
\begin{equation}\label{eq_phi2as}
\Phi^{as}_2(t) =  \xi_2\cdot I_1^2\cdot
 \frac{\Gamma\left(\frac{5}{2} \right)}{ \Gamma\left(2-\frac{it}{2\pi} \right)} e^{\frac{(\gamma - 1) \cdot i t}{2\pi}} \cdot _2F_1\left(1,\frac{5}{2}; 2-\frac{it}{2\pi}; -\frac{4it}{I_4}\right)
\end{equation}   
and equations \eqref{eq_asdist0} and \eqref{eq_asdist2} now follow from standard features of the hypergeometric function $_2F_1$ especially \cite[eq. 9.122, 9.131]{rg} and known features of the Gamma distribution. We have used that numerically 
\begin{equation}
\frac{1}{2\pi A} \approx 0.933112776025 < \frac{3}{2}
\end{equation}
and therefore $\mu_0 >0 $ and $\mu_2 > 0$ which is a condition for the use of \cite[eq. 9.122]{rg}.
This behaviour is a refinement of the one proven for $\nu_2$ in \cite{bass_chen}, noting that the intersection local time is $-\nu_2$ up to  a  known constant of proportionality. 
\end{proof}
For $\nu_2$ the falling edge behaviour has been proven to be double exponential in \cite{bass_chen} too. Equations \eqref{eq_phi0as} and \eqref{eq_phi2as} on the other hand do not give us sufficient information about the falling edge behaviour of $\nu_0$ and $\nu_2$ as they are not good approximations outside the circle of convergence. 
\section{Conclusion}
In this paper we have defined counter clockwise continous branches of holomorphic/meromorphic functions $\zeta_0$ and $\zeta_2$ on a Riemann surface with infinitely many sheets as simple antiderivatives of
integral transforms of the characteristic function of certain mathematically well defined probability distributions $\nu_0$ and $\nu_2$ related to multiple points of random walks. One branch of each of these functions respectively has a uniform asymptotic expansion in the origin. They are identical to the standard perturbation expansion of the proper functions $\Gamma^{[0],pt}$ (the free energy) and $\Gamma^{[2],pt}$ (the proper two point function) of the planar $\phi^4$ theory with $m$ complex components for $m=0$ (a precise definition of this case is provided).  In this way these proper functions, which are the building blocks of quantum field theories have been well defined mathematically beyond perturbation theory and shown to have their proper mathematical understanding in the realm of random walks. The somewhat ad hoc dismissal of Feynman graphs which contain tadpoles comes out naturally in this concept. The relationship to random walks also provides a natural summation procedure for the perturbation series which turns out to be a modified Borel summation. The kernel of the integral transform $\vartheta_{0}(g,b)$ (see equation \eqref{eq_dtheta}) of the Borel summation is closely related to the multiple return features of the random walk to the same point, i.e. to graphs with tadpoles/dams \citep[section 7.3]{dhoefc} (compare \citep[equation 9.72]{dhoefc} and equation \eqref{eq_inttr}) on this random walk level.\\ 
As quantum field theory thus can be related to characteristic functions we are also connected to the wealth of knowledge about probability distributions.  Standard perturbation theory in our case defines the moments of the corresponding probability distribution. If this distribution is defined uniquely by its moments then perturbation theory also uniquely defines the proper functions beyond perturbation theory.
So this fundamental question of quantum field theory is reduced to the uniqueness part of a Hamburger problem (see \cite{Hamburg}) which in the case of $\nu_2$ and therefore $\zeta_2$ has long been proven mathematically. If Lipatov asymptotics is correct then this question has a positive answer also for $\nu_0$ and therefore $\zeta_0$. Using Lipatov asymptotics on the other hand the rising edge behaviour of the distribution for large lenght of multiple points of planar random walks has been calculated too to be of the type of a Gamma distribution. The falling edge behaviour should be calculateable too, a future paper will be dedicated to it. It would make an almost quantitative evaluation of the distribution functions possible as the first few moments are also known. This would in turn probably make an almost quantitative evaluation of the proper functions possible by the integral transforms.\\
This paper gives the hope of generally defining proper functions of quantum field theories mathematically sound in terms of probability distributions of random geometrical objects. In this article we already have given the framework 
to extend this work to general $m \neq 0$ and $d=2$ and to some extend to general \lq\lq external momenta\rq\rq\  $P \neq 0$ and also higher order proper functions $\Gamma^{[2p]}$ with $p > 1$. What needs to be worked out is the distribution of the joint multiple points of multiple random walks, so called loop soups along the lines of \citep{dhoefc}, but doing this seems straightforward. Specifically theorem \ref{th_weifp} already contains a reformulation of the weight factors of standard perturbation theory of the corresponding proper functions in terms of multiple paths on graphs which is crucial for the connection to random walks. For higher dimensions $d \geq 3$ the procedure 
has to be amended by further ideas related to what physicists call renormalization. But it seems that quantum field theory can be tamed mathematically by methods similar to the ones in this paper. 
\bibliographystyle{imsart-number}
\bibliography{hoef_phi4}

\begin{thebibliography}{22}

\bibitem{bass_chen}
\begin{barticle}[author]
\bauthor{\bsnm{Bass},~\bfnm{Richard~F.}\binits{R.~F.}} \AND
  \bauthor{\bsnm{Chen},~\bfnm{Xia}\binits{X.}}
(\byear{2004}).
\btitle{Self-intersection local time: critical exponent, large deviations, and
  laws of the iterated logarithm}.
\bjournal{Ann. Probab.}
\bvolume{32}
\bpages{3221--3247}.
\bdoi{10.1214/009117904000000504}
\bmrnumber{2094444}
\end{barticle}
\endbibitem

\bibitem{bender_wu}
\begin{barticle}[author]
\bauthor{\bsnm{Bender},~\bfnm{Carl~M.}\binits{C.~M.}} \AND
  \bauthor{\bsnm{Wu},~\bfnm{Tai~Tsun}\binits{T.~T.}}
(\byear{1969}).
\btitle{Anharmonic oscillator}.
\bjournal{Phys. Rev. (2)}
\bvolume{184}
\bpages{1231--1260}.
\bmrnumber{0260323}
\end{barticle}
\endbibitem

\bibitem{Wf}
\begin{barticle}[author]
\bauthor{\bsnm{Corless},~\bfnm{R.~M.}\binits{R.~M.}},
  \bauthor{\bsnm{Gonnet},~\bfnm{G.~H.}\binits{G.~H.}},
  \bauthor{\bsnm{Hare},~\bfnm{D.~E.~G.}\binits{D.~E.~G.}},
  \bauthor{\bsnm{Jeffrey},~\bfnm{D.~J.}\binits{D.~J.}} \AND
  \bauthor{\bsnm{Knuth},~\bfnm{D.~E.}\binits{D.~E.}}
(\byear{1996}).
\btitle{On the {L}ambert {$W$} function}.
\bjournal{Adv. Comput. Math.}
\bvolume{5}
\bpages{329--359}.
\bdoi{10.1007/BF02124750}
\bmrnumber{1414285}
\end{barticle}
\endbibitem

\bibitem{danraf}
\begin{bbook}[author]
\bauthor{\bsnm{Danos},~\bfnm{M.}\binits{M.}} \AND
  \bauthor{\bsnm{Rafelski},~\bfnm{J.}\binits{J.}}
(\byear{1984}).
\btitle{Pocketbook of Mathematical Functions}.
\bpublisher{Harry Deutsch}, \baddress{Thun and Frankfurt am Main}.
\end{bbook}
\endbibitem

\bibitem{Tek}
\begin{barticle}[author]
\bauthor{\bsnm{Dvoretzky},~\bfnm{A.}\binits{A.}},
  \bauthor{\bsnm{Erd{\"o}s},~\bfnm{P.}\binits{P.}} \AND
  \bauthor{\bsnm{Kakutani},~\bfnm{S.}\binits{S.}}
(\byear{1950}).
\btitle{Double points of paths of {B}rownian motion in {$n$}-space}.
\bjournal{Acta Sci. Math. Szeged}
\bvolume{12}
\bpages{75--81}.
\bmrnumber{0034972 (11,671e)}
\end{barticle}
\endbibitem

\bibitem{Fla}
\begin{barticle}[author]
\bauthor{\bsnm{Flatto},~\bfnm{Leopold}\binits{L.}}
(\byear{1976}).
\btitle{The multiple range of two-dimensional recurrent walk}.
\bjournal{Ann. Probability}
\bvolume{4}
\bpages{229--248}.
\bmrnumber{0431388 (55 \#\#4388)}
\end{barticle}
\endbibitem

\bibitem{rg}
\begin{bbook}[author]
\bauthor{\bsnm{Gradstein},~\bfnm{I.~S.}\binits{I.~S.}} \AND
  \bauthor{\bsnm{Ryshik},~\bfnm{I.~M.}\binits{I.~M.}}
(\byear{1981}).
\btitle{Summen- Produkt- und Integraltafeln}.
\bpublisher{Harry Deutsch}, \baddress{Thun and Frankfurt am Main}.
\end{bbook}
\endbibitem

\bibitem{hamana_ann}
\begin{barticle}[author]
\bauthor{\bsnm{Hamana},~\bfnm{Yuji}\binits{Y.}}
(\byear{1997}).
\btitle{The fluctuation result for the multiple point range of two-dimensional
  recurrent random walks}.
\bjournal{Ann. Probab.}
\bvolume{25}
\bpages{598--639}.
\bdoi{10.1214/aop/1024404413}
\bmrnumber{1434120 (98f:60136)}
\end{barticle}
\endbibitem

\bibitem{dhoefc}
\begin{barticle}[author]
\bauthor{\bsnm{{Hoef}},~\bfnm{D.}\binits{D.}}
(\byear{2013}).
\btitle{{The Characteristic Function of the Renormalized Intersection Local
  Time of the Planar Brownian Motion}}.
\bjournal{ArXiv e-prints}.
\end{barticle}
\endbibitem

\bibitem{dhoefm}
\begin{barticle}[author]
\bauthor{\bsnm{{H{\"o}f}},~\bfnm{D.}\binits{D.}}
(\byear{2014}).
\btitle{{The Third and Fourth Moment of the Renormalized Intersection Local
  Time}}.
\bjournal{ArXiv e-prints}.
\end{barticle}
\endbibitem

\bibitem{itzykson}
\begin{bbook}[author]
\bauthor{\bsnm{Itzykson},~\bfnm{C.}\binits{C.}} \AND
  \bauthor{\bsnm{Zuber},~\bfnm{J.~B.}\binits{J.~B.}}
(\byear{1988}).
\btitle{Quantum Field Theory}.
\bpublisher{Mc Graw Hill}, \baddress{New York}.
\end{bbook}
\endbibitem

\bibitem{jain}
\begin{binproceedings}[author]
\bauthor{\bsnm{Jain},~\bfnm{N.~C.}\binits{N.~C.}} \AND
  \bauthor{\bsnm{Pruitt},~\bfnm{W.~E.}\binits{W.~E.}}
(\byear{1973}).
\btitle{The range of random walks}.
In \bbooktitle{Proc. Sixth Berkeley Symp. Math. Stat. Probab.}
\bpages{31--50}.
\end{binproceedings}
\endbibitem

\bibitem{phi4}
\begin{bbook}[author]
\bauthor{\bsnm{Kleinert},~\bfnm{Hagen}\binits{H.}} \AND
  \bauthor{\bsnm{Schulte-Frohlinde},~\bfnm{Verena}\binits{V.}}
(\byear{2001}).
\btitle{Critical properties of {$\phi^4$}-theories}.
\bpublisher{World Scientific Publishing Co., Inc., River Edge, NJ}.
\bdoi{10.1142/9789812799944}
\bmrnumber{1868655}
\end{bbook}
\endbibitem

\bibitem{legall}
\begin{barticle}[author]
\bauthor{\bsnm{Le~Gall},~\bfnm{J.~F.}\binits{J.~F.}}
(\byear{1986}).
\btitle{Propri\'et\'es d'intersection des marches al\'eatoires. {I}.
  {C}onvergence vers le temps local d'intersection}.
\bjournal{Comm. Math. Phys.}
\bvolume{104}
\bpages{471--507}.
\bmrnumber{840748 (88d:60182)}
\end{barticle}
\endbibitem

\bibitem{lipatov}
\begin{barticle}[author]
\bauthor{\bsnm{Lipatov},~\bfnm{L.~N.}\binits{L.~N.}}
(\byear{1977}).
\btitle{Divergence of the Perturbation Theory Series and the Quasiclassical
  Theory}.
\bjournal{Sov.Phys.JETP}
\bvolume{45}
\bpages{216--223}.
\end{barticle}
\endbibitem

\bibitem{loeve}
\begin{bbook}[author]
\bauthor{\bsnm{Lo\`eve},~\bfnm{Michel}\binits{M.}}
(\byear{1977}).
\btitle{Probability theory. {I}},
\bedition{Fourth} ed.
\bpublisher{Springer-Verlag, New York-Heidelberg}
\bnote{Graduate Texts in Mathematics, Vol. 45}.
\bmrnumber{0651017}
\end{bbook}
\endbibitem

\bibitem{Hamburg}
\begin{bmastersthesis}[author]
\bauthor{\bsnm{{Nielsen}},~\bfnm{D.~M.}\binits{D.~M.}}
(\byear{2010}).
\btitle{{The Hamburger Moment Problem}}
\btype{Master's thesis},
\bpublisher{University of Copenhagen}.
\end{bmastersthesis}
\endbibitem

\bibitem{Suslov_1997}
\begin{barticle}[author]
\bauthor{\bsnm{{Suslov}},~\bfnm{I.~M.}\binits{I.~M.}}
(\byear{1997}).
\btitle{{Density of states near an Anderson transition in a space of
  dimensionality $d = 4 - \epsilon$}}.
\bjournal{Soviet Journal of Experimental and Theoretical Physics}
\bvolume{84}
\bpages{1036-1046}.
\bdoi{10.1134/1.558221}
\end{barticle}
\endbibitem

\bibitem{tutte}
\begin{bbook}[author]
\bauthor{\bsnm{Tutte},~\bfnm{W.}\binits{W.}}
(\byear{1984}).
\btitle{Graph Theory}.
\bpublisher{Addison Wesley}, \baddress{New York}.
\end{bbook}
\endbibitem

\bibitem{aar}
\begin{barticle}[author]
\bauthor{\bparticle{van} \bsnm{Aardenne-Ehrenfest},~\bfnm{T.}\binits{T.}} \AND
  \bauthor{\bparticle{de} \bsnm{Bruijin},~\bfnm{N.~G.}\binits{N.~G.}}
(\byear{1951}).
\btitle{Circuits and trees in oriented linear graphs}.
\bjournal{Simon Stevin Wis. Natuurkd. Tijdschr.}
\bvolume{28}
\bpages{203--217}.
\end{barticle}
\endbibitem

\bibitem{psp}
\begin{barticle}[author]
\bauthor{\bsnm{Zhang},~\bfnm{Hong-Hao}\binits{H.-H.}},
  \bauthor{\bsnm{Feng},~\bfnm{Kai-Xi}\binits{K.-X.}},
  \bauthor{\bsnm{Qiu},~\bfnm{Si-Wei}\binits{S.-W.}},
  \bauthor{\bsnm{Zhao},~\bfnm{An}\binits{A.}} \AND
  \bauthor{\bsnm{Li},~\bfnm{Xue-Song}\binits{X.-S.}}
(\byear{2010}).
\btitle{A Note on analytic formulas of Feynman propagators in position space}.
\bjournal{Chinese Physics C}
\bvolume{34}
\bpages{1576-1582}.
\end{barticle}
\endbibitem

\bibitem{zinn_justin}
\begin{bbook}[author]
\bauthor{\bsnm{Zinn-Justin},~\bfnm{J.}\binits{J.}}
(\byear{1989}).
\btitle{Quantum field theory and critical phenomena}.
\bseries{International Series of Monographs on Physics}
\bvolume{77}.
\bpublisher{The Clarendon Press, Oxford University Press, New York}
\bnote{Oxford Science Publications}.
\bmrnumber{1079938}
\end{bbook}
\endbibitem

\end{thebibliography}
\end{document}